\documentclass[11pt]{amsart}

\usepackage{enumerate, amssymb,  mathrsfs, nicefrac, upref, graphicx, pdfsync}
\newtheorem{theorem}{Theorem}[section]
\newtheorem{lemma}[theorem]{Lemma}
\newtheorem*{lemma*}{Lemma}

\newtheorem{proposition}[theorem]{Proposition}
\newtheorem{corollary}[theorem]{Corollary}

\theoremstyle{definition}
\newtheorem{definition}[theorem]{Definition}
\newtheorem{example}[theorem]{Example}

\newtheorem{question}[theorem]{Question}

\theoremstyle{remark}
\newtheorem{remark}[theorem]{Remark}

\numberwithin{equation}{section}

\newcommand{\be}[1]{\begin{equation}\label{#1}}
\newcommand{\ee}{\end{equation}}

\newcommand{\abs}[1]{\lvert#1\rvert}

\newcommand{\norm}[1]{\lVert#1\rVert}

\newcommand{\A}{\mathbb{A}}
\newcommand{\C}{\mathbb{C}}

\newcommand{\W}{\mathscr{W}}

\newcommand{\R}{\mathbb{R}}
\newcommand{\RR}{\mathcal{R}}
\newcommand{\X}{\mathbb{X}}
\newcommand{\U}{\mathbb{U}}

\newcommand{\Y}{\mathbb{Y}}

\newcommand{\Ho}{\mathscr{H}}

\newcommand{\dtext}{\textnormal d}

\newcommand{\onto}{\xrightarrow[]{{}_{\!\!\textnormal{onto\,\,}\!\!}}}

\newcommand{\deff}{\stackrel {\textnormal{\tiny{def}}}{=\!\!=} }
\DeclareMathOperator{\diam}{diam}

\DeclareMathOperator{\dist}{dist}

\DeclareMathOperator{\re}{Re}
\DeclareMathOperator{\im}{Im}

\DeclareMathOperator{\loc}{loc}

\DeclareMathOperator{\Mod}{Mod}

\def\XXint#1#2#3{{\setbox0=\hbox{$#1{#2#3}{\int}$}\vcenter{\hbox{$#2#3$}}\kern-.5\wd0}}

\def\XXiint#1#2#3{{\setbox0=\hbox{$#1{#2#3}{\iint}$}\vcenter{\hbox{$#2#3$}}\kern-.5\wd0}}

\def\le{\leqslant}
\def\ge{\geqslant}

\begin{document}

\title[Dirichlet Energy and Hopf differentials]{Mappings of Least Dirichlet Energy \\and their Hopf Differentials}

\author[T. Iwaniec]{Tadeusz Iwaniec}
\address{Department of Mathematics, Syracuse University, Syracuse,
NY 13244, USA and Department of Mathematics and Statistics,
University of Helsinki, Finland}
\email{tiwaniec@syr.edu}

\author[J. Onninen]{Jani Onninen}
\address{Department of Mathematics, Syracuse University, Syracuse,
NY 13244, USA}
\email{jkonnine@syr.edu}
\thanks{Iwaniec was supported by the NSF grant DMS-0800416 and the Academy of Finland project 1128331.  Onninen was supported by the NSF grant DMS-1001620.}



\begin{abstract}
The paper is concerned with mappings $\,h \colon \mathbb X \onto\mathbb Y\,$ between planar domains having least Dirichlet energy. The existence and uniqueness (up to a conformal change of variables in $\,\mathbb X\,$)  of the energy-minimal mappings is established within the class $\,\overline{\mathscr H}_2(\mathbb X, \mathbb Y)\,$ of strong limits of homeomorphisms in the Sobolev space  $\,\mathscr W^{1,2}(\mathbb X , \mathbb Y)\,$, a result of considerable interest in the mathematical models of Nonlinear Elasticity. The inner variation (of the independent variable in $\,\mathbb X $) leads to the Hopf differential $ \,h_z \overline{h_{\bar{z}}} \,\textnormal d z \otimes \textnormal{d} z \,$ and its trajectories. For a pair of doubly connected domains, in which $\,\mathbb X\,$ has finite conformal modulus,  we establish the following principle:  \begin{center}\textit{A mapping $h \in \overline{\Ho}_2 (\X, \Y)$ is  energy-minimal if and only if its Hopf-differential is analytic in $\,\mathbb X\,$ and real along $\partial \X$.}\end{center}
   In general, the energy-minimal mappings may not be injective, in which case one observes the occurrence of cracks in $\mathbb X\,$.  Nevertheless, cracks are triggered only by the points in $\,\partial \mathbb Y\,$ where $\,\mathbb Y\,$  fails to be convex. The general law of formation of cracks reads as follows:

 \begin{center}\textit{Cracks propagate along vertical trajectories of the Hopf differential\\ from $\partial \mathbb X\,$ toward the interior of $\,\mathbb X\,$ where they eventually terminate before making a crosscut.}
 \end{center}
\end{abstract}

\maketitle

 \tableofcontents

\section{Introduction}

Recently there have been new challenges and substantial work done \cite{AIM, AIMb, AIMO,  IKKO, IOtr, IOa} on minimizing the Dirichlet energy integral

\begin{equation}\label{direnergy}
\mathscr E _{_{\mathbb X}} [h]= \iint_{\X} \abs{Dh}^2 =2 \iint_{\X} \left( \abs{h_z}^2+ \abs{h_{\bar z}}^2 \right)\, \dtext z\;,\;\;\;\dtext z = \dtext x_1\, \dtext x_2
\end{equation}
subject to mappings $\,h \colon \mathbb X \onto\mathbb Y\,$ between bounded planar domains $\,\mathbb X\,$   and $\,\mathbb Y\,$. The novelty of new problems is that the mappings in question are allowed to slip along certain parts of $\,\partial \mathbb X\,$; so-called \textit{traction free} boundary problems. Motivated by mathematical models of nonlinear elasticity \cite{Ba2, Ba,  Ba1}  the focus has been on finding a class of mappings as close  to homeomorphisms as possible, in which the minimum energy is attained. The case when $\,\mathbb X\,$ and $\,\mathbb Y\,$ are simply connected is clear, by virtue of Riemann Mapping Theorem. Thus throughout this paper  the domains $\,\mathbb X\,$   and $\,\mathbb Y\,$ will be predominantly \textit{multiply connected} and of the same topological type, meaning that their complements $\mathbb C \setminus \mathbb X\,$ and  $\mathbb C \setminus \mathbb Y\,$ consists of the same finite number of components, say $\,\ell \geqslant 2\,$.
From these perspectives, the most desirable spaces in which to look for the energy-minimal mappings are the class $\Ho_2 (\X, \Y) \subset \W^{1,2} (\X, \Y)\,$ of homeomorphisms and its strong closure $\,\overline{\Ho_2} (\X, \Y)\,$ in the Sobolev space $\,\W^{1,2} (\X, \Y)$, see Definition \ref{spaces} for precise conditions. Note that the class $\Ho_2 (\X, \Y) \subset \W^{1,2} (\X, \Y)$ can be empty if $\,\partial \mathbb X\,$ contains isolated points as  boundary components,  but $\,\partial \mathbb Y\,$ does not. Isolated points are removable singularities for homeomorphisms in $\mathscr W^{1,2} (\X, \Y)$. Such $\,\mathbb X\,$ can be viewed as a domain with punctures, which we call \textit{degenerate} multiply connected domain. Apart from such a situation we have:

\begin{theorem}[Existence]\label{Emin} Consider a pair $\,(\mathbb X\, , \mathbb Y)\,$ of nondegenerate multiply connected domains,  in which $\,\mathbb Y\,$ is a Lipschitz domain. Then there exists $\,h \colon \mathbb X \rightarrow \overline{\mathbb Y} \,$ in the class  $\,\overline{\Ho}_2 (\X, \Y)\subset \W^{1,2} (\X, \Y)\,$ such that
\begin{equation}
\mathscr E_\X [h] = \inf \{\,\mathscr E_\X  [f]\,\colon \;f \in \Ho_2 (\X , \Y) \,\}
\end{equation}
\end{theorem}

Specifically, we say that $\,\mathbb Y\,$ is a Lipschitz domain if every local arc in $\,\partial \mathbb Y\,$ upon a suitable rotation becomes a graph of a Lipschitz function.
\begin{proof} We appeal to Theorem 1.1  and Remark 2.7 in \cite{IOa} which tell us that the class $\,\mathscr H^{1,2}_{\textnormal{lim}}(\mathbb X, \mathbb Y)\,$ of weak limits of homeomorphisms $\,f \colon \mathbb X \onto \mathbb Y\,$  in $\,\mathscr W^{1,2}(\mathbb X , \mathbb Y)\,$ (see Definition \ref{spaces} )  and $\,\overline{\Ho}_2 (\X, \Y) \subset \W^{1,2} (\X, \Y)\,$ are equal. Armed with this result, the direct method in the Calculus of Variations concludes the proof.
\end{proof}

Theorem \ref{Emin} calls  for the following concept:
\begin{definition}
Let $\,\mathbb X\,$ and $\,\mathbb Y\,$ be bounded domains. A map $\,h \in \overline{\Ho}_2 (\X, \Y)\,$ such that
\begin{equation}\label{EM}
\mathscr E_\X [h] = \inf \{\,\mathscr E_\X  [f]\,\colon \;f \in \Ho_2 (\X , \Y) \,\}
\end{equation}
will hereafter be referred to as an \textit{energy-minimal map}.
\end{definition}
However the proof of uniqueness of the energy-minimal map is tricky and far more involved. 

\begin{theorem}[Uniqueness]\label{thmuni}
Let $\,\X\,$ be a nondegenerate doubly connected domain  and $\Y$ a bounded  doubly connected Lipschitz domain. Then the energy-minimal map $h \in  \overline{\Ho}_2 (\X, \Y)$  is unique up to a conformal change of variables in  $\X$.
\end{theorem}

If one wants to find an easy and clear way of verifying whether a given map $\, h \colon \mathbb X \onto \mathbb Y\,$  is  energy-minimal, one must look into  the Hopf differential $h_z \overline{h_{\bar z}}\; \dtext z \otimes \dtext z\,$. Here is the recipe.

\begin{theorem}\label{thmhopf}
Let $\X$ and $\Y$ be bounded doubly connected domains,  $\,\X\,$ being nondegenerate. Then a mapping $h \in \overline{\Ho}_2 (\X, \Y)$ is  energy-minimal if and only if its Hopf-differential is analytic and real along $\partial \X$.
\end{theorem}
It is surprising that this very useful criterion has not been established before.
  For the convenience of the reader we devote the entire \S\ref{InnerVar} to the inner variations and the associated Hopf differentials; in particular, see Definition~\ref{RealDiff} that is pertinent to Theorem~\ref{thmhopf}. The key observations in the proof of the uniqueness Theorem \ref{thmuni} is that the difference of two solutions to the same  Hopf-Laplace equation has  finite (not necessarily bounded) distortion on the set where at least one of the solutions has strictly positive Jacobian, see \S\ref{Distortion}. 

\textit{Propagation of cracks}  is an interesting phenomenon not only in modern theories of elasticity and plasticity, materials science or microscopic crystallographic defects found in real materials but also from mathematical point of view.   Let $\,h \in    \overline{\Ho}_2 (\X, \Y)$ be an energy-minimal map between multiply connected domains. It should be noted that in general $h$, being a limit of homeomorphisms from $\mathbb X$  onto $\mathbb Y$,  has range in the closure of $\mathbb Y$,  $\, h \colon \mathbb X \to\overline{\mathbb Y}$. The fact, referred to as \textit{partial harmonicity},  is that every $\,h \in    \overline{\Ho}_2 (\X, \Y)\,$is a harmonic diffeomorphism of $\,h^{-1}(\mathbb Y) \subset \mathbb X\,$  onto $\mathbb Y$, see~\cite{CIKO} and Proposition~\ref{Partialharmonicity}.
There might, however, be sets in $\,\mathbb X\,$ which are taken into a single point in $\,\partial \mathbb Y\,$.

\begin{definition}
Given a point $\,a\in \partial \mathbb Y\,$, the term \textit{crack} (or $\,a$-crack) in $\,\mathbb X\,$ refers to any connected component of the set $\,\{\,x \in \mathbb X \colon  h(x) = a\,\}\,$.
\end{definition}
It will be  seen (at least when $\,\mathbb Y\,$  is Lipschitz regular) via Lemma\,\ref{cdLemma}\,   that cracks always emanate  from $\,\partial\mathbb X\,$; thus, are  never reduced to a single point  or any continuum in $\mathbb X\,$.

How to predict occurrence of cracks?  One practical evidence is contained in the following result.
\begin{theorem}\label{nocracks}
Let $\,h \in    \overline{\Ho}_2 (\X, \Y)\,$ be an energy-minimal map, where  $\,\X\,$ is a Jordan domain and $\,\Y\,$ a Lipschitz domain, both multiply connected. Suppose $\,\mathbb Y\,$ is convex at a boundary point $\,a\in \partial \Y\,,$  meaning that the set $\, B \cap \mathbb Y\,$ is convex for some ball $\,B\,$ centered at $\,a\,$. Then  $\,a \not \in h(\X)\,$.
\end{theorem}
This theorem may be interpreted as saying that cracks do not collapse  into a boundary point at which $\, \mathbb Y\,$ is convex, see Theorem \ref{critical radii} and  \S \ref{NitscheHammering}\,, see also  \cite{IOnAnnuli} for similar phenomenon in higher dimensions.
  In spite of the impressive progress in the field,  formation of cracks under energy-minimal deformations, to our knowledge, is not fully resolved. In this domain the best reference is the book by K. B.  Broberg~\cite{Br}.

 The problem clearly depends on the geometry of trajectories of the Hopf differential $h_z \overline{h_{\bar z}}\; \dtext z \otimes \dtext z$.
  In this paper we give a detailed description of cracks in case of doubly connected domains. Accordingly, if the conformal modulus of $\,\mathbb Y\,$ (deformed configuration) is large relative to $\,\mathbb X\,$ then the Hopf differential $h_z \overline{h_{\bar z}}\; \dtext z \otimes \dtext z$ is real and negative along $\partial \mathbb X$. Consequently no cracks occur, even in the presence of points of concavity in the  boundary of $\,\mathbb Y\,$, see~\cite{IKKO} and Theorem~\ref{thmcpo}.  If, on the other hand, the target is conformally very thin then the cracks are unavoidable. In this case the boundary components of $\,\mathbb X\,$ are horizontal trajectories along which the Hopf differential is positive. The cracks are born in  $\,\partial \mathbb X\,$ and propagate along the vertical trajectories toward the interior of $\,\mathbb X\,$ where they eventually terminate. In particular, cracks create no crosscuts in $\,\mathbb X\,$; the domain $\,\mathbb X\,$ with cracks remains doubly connected, see Theorem \ref{thmcracks}. 
  
  A thoughtful analysis of the doubly connected case shows that one needs, in order to reveal the curiosities and the mechanics of cracks, investigate and explore the geometry of trajectories of the Hopf quadratic differential $h_z \overline{h_{\bar z}}\; \dtext z \otimes \dtext z$ by means of the level sets of $h$. Theoretical prediction of failure of bodies caused by cracks  is a good motivation that should appeal to mathematical analysts and researchers in the engineering fields.

\section{Notation and Preliminaries}
\subsection{Domains}

We shall study a pair of bounded domains $\mathbb X\,,\,\mathbb Y \subset \mathbb R^2\simeq \mathbb C\,$ of the same connectivity $ 1\leqslant \ell <\infty\,. $  This amounts to saying that each of the complements $\,\mathbb C \setminus \mathbb X\,$ and $\,\mathbb C \setminus \mathbb Y\,$ consists of $\,\ell\,$ disjoint closed connected sets.  Let us introduce the notation,
$$\mathbb X_1, \mathbb X_2, ..., \mathbb X_\ell \;\;\;\;\;\;\; \textnormal{-the components of}  \;\;\mathbb C \setminus \mathbb X $$
$$\mathbb Y_1, \mathbb Y_2, ..., \mathbb Y_\ell  \;\;\;\;\;\;\textnormal{-the components of} \;\; \mathbb C \setminus \mathbb Y$$
Their boundaries are exactly the components of $\partial \mathbb X$ and $\partial \mathbb Y$, respectively;
$$\partial\mathbb X_1, \partial\mathbb X_2, ..., \partial\mathbb X_\ell  \;\;\;\;\;\;\textnormal{-the components of} \;\;\partial \mathbb X $$
$$\partial\mathbb Y_1, \partial\mathbb Y_2, ..., \partial\mathbb Y_\ell  \;\;\;\;\;\;\textnormal{-the components of} \;\;\partial \mathbb Y $$

Note that adding any number of sets in $\{\mathbb X_1, \mathbb X_2, ..., \mathbb X_\ell\} $ to $\mathbb X$  results in a domain. \subsubsection{Doubly connected domains}
For a doubly connected domain $\X$  we reserve special notation, $\X_I$ and  $\X_O$, for the bounded and unbounded components of $\C \setminus \X$\,, respectively.  Thus $\,\partial \X\,$ consists  of two continua $\,\partial \mathbb X_I\,$ and $\,\partial \mathbb X_O\,$, referred to as  \textit{inner} and \textit{outer boundaries}.

\subsubsection{Boundary Correspondence}

Every homeomorphism $\, h\colon  \mathbb X  \overset{\textnormal{\tiny{onto}}}{\longrightarrow} \mathbb Y $ gives rise to a one-to-one correspondence between boundary components of $\mathbb X$ and boundary components of $\mathbb Y\,$ which, upon a suitable arrangement of indices, will be designated as
\begin{equation}\label{Correspondence}   h \colon  \partial \mathbb  X_\nu \rightsquigarrow \partial \mathbb Y_\nu \;,\;\;\; \textnormal{for}\;\;\; \nu = 1,..., \ell
\end{equation}
Precisely this means that the cluster set of $\,\partial \mathbb  X_\nu\,$ under $\,h\,$ is the boundary component $\,\partial \mathbb  Y_\nu\,$. The components will be so ordered that the outer boundary of $\,\mathbb X\,$ corresponds to the outer boundary of $\,\mathbb Y\,$.
\subsection{The classes $\,\mathscr H(\mathbb X, \mathbb Y)\,$ and $\,\mathscr H_{cd}(\mathbb X, \mathbb Y)$}
Let $\,\mathscr{H}(\mathbb X, \mathbb Y)\,$  denote the class of all orientation preserving  homeomorphisms $\, h\colon  \mathbb X  \overset{\textnormal{\tiny{onto}}}{\longrightarrow} \mathbb Y $ which preserve the order of boundary components  in  (\ref{Correspondence}) .\\
The following geometric concept of convergence \cite{IKKO} proves very useful.
 \begin{definition} A sequence of continuous mappings $h_j \colon  \mathbb X \rightarrow \mathbb Y\, $ is said to converge \textit{cd-uniformly}  to a mapping  $h \colon  \mathbb X \rightarrow \overline{\mathbb Y} $  if ;
 \begin{itemize}
 \item $ h_j \rightarrow h \,$ ,\; $c$-uniformly \,(uniformly on compact subsets of $\mathbb X$ )
 \item $\textnormal{dist}[ h_j(x),\, \partial \mathbb Y ] \; \rightarrow \;\textnormal{dist}[ h(x),\, \partial \mathbb Y ] \;$, \;\;uniformly  in $ \mathbb X$
 \end{itemize}
 We write it as $h_j  \overset{\textnormal{\tiny{cd}}}{\longrightarrow} h\,$.
 \end{definition}

 Let $\,\mathscr H_{cd}(\mathbb X, \overline{\mathbb Y}) \,$ denote the class of   $\,cd$ -limits of homeomorphisms  in  $\mathscr H(\mathbb X, \mathbb Y) \,.$
  Passing to the $cd$-limit one may lose injectivity but the boundary correspondence in (\ref{Correspondence}) remains the same.
  It is a matter of topological routine to see that

  \begin{lemma}\label{cdLemma}
  For every $\,h \in \mathscr H_{cd}(\mathbb X, \overline{\mathbb Y}) \,$ it holds
  \begin{itemize}
  \item [(i)] \;\;\;$\mathbb Y \subset h(\mathbb X) \subset \overline{\mathbb Y} $
  \item [(ii)]\;\; Given a boundary point  $\,a\in \,\partial \mathbb Y_\nu \subset \partial \mathbb Y\,$, every component of the set $\, h^{-1}(a) = \{\, x \in \mathbb X\,;\, h(x)  = a\,\}\,$ \textnormal{(later interpreted as crack)}\, touches the boundary component $\,\partial \mathbb X_\nu \subset \partial \mathbb X\,$, i.e.  its closure intersects $\,\partial \mathbb X_\nu\,$.
  \end{itemize}
  \end{lemma}
Another concept we are going to take from \cite{IKKO} is
\begin{definition}
A mapping $\,h \colon  \mathbb X \rightarrow \overline{\mathbb Y}\,$  is called \textit{deformation} if
\begin{itemize}
\item $\,h \in \mathscr W^{1,2}(\mathbb X)\,$
\item $\,J(x, h) \geqslant 0\,$  almost everywhere
\item $\,\iint_\mathbb X J(x, h) \,\textnormal{d} x  \leqslant |\mathbb Y| \,$
\item There exist sense preserving homeomorphisms $\, f_j \colon \mathbb X \onto \mathbb Y\,$ (called an approximating sequence) converging $\,cd$ -uniformly to $\,h\,$ .
\end{itemize}
Let us emphasize that the approximating homeomorphisms $\,f_j\,$ need not have finite energy. The set of deformations is denoted by $\,\mathfrak D (\mathbb X , \mathbb Y)\,$.
\end{definition}

\subsection{The Dirichlet Energy} The Dirichlet energy (or Dirichlet integral)  for functions $\,h \colon  \mathbb X  \rightarrow \mathbb C\,$ in the Sobolev space  $\,\mathscr W^{1,2}(\mathbb X,\mathbb C)\,$ is defined by
\begin{equation}\label{direnergy}
\mathscr E _{_{\mathbb X}} [h]= \iint_{\X} \abs{Dh}^2 =2 \iint_{\X} \left( \abs{h_z}^2+ \abs{h_{\bar z}}^2 \right)\, \dtext z\;,\;\;\;\dtext z = \dtext x_1\, \dtext x_2
\end{equation}
Here and throughout, we take advantage of the complex variable $z = x_1 + i\,x_2  $ and partial derivatives (Wirtinger operators)
$$
 \quad \quad h_z  = \frac{\partial h}{\partial z}= \frac{1}{2} \left( \frac{\partial h}{\partial x_1} - i \frac{\partial h}{\partial x_2}  \right)\,,   \quad
h_{\bar z}  = \frac{\partial h}{\partial \bar z}= \frac{1}{2} \left( \frac{\partial h}{\partial x_1} + i \frac{\partial h}{\partial x_2}  \right).
$$
Thus, the Jacobian determinant is
\begin{equation}
J(z, h) = J_h(z) = \textnormal{det}\, Dh (z) = |h_z|^2 - |h_{\bar z} |^2
\end{equation}
\subsection{The Space $\mathscr{H}_{_2}(\mathbb X, \mathbb Y)\,$  and other Relevant Classes of Mappings}
Various classes of mappings of finite energy will be considered.

\begin{definition}\label{spaces} Here they are:
\begin{itemize}
\item $\mathscr{H}_{_2}(\mathbb X, \mathbb Y) = \mathscr{H}(\mathbb X, \mathbb Y)\cap \mathscr W^{1,2}(\mathbb X,\mathbb Y) \,$.
    \item $\overline{\mathscr{H}_{_2}}(\mathbb X, \mathbb Y)\,$ is the closure of $\,\mathscr{H}_{_2}(\mathbb X, \mathbb Y)\,$ in strong topology of $\,\mathscr W^{1,2}(\mathbb X,\mathbb Y)\,$
        \item $\widetilde{\mathscr{H}_{_2}}(\mathbb X, \mathbb Y)\,$ is the closure of $\,\mathscr{H}_{_2}(\mathbb X, \mathbb Y)\,$ in weak topology of $\,\mathscr W^{1,2}(\mathbb X,\mathbb Y)\,$
    \item $\textit{Diff}_{_{\,2}}\,(\mathbb X\,,\mathbb Y)\,$ consists of  $\mathscr C^\infty$-diffeomorphisms in $\mathscr{H}_{_2}(\mathbb X, \mathbb Y)\,$. The fact is that  $\overline{\mathscr{H}_{_2}} (\mathbb X, \mathbb Y) = \overline{\textit{Diff}_{_{\,2}}}\,(\mathbb X\,,\mathbb Y)\,$, see \cite{IKOApproximation}.
        \item $\,{\mathscr{H}}^{1,2}_{\lim}(\mathbb \X, \mathbb Y)\,$ stands for the family of all weak limits of homeomorphisms  $\,h_j \colon  \mathbb X  \onto \mathbb Y\,, \; j = 1, 2, ...\,,$ in the Sobolev Space $\,W^{1,2}(\mathbb X, \mathbb Y)$.
\end{itemize}

\end{definition}

Let us call attention to an important detail that $\,h\,$, in any of the above classes, has  a range  in the closure of the target domain; that is, $\,h \colon  \mathbb X \rightarrow \overline{\mathbb Y}\,$.

\begin{remark}

In general, the family of all weak limits of a set in a Banach space need not be weakly closed. That is why we have only the inclusion $\,{\mathscr{H}}^{1,2}_{\lim}(\mathbb \X, \mathbb Y)\, \varsubsetneq\widetilde{\mathscr{H}_{_2}}(\mathbb X, \mathbb Y)\,$. However, if $\X$ is a Jordan domain and $\Y$ is a Lipschitz domain, both not simply connected, then
\begin{equation}\label{setiden}
{\mathscr{H}}^{1,2}_{\lim}(\mathbb \X, \mathbb Y)\, = \widetilde{\mathscr{H}_{_2}}(\mathbb X, \mathbb Y)\, = \overline{\mathscr{H}}_2(\mathbb \X, \mathbb Y)\, =  \overline{\mathscr{H}}_2(\overline{\mathbb \X}, \overline{\mathbb Y}) \subset \mathscr H_{cd}(\mathbb X, \mathbb Y) .
\end{equation}
For the first two equalities we refer to~\cite{IOa}.
The  symbol  $\,\overline{\mathscr{H}}_2(\overline{\mathbb \X}, \overline{\mathbb Y}) \,$ denotes the class of strong limits  of homeomorphisms up to the boundaries, $\,h_j \colon  \overline{\mathbb X} \onto \overline{\mathbb Y}\,$, in the Sobolev space $\,\mathscr W^{1,2}(\mathbb X, \mathbb Y)\,$; the equality is immediate from~\cite{IKOsa}. The last inclusion is then obvious.
\end{remark}

Finally, we infer from \cite{IKKO} that
\begin{lemma}\label{inclusions}
Let $\,\mathbb X\,$ and $\,\mathbb Y\,$ be multiply connected domains, $\,\mathbb X\,$ being nondegenerate.  If a sequence $\,\{ h_j\} \subset \mathfrak D(\mathbb X, \mathbb Y)\,$ converges weakly in $\,\mathscr W^{1,2}(\mathbb X\,\mathbb Y)\,$, then its limit belongs to $
 \,\mathfrak D(\mathbb X, \mathbb Y)\,$. Thus, in particular
 \begin{equation}
 \mathscr H^{1,2}_{\textnormal{lim}}(\mathbb X, \mathbb Y) \subset \mathfrak D (\mathbb X , \mathbb Y) \subset \mathscr H_{cd}(\mathbb X, \overline{\mathbb Y})
 \end{equation}
 \begin{equation}
 \mathbb Y \subset h(\mathbb X) \subset \overline{\mathbb Y}  \;,\;\;\textnormal{for every}\;\; h \in \mathscr H^{1,2}_{\textnormal{lim}}(\mathbb X, \mathbb Y)
 \end{equation}
\end{lemma}

\subsection{Monotone mappings}
Monotone mappings were born  in the work of C.B. Morrey  \cite{Mor}. We will be concerned with compact subsets of $\,\mathbb R^2\,$ so let us combine Morrey's original idea with a theorem of G.T. Whyburn , see page 2  in \cite{McAuley}, to make the following definition:
 \begin{definition}[monotonicity]
 A continuous mapping $\,f\colon  \mathbf X \onto \mathbf Y\,$ between compact metric spaces is monotone if for each connected set $\,C \subset \mathbf Y\,$ its preimage $\,f^{-1}(C)\,$ is connected in $\mathbf X\,$.
 \end{definition}
\begin{theorem}[Kuratowski-Lacher]\label{Kuratowski-Lacher}
   Let $\,\mathbf X\,$ and $\,\mathbf Y\,$ be compact Hausdorff spaces,  $\,\mathbf Y\,$ being locally connected. Suppose we are given a sequence of monotone mappings $\,f_k \colon  \mathbf X \onto \mathbf Y\,$  converging uniformly to a mapping $\, f \colon \mathbf X \rightarrow \mathbf Y\,$, then $\,f \colon  \mathbf X \onto \mathbf Y\,$ is monotone.
 \end{theorem}
See \cite{K, KL} and \cite{Rado, Why} for further reading about monotone mappings.
\subsection{Equicontinuity and monotonicity}
 \begin{lemma}\label{Equicontinuity} Let $\,\mathbb X\,$  and $\,\mathbb Y\,$ be  Lipschitz domains of connectivity $\,\ell \geqslant 2\,$. Then,
 every $\,h \in {\mathscr{H}}^{1,2}_{\lim}(\mathbb \X, \mathbb Y)\,$ extends to a continuous monotone map $\,h \colon \overline{\mathbb X} \onto \overline{\mathbb Y}\,$. The boundary map $\, h \colon \partial\overline{\mathbb X} \onto \partial \overline{\mathbb Y}\,$ is also monotone. Moreover,
 \begin{equation}\label{r2}
\abs{h(x_1)- h(x_2)}^2 \le \frac{C(\X, \Y)}{\log \left(e+ \frac{\textnormal{diam}\, \mathbb X}{\abs{x_1-x_2}}\right)} \iint_{\X} \abs{D h}^2, \quad x_1, x_2 \in \overline{\X}.
\end{equation}
 \end{lemma}
\begin{proof} It is known that a homeomorphism $\,h \colon \mathbb X \onto\mathbb Y\,$ between Lipschitz domains in the Sobolev space $\,\mathscr W^{1,2}(\mathbb X , \mathbb Y)\,$ extends continuously up to the boundaries \cite{IOtr}. And it is topologically clear that a continuous extension of a homeomorphism $\,h \colon \mathbb X \onto\mathbb Y\,$ results in  monotone mappings $\,h \colon  \overline{\mathbb X} \onto\overline{\mathbb Y}\,$ and $\,h \colon  \partial\mathbb X \onto\partial\mathbb Y\,$.  These properties carry over to the uniform limits of homeomorphisms. Thus it only remains to justify the uniform estimates (\ref{r2}). Local estimates like this are well known  for monotone mappings of Sobolev class $\,\mathscr W^{1,2}(\mathbb X\,,\mathbb Y )\,$ \cite{IKosO, IMb}. The global inequality follows by the standard method of extending $\,h\,$ beyond the boundaries. It is at this point that the Lipschitz regularity of the domains is required again.  We leave the routine details to the readers.
\end{proof}
\begin{remark}\label{normalization}

Similar arguments, including the extension procedure, apply to a pair of simply connected Lipschitz domains if the mappings are fixed at some interior point ($\,h(x_\circ) = y _\circ\,$ for some $\,x_\circ \in \mathbb X\,$ and $\,y_\circ \in \mathbb Y\,$) or at three boundary points ($\,h(x_1) = y _1 \,,\, h(x_2) = y _2 \,,\, h(x_2) = y _2 \,$). We shall return to this later case in \S \ref{washer-nut}, where we will discuss quadrilateral mappings.

\end{remark}

\subsection{Rad\'{o}-Kneser-Choquet theorem}
The following proposition strengthens the well know theorem of Rad\'{o}-Kneser-Choquet.
\begin{theorem}\label{RaKnCh}
 A harmonic map $\,h \colon \Omega \rightarrow \mathbb C\,$ of a Jordan domain, which extends continuously as a monotone map of $\,\partial \Omega\,$ onto a boundary of a convex region is a $\,\mathscr C^\infty$ -diffeomorphism of $\, \Omega\,$ onto this region.
\end{theorem}
For the proof and the best general references see ~\cite{Duren}.

\section{Inner Variation and Hopf Differentials}\label{InnerVar}
Although the concept of inner variation has been used in the Calculus of Variations for a long time \cite{BOP, Cob,  Job, Ta, SSe}, see also more recent ones \cite{CIKO, IKOli, Mo, Ya}, some of its nuances are still to be scrutinized.

\subsection{Inner Variation}
Throughout this section $\,\mathbb X\subset \mathbb R^2 \backsimeq \mathbb C\,$\, is a bounded domain whose points, also called \textit{variables}, will be denoted by $\,x = x_1 + i x_2 \,$. To cover all the instances discussed later, we distinguish a compact subset $\,\mathbb K \subset \partial \mathbb X\,$, possibly empty or the entire boundary.
\subsubsection{Change of variables}
The term change of variables in $\mathbb X\,$ that is fixed at the compact subset $\,\mathbb K \subset \partial \mathbb X\,$, refers to a $\mathscr C^\infty $-diffeomorphism $\,\Psi \colon \mathbb X \onto\mathbb X \,$ which is continuous up to $\,\partial \mathbb X\,$ and satisfies  $\, \Psi(x)  =  x\,$ for $\,x \in \,\mathbb K\,$.

\subsubsection{Variation  of variables in $\,\mathbb X\,$ }\label{changeofVariables} This is a one-parameter family of change of variables $\,\{\Psi^{\,\varepsilon}\}_{-\varepsilon_\circ < \varepsilon <\varepsilon_\circ}\,$, such that

 \begin{itemize}
 \item [(i)] The function $\,(\varepsilon, x) \rightsquigarrow  \Psi^{\,\varepsilon}(x)\,$ is $\,\mathscr C^\infty $-smooth in $\,(-\varepsilon_\circ \,,\,\varepsilon_\circ )\times {\mathbb X}\,$
 \item [(ii)] $\, \Psi^{\,0} = id  \,$\;\; in \;$\mathbb X\,$
 \item [(iii)] $\, \Psi^{\,\varepsilon} = \,id\,$\;\; in\; $\mathbb K\,$,\;\;\textnormal{for every}\;\;\;   $\,\varepsilon \in (-\varepsilon_\circ\, ,  \,\varepsilon_\circ )\,$
 \end{itemize}

  Define $\,\lambda \deff \; \frac{\textnormal d}{\textnormal d  \varepsilon}\big{|}_{\varepsilon = 0}  \Psi^{\,\varepsilon}\;$  and note that  $\,D\lambda \equiv \frac{\textnormal d}{\textnormal d  \varepsilon}\big{|}_{\varepsilon = 0} D\Psi^{\,\varepsilon}\,$.
\subsubsection{Inner variation of a mapping} Let $\, h\,$ be a function in the Sobolev space $\,\mathscr W^{1,2}(\mathbb X , \mathbb C)\,$ and $\,\{\Psi^{\,\varepsilon}\}\,$  a  variation of variables in $\,\mathbb X\,$ that is fixed at $\,\mathbb K \subset \partial \mathbb X\,$.  The family $\, h^\varepsilon =  h\,\circ\, \Psi^\varepsilon\, $ is called an \textit{inner variation} of $\,h\,$.  Observe that all the mappings $\, h^\varepsilon\,$ have the same range as $\,h\, $, which is one of the desired  properties that motivates the use of inner variations. Additional assumption on the behavior of $\,\Psi^\varepsilon\, $ near $\,\partial \mathbb X\,$ will be needed in order to ensure that $\,h^\varepsilon \in \mathscr W^{1,2}(\mathbb X , \mathbb C)\,$. In our applications, however, this property will always be satisfied, so there is no need for formulating any explicit conditions here. Nevertheless, the interested reader may notice that if all the diffeomorphisms $\,\Psi^{\,\varepsilon} \colon \mathbb X \onto\mathbb X\,$  are $\,K$-quasiconformal; that is, $\, |D\Psi^\varepsilon | \leqslant K \,\textnormal{det}
  \, D \Psi^\varepsilon \;$, then
 $$
 \mathscr E_\mathbb X [h^\varepsilon] \;= \; \iint _\mathbb X |Dh^\varepsilon(x)|^2 \,\textnormal{d} x\;   \leqslant    K  \iint _\mathbb X |Dh(x)|^2 \,\textnormal{d} x\;    < \infty
 $$

  \subsubsection{Critical points} Choose and fix a variation  $\,\{\Psi^{\,\varepsilon}\}_{-\varepsilon_\circ < \varepsilon <\varepsilon_\circ}\,$ of variables in $\,\mathbb X\,$.  We say that $\,h \in \mathscr W^{1,2}(\mathbb X , \mathbb C)\,$ is a \textit{critical point} for
 $\,\{\Psi^{\,\varepsilon}\}\,$ if
  $$
  \frac{\textnormal{d}\,\mathscr E_\mathbb X[h^\varepsilon]}{\textnormal{d}\varepsilon}\;\Big{|}_{ \varepsilon=0}\;\equiv\;0
  $$
  \begin{lemma}[Integral form of the variational equation]\label{WeakHL}  Every critical point for $\,\{\Psi^{\,\varepsilon}\}_{-\varepsilon_\circ < \varepsilon <\varepsilon_\circ}\,$ satisfies the equation
  \begin{equation}\label{critical}
  \re \iint _{\mathbb X }  \overline{h_{\overline{z}}}\, h_z \, \lambda_{\overline{z}} \,\,\textnormal{d} z\; = \;0  \;,\;\;\;\;where \;\;\;\lambda \deff \frac{\textnormal{d}\,\Phi^\varepsilon}{\textnormal{d}\varepsilon}\;\Big{|}_{ \varepsilon=0}
  \end{equation}
  \end{lemma}
\begin{proof}
 The peculiarity of our derivation of (\ref{critical})  lies in pointing out the relevance of the equation to the \textit{Cauchy-Green stress tensor} of $\, h\,$ and the \textit{Ahlfors infinitesimal deformation operator} for $\,\lambda\,$. These connections were certainly overlooked in the literature;  both are worth noting for prospective generalizations. For example, the same ideas work for the  conformally invariant  $\,n$-harmonic integrals in  $\,\mathbb R^n\,$, see \,\cite{IOnAnnuli}.\\
Let us begin with the composition rule $\, Dh^\varepsilon = Dh(\Psi^\varepsilon) \cdot D \Psi^\varepsilon \,$. We express the Hilbert-Schmidt norm of  the differential matrix \,$Dh^\varepsilon\,$ via scalar product of matrices
\begin{equation}
\begin{split}
 |Dh^\varepsilon |^2 &= \langle Dh^\varepsilon \, |\,Dh^\varepsilon  \rangle  \;=\; \langle Dh(\Psi^\varepsilon) \cdot D \Psi^\varepsilon\,\,|\,\,Dh(\Psi^\varepsilon) \cdot D \Psi^\varepsilon\rangle\, \\ & =
 \langle \,D^* h (\Psi^\varepsilon)\cdot Dh(\Psi^\varepsilon)\, \,\,|\,\, \,D \Psi^\varepsilon \cdot D^*\Psi^\varepsilon\,\rangle\,
\end{split}
\end{equation}
where $\,D^*\,$ stands for the transposed differential.
Then we integrate it over $\,\mathbb X\,$ by substitution $\,\xi = \Psi^\varepsilon(z)\,$

\begin{equation}
 \mathscr E_\mathbb X [h^\varepsilon]  = \iint_\mathbb X \Big{\langle} \,D^*h\cdot Dh  \,\; \Big{|}\,\; \,\frac{D \Psi^\varepsilon \cdot D^*\Psi^\varepsilon}{ \det  D \Psi^\varepsilon }  \, \Big{\rangle } \, \textnormal{d}\xi\;
\end{equation}
At this point one may recall the  Cauchy-Green stress tensor of $ \,h \,$
$$
\textbf{C}[h] \,\deff \,D^*h\cdot Dh
$$
as well as  Weyl's conformal tensor of $\,\Psi^\varepsilon\,$
 $$
\textbf{W}[\Psi^\varepsilon]\,\deff \,\;
\,\frac{D \Psi^\varepsilon \cdot D^*\Psi^\varepsilon \;
\;}{\;\det  D \Psi^\varepsilon}
$$
Since  $\,\Psi^0 = \textnormal{id}\,$,  $\,\frac{\textnormal{d} \Psi^\varepsilon}{\textnormal{d} \varepsilon} \big{|}_{\varepsilon =0}\equiv \lambda\,$ and   $\,\frac{\textnormal{d} D\Psi^\varepsilon}{\textnormal{d} \varepsilon}\big{|}_{\varepsilon =0}\equiv  D \lambda\,$, the infinitesimal variation of $\,\textbf{W}[\Psi^\varepsilon]\,$ at the identity map is readily computed as

  $$
   \frac{\textnormal{d} \textbf{W}[\Psi^\varepsilon]}{\textnormal{d} \varepsilon} {\,\Big |}_{\varepsilon = 0}\;= \;\,D\lambda\,+\, D^*\lambda\, -\,(\textnormal{Tr}\, D\lambda )\, I
  $$
  Here too, we recognize the Ahlfors infinitesimal deformation operator \cite{A1, A2, R}  which defines the conformal component of the matrix $\,D\lambda\,$,

  \begin{displaymath}
  \textbf{S} \lambda \,\deff \, \frac{1}{2}\left[D\lambda \,+\, D^*\lambda \,-\,(\textnormal{Tr}\, D\lambda)\, I \right]\, = \, \left[ \begin{array}{cc}
  \re  \lambda_{\overline{z}}\;\;, & \im \,\lambda_{\overline{z}} \\
  \im \,\lambda_{\overline{z}}\;\;, & -\,\re \lambda_{\overline{z}}
  \end{array} \right]
  \end{displaymath}
  The anticonformal component is afforded by the \textit{Ahlfors' adjoint} operator
  \begin{displaymath}
  \textbf{A} \lambda \,\deff\, \frac{1}{2}\left[ D\lambda \,-\, D^*\lambda \,+\,(\textnormal{Tr}\, D\lambda)\, I \right]\, = \, \left[ \begin{array}{cc}
  \re \lambda_{z}\;\;, & -\im \lambda_{z} \\
  \im \,\lambda_{z}\;\;, & \,\re \lambda_{z}
  \end{array} \right]
  \end{displaymath}
  Thus we have a pointwise orthogonal decomposition  $ \,D\lambda = \textbf{S} \lambda + \textbf{A}\lambda\, $.
Now, a straightforward computation gives  the variation of the energy integral

\begin{equation}\label{InnerVariation}
\begin{split}
\frac{\textnormal{d}\,\mathscr E_\mathbb X[h^\varepsilon]}{\textnormal{d}\varepsilon}\;\Big{|}_{ \varepsilon=0}& =   2\,\iint_\mathbb X \Big{\langle} \,D^*h\,\cdot Dh  \,\; \Big{|}\,\; \textbf{S}\,\lambda  \, \Big{\rangle } \, \textnormal{d}\xi\;\\&= \;
8\,\re \iint _{\mathbb X }  \overline{h_{\overline{z}}}\, h_z \, \lambda_{\overline{z}} \,\,\textnormal{d}z\;
\end{split}
\end{equation}
because $\,\big{\langle} \,D^*h\,\cdot Dh  \,\; \big{|}\,\; \textbf{S}\,\lambda  \, \big{\rangle }\,= \, 4\,\re (\overline{h_{\overline{z}}}\, h_z \, \lambda_{\overline{z}})\,$ .
\end{proof}

\subsection{Hopf-Laplace equation}
 Every complex function $\,\lambda \in \mathscr C^\infty_0 (\mathbb X, \mathbb C)\,$ gives rise to a variation $\,\Psi^\varepsilon (z) =  z + \varepsilon \lambda(z)\,$ of variables in $\,\mathbb X\,$ that is fixed at every point of $\,\partial \mathbb X\,$. This variation is legitimate for any function $\,h \in \mathscr W^{1,2}(\mathbb X,\mathbb C)\,$ that is fixed at a given compact subset $\,\mathbb K \subset \partial \mathbb X\,$. In particular, applying (\ref{critical}) to such $\,\lambda\,$  we infer via the classical Weyl's lemma that

 \begin{proposition}[Hopf-Laplace Equation]   A function $\,h \in \mathscr W^{1,2}(\mathbb X , \mathbb C)\,$ that is critical for all inner variations  (subject or not to boundary constraints)  must satisfy the \textit{Hopf-Laplace equation}  in $\,\mathbb X\,$,
  \begin{equation}\label{Critical1}
    \frac{\partial}{\partial \bar{z}}\,h_z\,\overline{h_{\overline{z}}}\, = 0\;\,, \;\;\;\;(\textnormal{in the sense of distributions})
  \end{equation}
  equivalently,
  \begin{equation}\label{Critical2}
   \,h_z\,\overline{h_{\overline{z}}}\, = \varphi(z) \;\,,\;\;\;\; \textnormal{for some analytic function} \;\;\;\;\varphi \in \mathscr L^1(\mathbb X)
  \end{equation}
  \end{proposition}
  If one allows the critical mapping $\,h\,$ to slip along an arc $\,\gamma \subset \partial \mathbb X\,$ then an additional equation on $\,\gamma\,$ will emerge, which in turn yields a specific form of the analytic function $\,\varphi\,$.  Let us begin with an example of such situation.
  \begin{example}[Critical solutions in an annulus]\label{criticalAnnulus}
   Consider a traction free problem in an annulus $\,\mathbb X = \{\, z \,;  \,r < |z| < R\,\}\,$; that is, allow $\,h \colon  \partial \mathbb X \to \partial \mathbb Y\,$ to slide along the boundary circles of $\,\mathbb X\,$.
  \begin{proposition}\label{HopfAnnulus}
    The Hopf-Laplace equation for a function  $\, h \in \mathscr W^{1,2}(\mathbb X , \mathbb C )\,$ that is  critical for all inner variations in $\,\mathbb X\,$ takes the form
  \begin{equation}
  \,h_z\,\overline{h_{\overline{z}}}\, = \frac{c}{z^2} \;\,,\;\;\;\; \textnormal{for all} \;\;\;\;z \in \mathbb X
  \end{equation}
  where $\, c \,$ is a real number
  \end{proposition}

  \begin{proof} We expand the Hopf product of $\,h\,$ into a Laurent series
 $$ \,h_z\,\overline{h_{\overline{z}}}\, = \varphi(z) =  \sum_{n = - \infty} ^ \infty  a_n \,z^n
  $$
  Then we test the integral equation (\ref{critical}) with the following variations of variables in $\,\mathbb X\,$,
  $$\Psi^\varepsilon (z)  = \; z \cdot  \frac{1 + \varepsilon ( a \bar{z}^k\;-\; \bar{a} z^k )}{ [\,1 - \varepsilon^2 ( a \bar{z}^k\;-\; \bar{a} z^k )^2 \,]^{1/2}}\;, \;\;\; k =  \pm 1, \pm 2, ...$$
  where $\,a\,$ can be any complex number and $ \varepsilon$ any sufficiently small real number. Evidently $\,|\Psi^\varepsilon (z) | \equiv |z|\,$, so $\,\Psi^\varepsilon \colon  \mathbb X \onto \mathbb X\,$ can easily be shown to be a $\,\mathscr C^\infty$ -diffeomorphism. A short computation gives
  $$
  \lambda \deff \frac{\textnormal{d}\,\Psi^\varepsilon}{\textnormal{d}\varepsilon}\;\Big{|}_{\varepsilon=0} \,=\, z ( a \bar{z}^k\;-\; \bar{a} z^k )\;\;\;\textnormal{and}\;\;\; \lambda_{\bar{z}} \,=\, k \,a z\, \bar{z} ^{k-1}
  $$
  Put these values of $\,\lambda_{\bar{z}} \,$ into the equation (\ref{critical}) to obtain

   $$  0 \,=\, \re \iint_\mathbb X \Big(\sum_{n = - \infty} ^ \infty  a_n \,z^n \Big)\,k \,a z\, \bar{z} ^{k-1}\; \textnormal{d} z\;=\; \re ( k\,a\, a_{k-2}) \,\iint_\mathbb X  |z| ^{2k-2} \;\textnormal{d} z
  $$
  for every complex number $\,a\,$. This yields $\,a_{k-2} = 0\,$, except for $\,k= 0\,$. We just proved that $\,\varphi(z) =  a_{_{-2}}\, z^{-2}\,$. To see that the coefficient $\, a_{_{-2}}\,$ is real we test the integral equation (\ref{critical} again, but with the following variation of variables in $\,\mathbb X\,$.
  $$\Psi^\varepsilon (z)  = \; z \cdot  \frac{1 + \,i\, \log |z|}{ [\,1+ \varepsilon^2 \log^2 |z| \,]^{1/2}}\,\;\;\;\lambda \deff \frac{\textnormal{d}\,\Phi^\varepsilon}{\textnormal{d}\varepsilon}\;\Big{|}_{\varepsilon=0} \,=\, i\,z \log |z|\;\;\;\textnormal{and}\;\;\; \lambda_{\bar{z}} \,= \,i \frac{z}{\bar{z}}
  $$
 Then we find that
 $$  0 \,=\, \re \iint_\mathbb X  \frac{a_{_{-2}}}{ z ^2} \; i \, \frac{z}{\bar{z}} \;\textnormal{d} z\, = - \im ( a_{_{-2}} ) \,\iint_\mathbb X  |z| ^{-2} \;\textnormal{d} z\;,\;\;\;\textnormal{hence}\;\;  a_{_{-2}} \in \mathbb R
  $$
  as desired.
  \end{proof}
  \end{example}

\subsection{Variations along a boundary arc}
In the study of \textit{traction free} elastic deformations $\, h \colon  \mathbb X \onto \mathbb Y\,$ one prescribes the values of $\,h\,$ in a compact set  $\,\mathbb K \subset \partial\mathbb X\,$ while allowing to slide along the rest of the boundary. The critical points, subjected to mappings which are traction free along  the set  $\,\Gamma\, \deff \, \partial \mathbb X \setminus \mathbb K \,$, must satisfy not only the Hopf-Laplace equation \,(\ref{Critical2})\,in $\,\mathbb X\,$  but also an equation on $\,\Gamma\,$. We begin with a smooth Jordan arc in $\,\Gamma\,$, and later extend this concept to more general settings, see Definition  \,\ref{RealDiff}.

\subsubsection{Boundary arcs}\label{RegularArcs}

A subset $\,\gamma\subset\partial \mathbb X\,$ is said to be  $\,\mathscr C^k $ -smooth boundary arc  if there is an open set  $\,\mathbb U \subset \mathbb C\,$ and a $\mathscr C^k$ -smooth diffeomorphism  $\,F \colon \mathbb U \onto \mathbb C\,$ such that
\begin{itemize}
\item $\,\mathbb U \cap \partial \mathbb X \,= \gamma \,, \;\;\;\; F(\gamma) = \mathbb R $
\item $\, F(\mathbb X \cap \mathbb U ) = \mathbb C_+ \deff  \{\,z = x_1 + i\, x_2\,\,; \,\, \,x_2 > 0\}$
\end{itemize}
We say that  $\,F \colon \mathbb U \onto \mathbb C\,$  transforms $\,\gamma\,$ into  a line segment in   $\,\mathbb U\,$, and refer to $\,\mathbb U\,$ as vicinity of $\,\gamma\,$.

  \subsubsection{Test functions}
    A complex function  $\,\lambda \in \mathscr C^k_0 (\mathbb U)\,$ is a \textit{test function} for the arc $\,\gamma\subset \partial \mathbb X\,$ if the linear differential $\,\lambda(z)\,\textnormal d\overline{z}\,$ is real along  $\,\gamma $. This  means that for any (or just one) $\,\mathscr C^k$- parametrization of $\,\gamma\,$, say  $\,\gamma = \{ z(t) \colon a<t<b\;\} $, it holds $\,\lambda(z(t))\, \overline{\dot{z}(t)}\,\in \mathbb R\,$. Here, as usual,  $\, \dot{z}(t) \,$ denotes the derivative of $\,z(t)\,$. In other words, the complex number $\,\lambda (z)\,$ represents the tangent vector at  $\,z \in \gamma\,$.  A diffeomorphism $\,F \colon \mathbb U \onto \mathbb C\,$ which transforms the arc $\,\gamma\,$ into a line segment in $\,\mathbb R\,$ gives a function $\, \lambda \circ F^{-1} \in \mathscr C^k_0 (\mathbb C)\,$  that is real along $\,\mathbb R\,$.
    \begin{lemma}[Existence of the variation]\label{ExistVar} Given a smooth boundary arc $\,\gamma\subset \partial \mathbb X\,$ and its test function $\,\lambda \in \mathbb C^k_0 (\mathbb U)\,$,    there exists a variation of variables  $\,\Psi^\varepsilon \colon \mathbb X \onto \mathbb X\,$ along $\,\gamma\,$, such that $\, \frac{\textnormal d}{\textnormal d  \varepsilon}\big{|}_{\varepsilon = 0}  \Psi^{\,\varepsilon}\; = \lambda\,$
    \end{lemma}
     In particular,  by Lemma \ref{WeakHL},  we have
    \begin{corollary}\label{Weakreal} If $\,h \in \mathscr W^{1,2}(\mathbb X,\mathbb C)\,$ is a critical point for all variations along  $\,\gamma\,$ then,
    \begin{equation}
    \re \iint _{\mathbb X }  \overline{h_{\overline{z}}}\, h_z \, \lambda_{\overline{z}} \,\,\textnormal{d} z\; = \;0
    \end{equation}
    for all test functions $\,\lambda \in \mathbb C^k_0 (\mathbb U)\,$ whose linear differential $\,\lambda \,\textnormal{d} \bar{z}\,$ is real on $\,\gamma\,$.
    \end{corollary}
\begin{proof} An explicit construction is as follows.
       Fix a $\,\mathscr C^k$ -diffeomorphism $\,F \colon \mathbb U \onto \mathbb C\,$  which transforms  $\,\gamma\,$  into a line segment in $\mathbb R\,$. Then define the desired one-parameter family of diffeomorphisms by the rule,
 $$
  \Psi^{\,\varepsilon} = F^{-1}\left(F\,+ \, \varepsilon \lambda F_z + \varepsilon \overline{\lambda} F _{\overline{z}}   \right) \colon \mathbb U \cap \mathbb X \rightarrow \,\mathbb C
 $$
 with $\,  |\varepsilon | < \varepsilon_\circ\,$ sufficiently small, and set
 $\,\; \;\Psi^{\,\varepsilon}(z) = z \;\; \textnormal{in}\;\; \mathbb X \setminus \mathbb U
 $.\\
 Here $\,F_z\,$ and $\, F_{\bar{z}}\,$  denote the complex derivatives of   $\,F \colon \mathbb U \onto \mathbb C\,$ .
 \end{proof}

With Corollary  \,\ref{Weakreal}\, in mind, we introduce the following concept
\begin{definition}
A holomorphic quadratic differential $\varphi (z)\, \dtext z \otimes \dtext z$, with  $\varphi \in \mathscr L^{1} (\X)\,$, is  real along a smooth  arc  $\,\gamma\subset \partial \X\,$, \textit{in a weak sense}, if

\begin{equation}
    \re \iint _{\mathbb X }  \varphi(z) \, \lambda_{\overline{z}} \,\,\textnormal{d} z\; = \;0
    \end{equation}
    for all test functions $\,\lambda \in \mathbb C^k_0 (\mathbb U)\,$ whose linear differential $\,\lambda \,\textnormal{d} \bar{z}\,$ is real on $\,\gamma\,$.
\end{definition}
This definition is not more general than the classical one, but is convenient when dealing with critical points of inner variations along $\,\gamma\,$. Indeed, it turns out that $\,\varphi\,$ admits continuous extension up to $\,\mathbb X \cup \gamma\,$ and as such fulfils the classical condition:
 \begin{equation}
 \varphi(z(t)) \dot{z}(t)^2  \,\in \mathbb R\;
 \end{equation}
 for some (also for all) smooth parametrization of $\,\gamma\,$, $\,\gamma = \{ z(t)\colon  a<t<b\;\} $.

In full generality one might say that:
\begin{definition}\label{RealDiff}
A holomorphic quadratic differential $\varphi (x)\, \dtext z \otimes \dtext z\,$, with  $\,\varphi \in \mathscr L^{1} (\X)\,$, is said to be  real along $\partial \X$ if  upon conformal transformation of $\X$ onto a $\mathscr C^1$-smooth domain the differential becomes continuous and real along its boundary.
\end{definition}
The definition is legitimate, because the choice of conformal transformation plays no role.

 \subsection{Quadrilateral mappings}
In quasiconformal geometry  the term quadrilateral  is used to describe a simply connected Jordan domain $\, \mathbb Q \subset \mathbb C\,$ together with the ordered set of four  points $\, a,b,c,d\,\in \partial\, \mathbb Q\,$ listed  counterclockwise along the boundary and called \textit{corners} of $\,\mathbb Q\,$. The corners disconnect $\,\partial \,\mathbb Q\,$  into four open Jordan arcs;  two of them, having no common endpoint are called \textit{horizontal sides} and the other two  \textit{vertical sides}.
The exemplary quadrilaterals, from which other quadrilaterals are derived via conformal transformations, are the rectangles. But we shall choose annular sectors as representatives of quadrilaterals.
\subsubsection{Annular Sectors}

Given  $\,0\leqslant \alpha < \beta \leqslant  2\pi\,$ and $\, 0<r <R\,$, the annular sector
$$
 \mathbb S = \mathbb S_\alpha^\beta(r,\, R) \,\deff \,\{ z\,; \,r < |z| < R \,,\;\; \alpha < \arg\,z  < \beta \,\}
$$
is viewed as a quadrilateral whose corners listed counterclockwise are $\, a= r e^{i \,\alpha}\,,\, b =  R e^{i \,\alpha}\,,\, c = R e^{i \,\beta}\,,\, d = r e^{i \,\beta} \,$. The circular arcs in
 $\,\partial\, \mathbb S\,$ are viewed as vertical sides and the rays  as horizontal sides. \\

  \subsubsection{Quadrilateral Maps}
  \begin{definition}
   A homeomorphism $\, h \colon \mathbb X \onto\mathbb Y \,$ between two quadrilaterals which has well defined limits at the corners of $\,\mathbb X \,$ and that these limits coincide with  the corners of $\,\mathbb Y\,$ (in the respective order) is called  \textit{quadrilateral map}.
  \end{definition}
The reader is cautioned that the inverse homeomorphism $\, h^{-1} \colon \,\mathbb Y \onto\mathbb X \,$ may not be a quadrilateral map. It should be noted, however, that the cluster set of any side of  $\, \mathbb X\,$  under a quadrilateral map $\, h \colon \mathbb X \onto\mathbb Y \,$ is exactly the corresponding side of  $\,\mathbb Y \,$. We have  a quadrilateral conformal transformation $\,z \rightsquigarrow r e^{i\,\alpha + \beta\,z - \alpha\,z}\,$ of a rectangle $\,\mathcal R = (0,\,L)\times (0,\,1)\,$  onto the annular sector $\, \mathbb S = \mathbb S_\alpha^\beta(r,\, R)\,$, $\,L = \frac{1}{\beta-\alpha} \log \frac{R}{r}\,$.\\
Minimization of the Dirichlet energy among quadrilateral mappings $\, h\colon  \mathbb S \onto \mathbb Q\,$ calls for inner variations $\,\Psi^\varepsilon \colon \mathbb S \onto \mathbb S\,$ that keep the corners of the sector fixed, while allowing slipping along the sides. Call them \textit{quadrilateral variations}. With a few additional tricks the same inner variations as are used in  Proposition\, \ref{HopfAnnulus} show that

\begin{proposition}\label{HopfAnnulus}
    The Hopf-Laplace equation for a quadrilateral mapping  $\, h \in \mathscr W^{1,2}(\mathbb S , \mathbb C )\,$ that is  critical for all quadrilateral variations in $\,\mathbb S\,$ takes the form
  \begin{equation}\label{HLeq}
  \, \phi(z) \,=\, h_z\,\overline{h_{\overline{z}}}\, = \frac{c}{z^2} \;\,,\;\;\;\; \textnormal{for all} \;\;\;\;z \in \mathbb S
  \end{equation}
  where $\, c \,$ is a real number
  \end{proposition}
  \begin{proof}
  First observe that the power conformal map $\, \chi(z) = e^{ i \alpha} \,z ^{\frac{\pi - \alpha}{\beta }}\, $ takes the sector  $\,\mathbb S_\alpha^\beta(r,\, R)\,$ onto an upper half annulus $\,\mathbb S_0^\pi(r',\, R')\,$, and it is a quadrilateral map. Another useful observation is that the Hopf-Laplace equation (\ref{HLeq})\, is invariant under power type transformations of sectors. It is therefore enough to proof (\ref{HLeq}) for an upper half annulus. Quadrilateral variations that we are going to explore this time are:

  $$\Psi^\varepsilon (z)  = \; z \cdot  \frac{1 + \varepsilon ( \bar{z}^k\;-\; z^k )}{ [\,1 - \varepsilon^2 ( \bar{z}^k\;-\;  z^k )^2 \,]^{1/2}}\;, \;\;\; k =  \pm 1, \pm 2, ...$$
  so
  $$\lambda = \frac{\textnormal{d}\Phi^\varepsilon}{\textnormal{d}\varepsilon}\;\Big{|}_{\varepsilon=0} \,=\,\,z \,\bar{z} ^k \;\;, \;\;\; \lambda_{\bar{z}} \,=  k \,z \,\bar{z}^{k-1}\,$$
  Curiously, the counterpart of $\,k = 0\,$ is
  $$\Psi^\varepsilon (z)  = \; z \cdot  \frac{1 + \,\varepsilon\, \textnormal{Arg }z}{ [\,1+ \varepsilon^2 \textnormal{Arg }^2  z  \,]^{1/2}}\,,\;\;\lambda = \frac{\textnormal{d}\Phi^\varepsilon}{\textnormal{d}\varepsilon}\;\Big{|}_{\varepsilon=0} \,=\,\,z \textnormal{Arg}\,z \;\;, \;\;\; \lambda_{\bar{z}} \,= \,i \frac{z}{\bar{z}}$$
  though we can dispense with this variation.
  But we need a Laurent series expansion
 \begin{equation}\label{Laurent}
 h_z\,\overline{h_{\overline{z}}}\, = \phi(z) =  \sum_{n = - \infty} ^ \infty  a_n \,z^n\;\;,\;\;\; a_n = \overline{a_n}
  \end{equation}
  The natural way to obtain such formula is by extending $\,\varphi\,$  to the whole annulus $\,\mathbb A = \mathbb S_+ \cup \mathbb S_- \,$, where $\,\mathbb S_+  = \mathbb S\,$ and $\,\mathbb S_-\,$ is a reflection of $\,\mathbb S\,$ to the lower half plane. We extend $\,\phi\,$  by setting  $ \phi(z) = \overline{\phi(\bar{z})}\,$ for $\im z < 0\,$. It is immaterial how $\,\phi\,$ is defined for $\im z =0\,$.
Clearly, the extended function, still denoted by $\,\phi \,$,  lies in $\,\mathscr L^1 (\mathbb A)\, $. Precisely,  $\,\|\phi\|_{\mathscr L^1(\mathbb A)} = 2 \,\|\phi\|_{\mathscr L^1(\mathbb S)}\,$.
Let $\,\eta \in \mathscr C^\infty_0(\mathbb A)\,$ be any test function,  to be used for verifying that the distributional Cauchy-Rieman equations in $\,\mathbb A\,$ are satisfied. We can write:
\begin{equation}\nonumber
\begin{split}
\re \iint _{\mathbb A }  \phi(z)  \,\eta_{\overline{z}}\;&\,\textnormal{d}z =
   \re \iint _{\mathbb S_+ }  \phi(z) \, \eta_{\overline{z}}\;\,\textnormal{d} z \;+\; \re \iint _{\mathbb S_- }  \phi(z)  \,\eta_{\overline{z}}\;\,\textnormal{d} z\\& =
   \re \iint _{\mathbb S_+ }  \phi(z) \, \eta_{\overline{z}}(z)\;\,\textnormal{d} z \;+\; \re \iint _{\mathbb S_+ }  \phi(\overline{z})  \,\eta_{\overline{z}}(\overline{z})\;\,\textnormal{d} z\\& =
   \re \iint _{\mathbb S_+ }  \phi(z)  \,\eta_{\overline{z}}(z)\;\,\textnormal{d} z \;+\; \re \iint _{\mathbb S_+ }  \overline{\phi(\overline{z})} \; \overline{\eta_{\overline{z}}(\overline{z})}\;\,\textnormal{d} z\\& =
   \re \iint _{\mathbb S_+ }  \phi(z) \big[ \eta_{\overline{z}}(z)\;+ \; \overline{\eta_{\overline{z}}(\overline{z})}\;\big ]\,\textnormal{d} z\\&
   =\; \re \iint _{\mathbb S }  \phi(z)  \,\lambda _{\overline{z}}(z)\,\textnormal{d} z\;=\;0
\end{split}
\end{equation}
where $\,\lambda (z) = \eta(z) +  \overline{\eta(\overline{z})} \,$ belongs to $\,\mathscr C^\infty_0 (\mathbb S_+)\,$ and is obviously real on $\, \mathbb R\,$. The last equality follows from Lemma \,\ref{ExistVar}\,.  Replacing $\, \eta\,$ by $\, i\, \eta\,$ we conclude that  $\,\iint _{\mathbb A }  \phi(z) \, \eta_{\overline{z}}\;\,\textnormal{d} z = 0\,$, which means that $\,\phi\,$ is analytic in $\,\mathbb S\,$, by Weyl's Lemma.
  Having the Laurent expansion at hand we now test the integral form of the variational equation (\ref{critical})\,,
  $$\re \iint _{\mathbb S }  \overline{h_{\overline{z}}}\, h_z \, \lambda_{\overline{z}} \,\,\textnormal{d} z\; = \;0  \;,\;\;\;\;\textnormal{with } \;\;\;\lambda_{\bar{z}} \,=  k \,z \,\bar{z}^{k-1}\;\;,\;\;\;k =  \pm 1, \pm 2, ...
  $$
  To this effect note that $\,\re \iint _{\mathbb S }  a_n\,z^n \, \lambda_{\overline{z}} \,\,\textnormal{d} z\; = \,a_n\,\re \iint _{\mathbb S }  \,z^n \, \lambda_{\overline{z}} \,\,\textnormal{d} z\;=\;0 \,$, except for the case $\, n = k-2\,$. We just proved that $\,\phi(z) = a_{_{-2}} \,z^{-2}\,$, where $a_{_{-2}} = \overline{a_{_{-2}}}\,$. That $\,a_{_{-2}}\,$ must be real can also be seen by testing the variational equation with $\,\lambda_{\bar{z}} \,= \,i \frac{z}{\bar{z}}$.

  \end{proof}

\subsection{Vertical and horizontal arcs of a quadratic differential} Let $\,\varphi (z)\, \dtext z \otimes \dtext z\,$ be a holomorphic quadratic differential in  $\,\mathbb X\,$.  A \textit{vertical arc} is a $\mathscr C^\infty$-smooth curve $\,\gamma = \gamma (t)$, $a<t<b$, along which
\[[\dot{\gamma} (t)]^2 \varphi \big(\gamma (t)\big)<0, \quad a<t<b.\]
A {\it vertical trajectory} of $\varphi$ in $\X$ is a maximal vertical arc; that is, not properly contained in any other vertical arc. In exactly similar way are defined the \textit{horizontal arcs} and \textit{horizontal trajectories}, via the opposite inequality.

 If $\X$ is a circular annulus $\A= A(r,R)$ and $\,\varphi (z)\, \dtext z \otimes \dtext z\,$ is real along its entire boundary then $\,\varphi(z) = \, c z^{-2}\,$, for some $\,c\in \mathbb R\,$. For $ \, c > 0\,$ the concentric circles $\,\mathcal C_\rho = \{\,\rho\, e^{i\theta} \colon 0 \leqslant  \theta < 2\pi\}\,$, $\,\rho \in [r,R]\,$, are the vertical trajectories, whereas  the rays $\,\mathcal R^\theta= \{\,\rho \,e^{i\theta} \colon r < \rho < R\}\,$, $\,\theta \in [0, 2\pi)\,$, are the horizontal trajectories. For a negative $\,c\,$, this holds in reverse order.
 These two cases  exhibit different behavior in regard to the formation of cracks.

 Every continuous mapping $\,h \colon \X \to \R^2\,$ has well defined  multiplicity function which is measurable so one can speak of the essential supremum. We are concerned with mappings of bounded multiplicity, meaning that there is $\,1 \leqslant M  < \infty\,$ so that
\begin{equation}\label{muless} \,\# \{\,x \in \mathbb X \,; \;h(x) = y \} \leqslant M \,,\,\; \textnormal{for almost every }\; y\,\in  \R^2 .\end{equation}
Indeed, for $\,h\in \mathscr H^2_{\textnormal{lim}}(\mathbb \X, \mathbb Y)\,$ we have

$$\,\# \{\,x \in \mathbb X \,; \;h(x) = y \} \leqslant 1 \,,\,\; \textnormal{for almost every }\; y\,\in  \R^2 $$
see ~\cite[Lemma 3.8]{IKKO}.
The following Proposition is from~\cite[Proposition 5.1]{CIKO}.

\begin{proposition}\label{lily}
Suppose $\,h \in \mathscr C(\mathbb X,\mathbb C) \cap \W^{1,1}_{\loc} (\X ,\mathbb C )$ satisfies the condition (\ref{muless}) and solves the Hopf-Laplace equation
\[h_z \overline{h_{\bar z}}= \varphi\;, \qquad \textnormal{\textit{ where}}\;\, \,\varphi \not\equiv 0\, \textnormal{ \textit{is a holomorphic function in}}\;\; \X \subset \C           \]
Then for each $y_\circ \in \R^2$ the union of all vertical trajectories of the quadratic differential $\varphi (z)\, \dtext z \otimes \dtext z\,$ that intersect $h^{-1} (y_\circ)$ has zero measure.
\end{proposition}

 \subsection{Bizarre solutions}\label{bizarre}  Complex harmonic functions in the Sobolev space $\,\mathscr W^{1,2}(\mathbb X\,,\mathbb C)\,$ are among the solutions of the Hopf-Laplace equation. But there are many more, sometimes surreal ones. Harmonicity is lost exactly at the points where $\,h\,$ fails to be injective~\cite{IKOhopf}. Consider, for example, an origami folding of a square sheet of paper, denoted  by $\,\mathbb X\,$. Upon a countable number of suitable folds all the points in $\,\partial \mathbb X\,$ stock up into one point, say the origin. In this way one obtains a piece-wise orthogonal Lipschitz map $\,h \colon \mathbb X \onto \mathbb Y\,$ which vanishes on $\,\partial \mathbb X\,$. In fact the range of the differential $\,Dh\,$ may consist  of a finite number (eight is enough) of orthogonal matrices. Therefore, at almost every point $\,z \in \mathbb X\,$ we either have $\,h_z = 0\,$ or $\,h_{\bar{z}} = 0\,$. The interested reader is referred to \cite{IVV, Ki, KS} for more details. Thus we have a nonzero  solution of the homogeneous Dirichlet problem
 \begin{displaymath}
 \left\{\begin{array}{ll}\,h_z\,\overline{h_{\overline{z}}}\, = 0 \,\,,\;\;\;\; h\in  Lip \,(\mathbb X, \mathbb Y)&\\
 h\,  \equiv \,0 \,\,,\;\;\;\;\;\; \textnormal{on}\;\;\; \partial \mathbb X
 \end{array} \right.
 \end{displaymath}
 There are, perhaps, even more exotic solutions to the homogeneous Hopf-Laplace equation, $\,\,h_z\,\overline{h_{\overline{z}}}\, = 0\,$, interesting not only in their own right.

 \subsection{The natural domain of  definition and admissible solutions}

 It becomes clear that in order to build a viable theory of the  equation (\ref{critical}) one must assume that the Jacobian determinant is nonnegative almost everywhere. Such reflections urge the following concept.
 \begin{definition} The class $\,\mathscr W^{1,2}_+(\mathbb X, \mathbb C)\,$ of mappings in $\, \W^{1,2}_{\loc} (\X, \C)$ whose Jacobian $\,J(x,h) \geqslant 0\,$ almost everywhere in $\,\mathbb X\,$ will be regarded as the \textit{natural domain of definition} of the Hopf-Laplace operator. A map $\,h \in \mathscr W^{1,2}_+(\mathbb X, \mathbb C)\,$ which satisfies the equation $\frac{\partial}{\partial \bar z} \left(h_z \overline{h_{\bar z}}\right)= 0\,$\, in the sense of distributions, will be referred to as  \textit{admissible solution}.
 \end{definition}
 Critical points for the Dirichlet energy within any of the classes in Definition \ref{spaces}  comply with these requirements.
It should be remarked that all the admissible solutions to (\ref{critical}) are locally Lipschitz continuous ~\cite{CIKO} and ~\cite{IKOli}, but not $\mathscr C^1$-smooth in general. See also Theorem  \ref{globLip} which, in a specific case,  asserts Lipschitz regularity  up to the boundary.

\section{Harmonic Replacement\label{hrmRep}} The main idea in the proof of Theorem \,\ref{nocracks} is to show that  every $\,h \in {\mathscr{H}}^{1,2}_{\lim}(\mathbb \X, \mathbb Y) = \overline{\mathscr H}_2(\mathbb X, \mathbb Y) \,$ can be modified in a neighborhood of the set  $\,h^{-1}(a) \subset \mathbb X\,$, so-called a boundary cell,  with  lower energy, unless $\,h\,$  already takes the boundary cell diffeomorphically into $\,\mathbb Y\,$. To this effect we need  a thorough analysis of harmonic replacement procedure.\\

Let $\,h \in {\mathscr{H}}^{1,2}_{\lim}(\mathbb \X, \mathbb Y)\,$ and $\,B \subset \mathbb C\,$ be a ball. Assume that the set $\, \mathbb D = B \cap \mathbb Y\,$ is nonempty and \textbf{convex}. Then, because of monotonicity,  the set $\,\mathbb U = \{ x\in \mathbb X \,; \; h(x) \in B \cap \overline{\mathbb Y}\,\}\,$, is an open simply connected domain, called a \textit{cell} in $\,\mathbb X\,$.
More specifically, we call $\,\mathbb U\,$ an \textit{inner cell} if  $\,B \subset  \mathbb Y\,$ and a \textit{boundary cell} otherwise.
\begin{proposition}\label{replacement}
Let $\Y$ be a Lipschitz domain and $\,h \in {\mathscr{H}}^{1,2}_{\lim}(\mathbb \X, \mathbb Y)\,$. Then to every cell $\,\mathbb U \subset \mathbb X\,$   there corresponds (unique) mapping $\,h_{_\mathbb U} \in  \overline{\mathscr{H}}_2({\mathbb \X}, {\mathbb Y})\,$ such that
\begin{enumerate}[(i)]
\item\label{en1} $\,h_{_\mathbb U} = h \colon {\mathbb X} \setminus  \mathbb U \onto \overline{\mathbb Y} \setminus B $
\item\label{en2} $\,h_{_\mathbb U}\colon \mathbb U \onto\,\mathbb D \,$ is a harmonic diffeomorphism.
\item\label{en3} $
\mathscr E_\mathbb X [h_{_\mathbb U}]  \leqslant \mathscr E_\mathbb X [h]  $
\end{enumerate}
Equality in (\textit{iii}) occurs if and only if $h \equiv h_{_\U}$.
We refer to $\,h_{_\mathbb U}\,$ as \textbf{harmonic replacement} of $\,h\,$ in  $\,\mathbb U\,$.
\end{proposition}
Before passing to the proof let us make a few useful observations. First,  Proposition \ref{replacement} is not affected by a conformal change of variables within  the domain $\,\X\,$. Thus it involves no loss of generality in assuming that $\,\X\,$ is a Schottky domain (bounded by circles), so  Lipschitz regular.
 Under this assumption, equalities in~\eqref{setiden} apply.  In particular, we dispose of a sequence of homeomorphisms $h_j \colon \overline{\X} \onto \overline{\Y}$ which converge to $\,h\,$ uniformly on $\overline{\X}$ and strongly in $\W^{1,2} (\X, \Y)$. The limit map $h$ extends continuously up to the closures, which we continue to denote by $\, h\colon \overline{\mathbb X} \onto\overline{\mathbb Y}\,$. The extended map $\, h\colon \overline{\mathbb X} \onto\overline{\mathbb Y}\,$ and the boundary map $\, h\colon \partial\mathbb X \onto\partial\mathbb Y\,$ are  monotone.
\begin{proof} The case of an inner cell is particularly convenient to begin with. Later, with additional topological arguments, we will be able to exploit this proof for the boundary cells as well.

{\bf Case 1. The inner cell.}  Suppose $B \subset \Y$ so $\mathbb U = h^{-1} (B) \subset \X$. Since $h$ is monotone, it follows that  $\mathbb U$ is a simply connected subdomain of $\mathbb X$. Consider an increasing sequence $B_1 \subset B_2 \subset \dots$, $\cup B_n =B$, of balls and the corresponding cells $\U_n = h^{-1} (B_n)$
\begin{equation}\label{eqseq3s}
\U_{n-1} \subset \overline{\U}_{n-1} \subset \U_n \subset \overline{\U}_n \subset \U_{n+1}
\end{equation}
and note that
\[\partial \U_n \subset h^{-1} (\partial B_n) \deff  \Gamma_n \, .\]
Here $\Gamma_n$ is a continuum disconnecting $\C$ into two components; the bounded component equals $\U_n$. Thus $\overline{\U}_{n-1}$ is a compact subset of the inner complement of  $\, \Gamma_n\,$. We look at the induced cells $\,\U_n^j = h_j^{-1} (B_n)\,$, $\,j=1,2, \dots\,$. Their boundaries  $\Gamma_n^j \deff \partial \U_n^j =h_j^{-1} (\partial B_n)$ are closed Jordan curves. Since $\,h_j\,$ converge uniformly to  $\,h\,$, it follows that
\[\lim_{j \to \infty} \;\sup \,\{\dist (x, \Gamma_n)\,;\;x \in \Gamma_n^j\}  =0\]
which in turn implies that for sufficiently large $j$, say $j \ge j_n$, we have $\Gamma_n^j \subset \U_{n+1}$ and $\overline{\U}_{n-1}$ is contained in $\U_n^j\,$ -the bounded component of $\C \setminus \Gamma_n^j$. We now appeal to Rad\'o-Kneser-Choquet theorem, see Theorem~\ref{RaKnCh}. Accordingly, we may extend the boundary homeomorphism $h_{j_n} \colon \Gamma_n^{j_n} \to \partial B_n$ inside the cell $\U_n^{j_n}$ to obtain a homeomorphism $\widetilde{h}_{j_n} \colon \overline{\X} \onto \overline{\Y}$ such that
\begin{itemize}
\item $\widetilde{h}_{j_n} = h_{j_n} \colon \overline{\X} \setminus \U^{j_n}_{n} \onto \overline{\Y} \setminus B_n$
\item $\widetilde{h}_{j_n}  \colon  \U^{j_n}_{n} \onto B_n$ is a harmonic diffeomorphism
\item $\mathscr E_\mathbb X [\widetilde{h}_{j_n}]  \leqslant \mathscr E_\mathbb X [h_{j_n}] \le M \, $, \;this is a bound independent of $n$.
\end{itemize}
Passing to a subsequence of $\{\widetilde{h}_{j_n}\}$, if necessary, we define $\widetilde{h} \in {\Ho}^{1,2}_{\lim} (\X, \Y)$ to be a weak limit of $\{\widetilde{h}_{j_n}\}$. This sequence also converges uniformly  to $\widetilde{h} \colon \X \onto \Y\,$, because of the uniform bounds in \, (\ref{r2}).
Moreover,
\[\widetilde{h}=h \colon \X \setminus \U \onto \Y \setminus B \, .\]
Let us demonstrate that
\begin{equation}
\widetilde{h} \colon \U \onto B \qquad \textnormal{is a harmonic diffeomorphism.}
\end{equation}

To this effect, fix any sequence $\{\mathbb K_n\}^\infty_{n=1}$ of  subdomains compactly contained in $\,\U_n\,$, such that $\cup \mathbb K_n = \U$.  Since  $\widetilde{h}_{j_k} \colon \U_n \to \C$, $\,k\geqslant n \,,$ converge uniformly to $\widetilde{h} \colon \mathbb K_n \to \C$ we shall infer that either $\widetilde{h} \colon \mathbb K_n \to \C$ is a diffeomorphism or $J(x, \widetilde{h}) \equiv 0$ in $\mathbb K_n$, at least for large values of $\,n\,$.  This fact is none other than a harmonic variant of the classical Hurwitz theorem for analytic functions.

\begin{lemma}[Hurwitz Lemma] \label{Hurwitz}
If a sequence of harmonic homeomorphisms $\,f_k \colon \Omega \rightarrow \mathbb C\,$ converges $\,c$-uniformly to $\,f \colon \Omega \rightarrow \mathbb C\,$ then either $\,f\,$ is a harmonic homeomorphisms (actually $\,\mathscr C^\infty$-diffeomorphism) or its Jacobian determinant vanishes identilcally, $\,J(x, f) \equiv 0\,$.
\end{lemma}

The proof of this lemma presents no difficulty; simply, consider a sequence $\,\{\frac{\partial f_k}{\partial z} \}\,$ of analytic functions.  Curiously, it also holds for  $\,p$-harmonic mappings, but in this case the  proof requires much more work, see  \cite{IOa}.

Returning to our demonstration, the equality $\,J(x, \widetilde{h}) \equiv 0\,$ in $\,\mathbb K_n\,$ is ruled out by the following computation
\[
\begin{split}
\iint_{\U} J(x, \widetilde{h})\, \dtext x & = \lim_{k\to \infty} \iint_{\U} J(x, \widetilde{h}_{j_k})\, \dtext x \ge \lim_{k \to \infty}  \iint_{\U_k^{j_k}} J(x, \widetilde{h}_{j_k})\, \dtext x \\&  = \lim_{k \to \infty} \abs{ \widetilde{h}_{j_k} (\U_k^{j_k}) } = \lim_{k \to \infty} \abs{B_{j_k}}= \abs{B}>0\, .
\end{split}
\]
Here, in the first equation, we appealed to the $\,\mathscr L^1_{\textnormal{loc}}$ - weak convergence of nonnegative Jacobians, a property discovered by S. M\"{u}ller  \cite{Muller}, also  see ~\cite[Theorem 8.4.2]{IMb}.
Thus, for  large $\,n\,$ we have $\iint_{\mathbb K_n} J(x, \widetilde{h})\, \dtext x >0$ and  so $J(x, \widetilde{h}) \not \equiv 0$ in $\,\mathbb K_n\,$. We infer that $\,\widetilde{h}$ is a local diffeomorphism in $\,\U$. Since $\widetilde{h}$ is  a $c$-uniform limit of homeomorphisms, the map  $\widetilde{h} \colon \U \onto B\,$ must be a global homeomorphism. Finally, we apply the Dirichlet principle which asserts that if in a bounded simply connected domain $\,\mathbb U\,$  two functions $\, h , \widetilde{h} \in \mathscr C(\overline{\mathbb U}) \cap \mathscr W^{1,2}(\mathbb U)\,$ coincide on $\,\partial \mathbb U\,$, then

\[\mathscr E_{\U} [\widetilde{h}] \le \mathscr E_{\U} [h] \, \;,\;\;\;\;\;\;\textnormal{whenever} \;\;\widetilde{h} \;\;\textnormal{is harmonic} \]
Equality occurs if and only if $\,h = \widetilde{h}\,$ on $\U$. The proof of Case 1 is complete.
\begin{remark}\label{reminner}
In the above proof we actually did not use any regularity of $\X$ or $\Y$. Thus Proposition~\ref{replacement} holds for arbitrary domains $\,\X\,$ and $\,\Y\,$,  provided $\, B \subset \mathbb Y\,$ ; that is , $\,\U=f^{-1} (B)\,$ is an inner cell. It is the case of the boundary cells that we really benefit from  $\,\Y\,$ being a  Lipschitz domain.
\end{remark}

{\bf Case 2. The boundary cell.} Apart from a few  adjustments, necessitated by some geometric complications near the boundary of $\,\mathbb X\,$, the proof differs very little from that in Case 1. The idea is to extend the mapping $\,h \colon \overline{\mathbb X} \onto \overline{\mathbb Y}\,$ beyond $\,\overline{\mathbb X}\,$ an $\,\overline{\mathbb Y}\,$ so the boundary cell of $\,h\,$  will be enlarged accordingly to become an inner cell of the extended mapping.
In all that follows $\,\mathbb X\,$ will be a Schottky domain (i.e. bounded by circles) and $\,\mathbb Y\,$ a Lipschitz domain.
\begin{lemma}[Extension Lemma] \label{extension}
There exist neighborhoods $\,\mathbb X_+ \supset \overline{\mathbb X}\,$ and $\,\mathbb Y_+ \supset \overline{\mathbb Y}\,$ and an extension operator
$$\;\, \widehat{\,} \colon \overline{\mathscr{H}}_2(\mathbb \X, \mathbb Y)\,\rightarrow \overline{\mathscr{H}}_2(\mathbb \X_+, \mathbb Y_+)\, $$
so that for every $\, h \in \overline{\mathscr{H}}_2(\mathbb \X, \mathbb Y)\,$ its extension $\,\widehat{h} \in \overline{\mathscr{H}}_2(\mathbb \X_+, \mathbb Y_+)\,$  satisfies :
\begin{itemize}
\item[(i)]  $\widehat{h}(x) \,=\, h(x)\,,\;\;\;\textnormal{for}\;\;\; x\in \overline{\mathbb X }$
\item[(ii)] $ \widehat{h} \colon  \mathbb X_+ \setminus \overline{\mathbb X} \onto \mathbb Y_+ \setminus \overline{\mathbb Y} $
\item[(iii)] $$ \iint_{\mathbb X_+} |D\widehat{h}(x)|^2 \,\textnormal{d}x \;\;\leqslant \;\; C_{_{\mathbb X \mathbb Y}}\iint_{\mathbb X} |D{h}(x)|^2 \,\textnormal{d}x   $$
\item[(iv)] Let a sequence of homeomorphisms $\,h_j \colon  {\mathbb X} \onto {\mathbb Y}\,$ converge  to $\,h\,$ uniformly and strongly in $\,\mathscr W^{1,2}(\mathbb X\,,\mathbb Y)\,$. Then the extended mappings $\,\widehat{h}_j \colon  {\mathbb X_+} \onto {\mathbb Y_+}\,$ are also homeomorphisms; they  converge to $\,\widehat{h}\,$ uniformly and strongly in $\,\mathscr W^{1,2}(\mathbb X_+\,,\mathbb Y_+)\,$.
\end{itemize}
\end{lemma}
\begin{proof} A brief sketch of such extension is the following. First, with the aid of a bi-Lipschitz automorphism of $\,\mathbb C\,$ we transform the target domain $\,\mathbb Y\,$ onto a circular domain, say $\,\mathbb Y'\,$. There is no need for such a transformation of $\,\mathbb X\,$ because it is  already a circular domain, record this assumption as  $\,\mathbb X' = \mathbb X\,$.
Next we reflect $\,\X'\,$ and $\,\Y'\,$ about their boundary circles and add the reflected images  to $\,\overline{\X'}\,$ and $\overline{\Y'}\,$, respectively. We obtain the neighborhoods $\,\X'_+\supset \overline{\mathbb X'}$ and $\,\Y'_+\supset \overline{\mathbb Y}\,$. Then we return, via inverse bi-Lipschitz automorphism,  to the
desired neighborhoods  $\,\X_+\supset \overline{\mathbb X}$ and $\,\Y_+\supset \overline{\mathbb Y}\,$. These same reflections give rise to the extensions  of the mappings $\,h_j\,$, denoted by   $\,{\widehat{h}_j} \colon \X_+ \onto \Y_+ \,$, and of the limit map $\,h\,$, denoted by   $\,{\widehat{h}} \colon \X_+ \onto \Y_+ \,$.  The interested reader is referred to~\cite{IOa} for details, especially to rather unexpected nuance about continuity of the composition of a Sobolev mapping  with  bi-Lipschitz transformation of the target domain. \end{proof}

Armed with Lemma (\ref{extension}) we are now able to repeat most of the arguments in Case 1. Recall the boundary cell $\,\mathbb U = \{ x \in \mathbb X ;\;\; h(x) \in B \cap \overline{\mathbb Y}\,\}\,$  for $\,h\,$. Here we may, and do, assume that the ball $\,B\,$ lies in $\,\mathbb X_+\,$, so the set $\{ x \in \mathbb X ;\;\; \widehat{h(}x) \in B \,\}$ is an inner cell for $\,\widehat{h}\,$. For convenience, we continue to use  the same notation   for the induced cells as in Case 1; namely,  $\,\mathbb U_n = \widehat{h}^{-1}(B_n)\,$,  $\,\mathbb U^j_n = \widehat{h_j}^{-1}(B_n)\,$ and $\,\Gamma ^j_n  = \partial \mathbb U^j_n\,$. These are subsets of $\,\mathbb X_+\,$. Fix any sequence $\{\mathbb K_n\}$ of domains compactly contained in $\U_n \cap \X\,$, such that $\,\bigcup_{n=1}^\infty \mathbb K_n = \U \cap \X$. However, this time  harmonic extensions of the boundary homeomorphisms $\, h_j \colon  \Gamma _n^j \onto \partial  B_n \,$ are useless.  They do not take the arcs $\,\mathbb U_n^j \cap \partial \mathbb X \,$ onto the corresponding arc $\,B_n\cap \partial \mathbb Y\,$; even if they do, such harmonic extensions  may not coincide with $\, h\,$ on these arcs. That is why we should perform harmonic replacements only on a portion of $\,\mathbb U_n^j\,$; precisely, on the set $\,\mathbb U_n^j \cap \mathbb X\,$. It is important to observe that $\,\mathbb U_n^j \cap \mathbb X\,$ is a simply connected Jordan domain.  To see this,  consider the open Jordan arc $\,\gamma = B_n\cap \partial \mathbb Y\,$, and its homeomorphic preimage $\,\beta = h^{-1}_j(\gamma)\,\subset \mathbb U_n^j \cap \partial \mathbb X\,$. The endpoints of the arc $\,\beta\,$ lie in the boundary of $\,\mathbb U_n^j\,$, as is easy to demonstrate. Thus $\,\beta\,$ is a crosscut of the simply connected domain $\,\mathbb U_n^j\,$. This is surely an elementary topological fact that such a crosscut $\beta\,$ splits $\, \mathbb U_n^j \,$  into two simply connected subdomains. The subdomain  of interest to us is exactly $\,\mathbb U_n^j \cap \mathbb X\,$. Now, upon harmonic replacements of $\,h_j\,$ in  $ \mathbb U_n^j \cap \X\,$, the proof of Proposition \ref{replacement} runs as in Case 1 with hardly any changes.
\end{proof}

\section{Proof of Theorem~\ref{nocracks}}
The case when $\X$ and $\Y$ are simply connected is obvious; in this case $\,h\,$ is a conformal mapping. Thus let $\X$ and $\Y$ be multiply connected Jordan domains, $\,\mathbb Y\,$ being a Lipschitz domain. Consider the boundary cell $\U= h^{-1} (B \cap \overline{\Y})$, where $\,a\in B \cap \overline{\Y}\,$, and the harmonic replacement $h_{\U} \in \overline{\Ho}_2 (\X, \Y)\,$ as  in Proposition~\ref{replacement}. Since $\,h\,$ is energy-minimal, we have $\,\mathcal E_\X [h] \le \mathcal E_\X [h_{\U}]\,$. Therefore, by condition~\eqref{en3} we must have equality  $\,\mathcal E_\X [h] = \mathcal E_\X [h_{\U}]\,$ which yields that $\,h \equiv h_\U\,$. The proof concludes by observing that
\[
\begin{split}
& h \colon \X \setminus \U \to \overline{\Y} \setminus B \, , \quad \textnormal{ thus } a \not \in h(\X \setminus \U)\\&
h \colon \U  \onto \mathbb D \, , \quad \textnormal{ thus } a \not \in h(\U)
\end{split}
\]
Therefore, $\,a \not \in h(\X \setminus \U) \;\cup\;  h(\U)\,=\, h(\X)\,$, as claimed. \qed

\section{The class $\Ho^{1,2}_{\lim} (\X, \Y)$\, revisited}
We emphasize that the boundary points of $\Y$ may be in the range of $h\in \,{\mathscr{H}}^{1,2}_{\lim}(\mathbb \X, \mathbb Y)\,$, which is the major source of difficulties.  Nevertheless, by Lemma  \ref{inclusions} we always have
\begin{equation}\label{range}
 \,\mathbb Y \subset h(\X) \subset \overline{\Y}, \quad \textnormal{ whenever } \;\;h \in \,{\mathscr{H}}^{1,2}_{\lim}(\mathbb \X, \mathbb Y)\;.
\end{equation}
\begin{lemma}\label{jacoprop}
For $ \,h\in \,{\mathscr{H}}^{1,2}_{\lim}(\mathbb \X, \mathbb Y)\,$ the Jacobian determinant vanishes almost everywhere in  $\,\X \setminus h^{-1}(\Y) = h^{-1} (\partial \Y) \,= \{\, x \colon h(x) \in \partial \mathbb Y\}\,$.
\end{lemma}
\begin{proof}[Proof of Lemma~\ref{jacoprop}]
Consider the set $\, \mathcal B =  \X \setminus h^{-1}(\Y) = h^{-1} (\partial \Y)\,$ and let $\mathbb K\,$ be any compact subset of $\mathcal B$. We invoke the sequence of homeomorphisms  $\,h_j \colon \mathbb X  \onto \mathbb Y\,$ converging to $\,h\,$ weakly in the Sobolev space $\,W^{1,2}(\mathbb X, \mathbb Y)$.
By $\,\mathscr L^1_{\textnormal{loc}}$ - weak convergence of nonnegative Jacobians, ~\cite[Theorem 8.4.2]{IMb} , we can write
\[ \int_{\mathbb K} J(x,h)\, \dtext x = \lim_{j \to \infty} \int_{\mathbb K} J(x, h_j) (x) \, \dtext x = \lim_{j \to \infty} \abs{h_j (\mathbb K)} .\]
{\bf Claim.} We have $\,h_j (\mathbb K) \subset \mathbb Y_j\,$, where

\[\Y_j = \{y \in \Y \colon \dist (y, \partial \Y) \leqslant \epsilon_j\} \quad \textnormal{and} \quad \epsilon_j = \sup_{k \ge j} \norm{f_k -f}_{\mathscr L^\infty (\mathbb K)} \to 0. \]
Indeed, let $\,y = h_j(x)\in h_j (\mathbb K)\,$ for some $x\in \mathbb K\,$, so $h(x) \in \partial \Y\,$. Therefore,
\[
\begin{split}
\dist (y, \partial \Y) & = \dist (h_j(x), \partial \Y) \le \abs{h_j (x)-h(x)} \\ &\le \norm{h_j-h}_{\mathscr L^\infty (\mathbb K)} \le \sup_{k \ge j}\; \norm {h_j -h}_{\mathscr L^\infty (\mathbb K)} = \epsilon_j .
\end{split}
\]
Hence $y\in \Y_j$, as claimed.\\

Now it follows that
\[\int_{\mathbb K} J_h (x) \, \dtext x \leqslant \lim_{j \to \infty} \abs{\Y_j}\; \rightarrow 0 \]
 because we have a decreasing sequence $\Y_1 \supset \Y_2 \supset \dots \;$ of measurable subsets of $\Y$ whose intersection is empty.   This yields
\[\int_{\mathbb K} J_h (x) \, \dtext x =0.\]
Finaly, $\,\mathcal B\,$ can be expressed as a union  of increasing sequence $\mathbb K' \subset \mathbb K'' \subset \mathbb K''' \subset \dots$ of compact subsets of $\,\mathcal B\,$ and a set $\,\mathbb E \subset \mathcal B\,$ of measure zero. We see that
\[\int_{\mathcal B } J_h = \int_{\mathbb E} J_h + \int_{\mathbb K' \cup \,\mathbb K''\, \cup \dots}   J_h = 0+0 =0.\]
Since $J_h \ge 0$ almost every in $\mathcal B$, we conclude $J_h \equiv 0$ in $\mathcal B$.
\end{proof}

\subsubsection{The $\delta$-function}
Let $\X$  and $\,\mathbb Y\,$ be bounded doubly connected domains. We recall the notation $\X_I$ and $\X_O$ for the bounded (inner) and unbounded (outer) components of $\widehat{\C} \setminus \X$.  Thus  $\partial \X_I$ and $\partial \X_O$ are exactly the boundary components of $\X$, so $\widehat{\C}= \overline{\X}_I \cup \X \cup \overline{\X}_O$. The same notation applies to $\,\mathbb Y\,$. Let  $h \in \Ho^{1,2}_{\lim} (\X, \Y)$. We define a function $\delta \colon \widehat{\C} \onto [0,1]$  by the rule
\begin{equation}\label{deltafun}
\delta(x) = \begin{cases} 0\, ,  & \textnormal{ if } x\in \overline{\X}_I\\
\frac{\dist [h(x)\, , \,\partial \Y_I]}{\dist [h(x) \,,\, \partial \Y_I]\; + \;\dist [h(x)\, , \,\partial \Y_O] }\,\; , & \textnormal{ if } x\in {\X} \\
1 \, , & \textnormal{ if } x\in \overline{\X}_O \, .
\end{cases}
\end{equation}
\begin{lemma}\label{lemdel}
The $\delta$--function  is continuous and monotone.
\end{lemma}
\begin{proof}
The mapping $h \in \Ho^{1,2}_{\lim} (\X, \Y)$ is a weak limit of homeomorphisms $h_j \colon \X \onto \Y$ in the Sobolev space $\W^{1,2} (\X\,, \Y)$. We appeal to~\cite[Theorem 1.1]{IOtr} which provides us with a uniform bound of the distance function of any homeomorphism $\,f \colon \X \onto \Y\,$ in the Sobolev space $\,\mathscr W^{1,2}(\mathbb X \,, \mathbb Y)\,$; namely,
\[\dist \left[f(x)\,, \partial \Y\right] \le \eta_{_{\X \, \Y} } (x) \, \mathscr E_{\X}[f] \quad \textnormal{for all} \;\;\;x \in \widehat{\C}\]
where $\,\,\eta_{_{\X \, \Y}} \colon \widehat{\C} \to [0\,, \infty )$ is (independent of $\,f\,$) a  continuous function  vanishing outside $\,\X\,$. Thus we have well defined continuous monotone functions
\[
\delta_j(x) = \begin{cases} 0\, ,  & \textnormal{ if } x\in \overline{\X}_I\\
\frac{\dist [h_j (x)\,, \,\partial \Y_I]}{\dist [h_j(x)\,,\, \partial \Y_I] \;+ \;\dist [h_j(x)\,,\, \partial \Y_O] } \, \;, & \textnormal{ if } x\in {\X} \\
1 \, , & \textnormal{ if } x\in \overline{\X}_O \, .
\end{cases}
\]
converging uniformly to $\delta (x)$. Therefore, the $\,\delta\,$ -function is also continuous and monotone.
\end{proof}
This implies that the set $\,h^{-1} (\Y)= \{x \in \widehat{\C} \colon 0< \delta (x) < 1\}\,$ is connected and its complement consists of two components \[\widehat{\C} \setminus h^{-1} (\Y)  =\{x \colon \delta(x)=0\} \cup \{x \colon \delta(x)=1\} \, . \]
These are disjoint continua containing $\partial \X_I$ and $\partial \X_O$, respectively. We have established the following fact.
\begin{lemma}\label{double}
Let $\X$ and $\Y$ be bounded doubly connected domains and $h \in \Ho^{1,2}_{\lim} (\X, \Y)$. Then $h^{-1} (\Y)$ is a doubly connected domain separating the boundary components of $\X$.
\end{lemma}

\begin{proposition}\label{Partialharmonicity}
  Every energy-minimal map $\,h\in \overline{\mathscr{H}}_2(\mathbb \X, \mathbb Y)$ is a harmonic diffeomorphism of $h^{-1}(\Y) \subset \X$ onto $\Y$.
\end{proposition}
\begin{proof}
This is immediate from Proposition \ref{replacement} on harmonic replacements. Indeed, we may cover $\mathbb X\,$ by inner cells. In each inner cell $\,\mathbb U\subset \mathbb X\,$ the map $\,h \,$ is a diffeomorphisms, exactly equal to its harmonic replacement, since otherwise we would have a map in $\, \overline{\mathscr{H}}_2(\mathbb \X, \mathbb Y)\,$ with smaller energy. It remains to observe that a local diffeomorphism which is a $\,c$ - uniform limit of homeomorphisms must be a global diffeomorphisms.
\end{proof}

The following existence and uniqueness result is from~\cite[Theorems 2.3 and 2.4]{IKKO}

\begin{theorem}\label{exstmodY}
Let $\X$ and $\Y$ be doubly connected domains so that $\Mod \Y \ge \Mod \X$. Then there exists a $\mathscr C^\infty$-diffeomorphism $\,H \colon \mathbb X \onto\mathbb Y\,$ in $\, \W^{1,2} (\X, \Y)$ that minimizes the Dirichlet energy among all mappings in $\,\mathfrak D (\X, \Y)$. Moreover, such a minimizer is unique up to a conformal change of variables in $\,\X\,$. If, however, $\,\Mod \Y\,$ is sufficiently small relative to $\,\Mod \X\,$, then no $\,\mathscr C^\infty\,$ -diffeomorphism of $\,\mathbb X\,$ onto $\,\mathbb Y\,$ is energy-minimal.
\end{theorem}

\section{Cracks in  doubly connected domains}
Let $\X$ and $\Y$ be bounded doubly connected domains of finite conformal modulus. We shall establish basic geometric features  of the solutions to the Hopf-Laplace equation
\begin{equation}\label{hleq}
h_z \overline{h_{\bar z}} = \varphi \qquad \textnormal{ for } h \in \W^{1,2} (\X , \Y)
\end{equation}
where $\varphi \in \mathscr L^1 (\X)$ is analytic. Without any additional assumptions of topological nature, this equation admits rather bizarre solutions, see \S \ref{bizarre}. Our main interest, however, is in the energy-minimal solutions. From this point of view it is natural to restrict the equation~\eqref{hleq} to the class $\Ho^{1,2}_{\lim} (\X, \Y)$ of weak limits of homeomorphisms $h_j \colon \X \onto \Y$ in the Sobolev space $\W^{1,2} (\X, \Y)$. We also assume that the Hopf differential $h_z \overline{h_{\bar z}} \, \dtext z \otimes \dtext z$ is real and nonvanishing along $\partial \X$. The simplest way to rigorously formulate this property is through a conformal transformation $f \colon \X \onto \mathbb A$ of $\X$ onto an annulus $\mathbb A = \{z \colon r< \abs{z}<R\}$. Upon such transformation the Hopf differential remains real along the boundary of $\mathbb A\,$, by definition. For $\X = \mathbb A$ this reads as
\begin{equation}\label{hleqr}
h_z \overline{h_{\bar z}}= \frac{c}{z^2} \, , \qquad c \in \R \, , \quad c \not=0 \, .
\end{equation}
Thus the rays $\,\RR^\theta=\{\rho e^{i\theta} \colon r< \rho < R\}\,$, $\,0 \le \theta < 2\pi\,,$ and the concentric circles $C_\rho =\{\rho e^{i \theta} \colon 0 \le \theta < 2 \pi\}$, $\,r< \rho <R\,,$ are the orthogonal families of the trajectories of $h_z \overline{h_{\bar z}} \, \dtext z \otimes \dtext z$. Their preimages under $f \colon \X \onto \mathbb A$, also known as isothermal curves, are the trajectories in $\X$. Note a dissonance that the differential \[\frac{c}{z^2}\,  \dtext z \otimes \dtext z\, , \qquad c \in \R \, , \quad c \not=0 \] is positive along $\,\partial \X\,$ if $\,c<0\,$ and negative along $\,\partial \X\,$ if $\,c>0\,$.
\begin{theorem}\label{thmcpo}
Suppose that the differential $\,h_z \overline{h_{\bar z}}\;  \dtext z\,  \otimes  \, \dtext z\,$ for $\,h \in \Ho^{1,2}_{\lim} (\X, \Y)\,$ is negative along the boundary of $\X$. Then $h \colon \X \onto \Y$ is  a harmonic diffeomorphism.
\end{theorem}
\begin{proof}
Applying  a conformal change of variables  we may assume that $\X$ is an annulus $\A=\{z \colon r< \abs{z}<R\}$. Upon such a transformation the Hopf-Laplace equation reads as
\[h_z \overline{h_{\bar z}} = \frac{c}{z^2} \qquad \textnormal{ for some } c>0.\]
First we show that $h$ takes $\A$ onto $\Y$. For this, recall the inclusions $\,\Y \subset h(\A) \subset \overline{\Y}$, see~\eqref{range}. Suppose to the contrary that there is   $x_\circ \in \A$ such that $=y_\circ:= h(x_\circ) \in \partial \Y$. We invoke the $\delta$-function given by~\eqref{deltafun}. Accordingly, $\delta(x_\circ)$ is either equal to $0$ or $1$; say $\delta(x_\circ)=0$. We look at the zero set $\delta^{-1}(0)=\{x \in \widehat{\C} \colon \delta (x)=0\}$. Clearly, it contains $\A_I$ and the point $x_\circ$, $r < \abs{x_\circ} <R$. By Lemma~\ref{lemdel} the function $\delta$ is continuous and monotone, so $\delta^{-1}(0)$ is a continuum in $\widehat{\C}$. The circles $\mathcal C_\tau = \{z \colon \abs{z}=\tau \}$, $\,r< \abs{z}< \abs{x_\circ}\,$, separate $\,x_\circ\,$ from $\,\A_I\,$, thus intersect $\delta^{-1}(0)$. These circles are vertical trajectories of $\,\varphi (z)=  c\,z^{-2}\,$, $\,c>0\,$. We now appeal to  Proposition~\ref{lily} which tells us that the union $\bigcup_{r<\tau< \abs{x_\circ}} \mathcal C_\tau$ has zero measure, a clear contradiction. Thus $h(\A)=\Y\,$. Finally, we apply
Theorem 1.12 in \cite{CIKO} which asserts:
\begin{theorem}[Partial Harmonicity]\label{Partial Harmonicity}
Every deformation $\,h \in \mathfrak D(\mathbb X, \mathbb Y) \,$ between bounded multiply connected domains, which satisfies the Hopf-Laplace eqyation is a harmonic diffeomorphism of $\,h^{-1}(\mathbb Y) \subset \mathbb X\,$ onto $\,\mathbb Y\,$.
\end{theorem}
We then conclude that $h \colon \A \onto \Y$ is a harmonic diffeomorphism.
\end{proof}

The situation is dramatically different if $\,h_z \overline{h_{\bar z}}\, \, \dtext z \otimes \dtext z\,$ is positive along $\partial \X$; cracks along vertical trajectories emanating from the boundary components and ending inside  $\,\X\,$ are likely to occur. A complete description of such cracks is given in Theorem~\ref{thmcracks} below where we assume, without loss of generality, that $\X$ is an annulus. The rays and concentric circles of the annulus are vertical and horizontal trajectories.
\begin{theorem}\label{thmcracks}
Suppose that $h \in \Ho^{1,2}_{\lim} (\A, \Y)$ satisfies the Hopf-Laplace equation
\[h_z \overline{h_{\bar z}} = \frac{c}{z^2} \qquad \mbox{ for some } c<0 \]
in an annulus $\A = \{z \colon r< \abs{z}<R\}$. Then every ray $\RR^\theta =\{\rho e^{i \theta} \colon r< \rho <R\}$, $0 \le \theta < 2 \pi\,,$ consists of three disjoint subintervals (some can be empty)
\[\RR^\theta = \RR^\theta_{I} \cup \RR^\theta_{\mathcal A} \cup \RR^\theta_O\]
where
\[
\begin{split}
\RR^\theta_I & = \{\rho e^{i \theta} \colon r < \rho \le r_\theta\} \qquad \textnormal{ {\it inner ray} (possibly empty) }\\
\RR^\theta_{\mathcal A} & = \{\rho e^{i \theta} \colon r_\theta < \rho < R_\theta\} \qquad \textnormal{ {\it middle ray}}\\
\RR^\theta_O & = \{\rho e^{i \theta} \colon R_\theta \le \rho <R \} \qquad \textnormal{ {\it outer ray} (possibly empty) }
\end{split}
\]
where $r \le r_\theta < R_\theta \le R$. The union of middle intervals, denoted by
\[\mathcal A= \bigcup_{0 \le \theta < 2 \pi } \RR^\theta_{\mathcal A}\]
is a doubly connected domain separating the boundary circles of $\A$. Moreover,
\begin{enumerate}[(i)]
\item $h \colon \mathcal A \onto \Y$ is a harmonic homeomorphism
\item $h \colon \RR^\theta_I \to \partial \Y_I$ is constant (if not empty)
\item $h \colon \RR^\theta_O \to \partial \Y_O$ is constant (if not empty)
\end{enumerate}
\end{theorem}


\begin{center}\begin{figure}[h]
\includegraphics[width=0.8\textwidth]{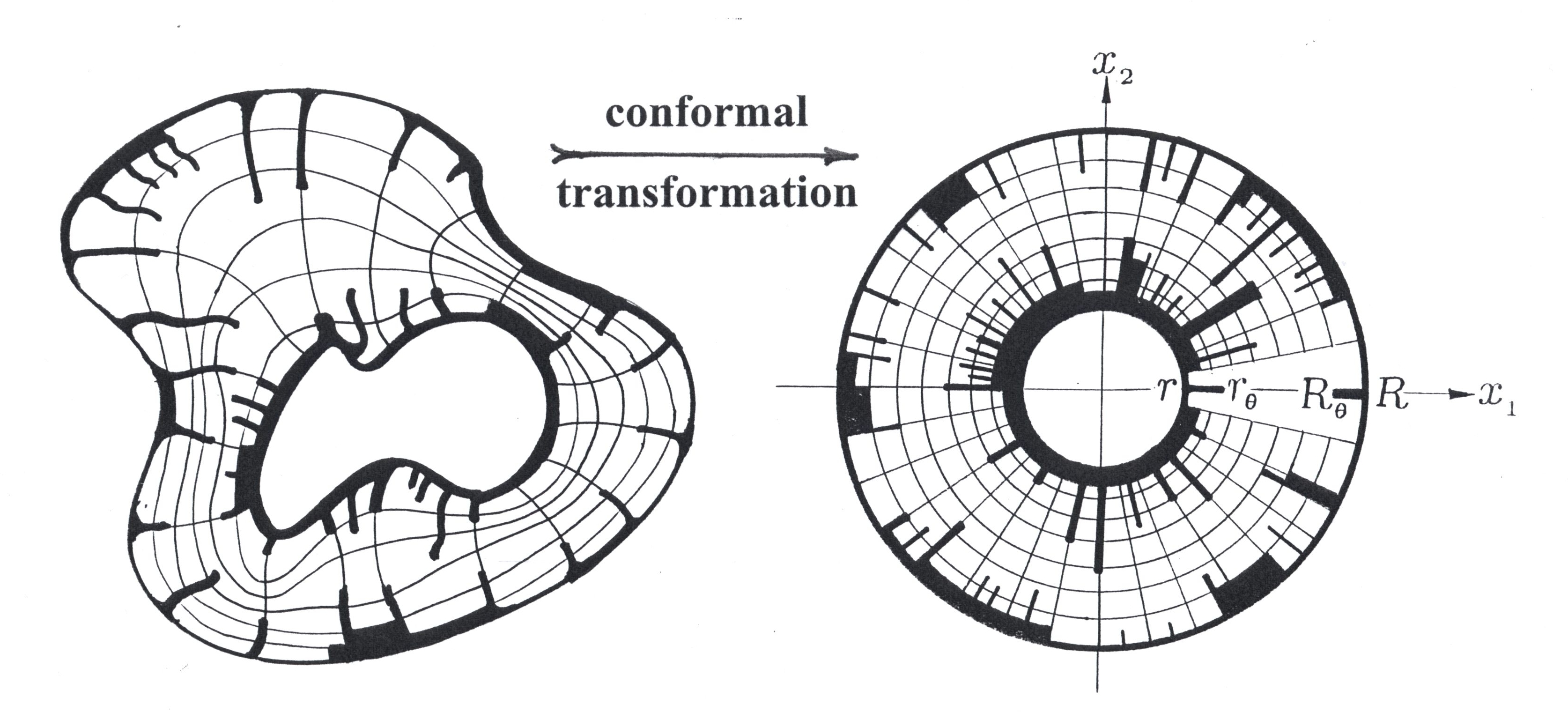} \caption{Cracks in a doubly connected domain and their conformal images (rays) in an annulus.}\label{fig1}
\end{figure}\end{center}

In what follows we denote the doubly connected domain $h^{-1} (\Y)$ by $\mathcal A$. Furthermore, $\mathcal A_I$ stands for the bounded  component of $\C \setminus \mathcal A$ and $\mathcal A_O$ for the unbounded  component of $\C \setminus \mathcal A$. We write $\mathcal B := \A \setminus \mathcal A$, $\mathcal B_r: = \mathcal A_I \cap \mathcal B$ and $\mathcal B_R: = \mathcal A_O \cap \mathcal B$. In view of~\eqref{range} $\mathcal B = h^{-1} (\partial \Y)$. By Lemma~\ref{double} the sets $\mathcal B_r$ and $\mathcal B_R$ are disjoint and relatively closed in $\A$.
\begin{remark}\label{unirem}
Knowing $\,h \colon \mathcal A \onto \Y\,$ will allow us to  determine uniquely the mapping $h \colon \A \to \overline{\Y}$.
\end{remark}
The proof of Theorem~\ref{thmcracks} will employ the following lemmas.

\begin{lemma}\label{lemycircgen}
If $h \colon \A \to \overline{\Y}$ belongs to $\Ho^{1,2}_{\lim} (\A, \Y)$ and $y_\circ \in \Y_I$  then the set $h^{-1} (y_\circ) \cup \A_I$  is connected. Similarly, if $y_\circ \in \Y_O$, then the set $h^{-1} (y_\circ) \cup \A_O$ is connected.
\end{lemma}
\begin{proof}
The problem translates into a question about monotone mappings of $\mathbb S^2$ onto itself. We choose and fix homeomorphisms $\varphi \colon \A \onto \mathbb S^2_\ast$ and $\psi \colon \Y \onto \mathbb S^2_\ast$ where $\mathbb S^2_\ast$ is the $2$-sphere with two points removed, say the north and south poles. We assume that $\varphi$ takes $\partial \A_I$ and $\partial \A_O$ into north and south poles, respectively, in the sense of cluster limits. Similarly, we assume that $\psi$ takes $\partial \Y_I$ and $\partial \Y_O$ into north  and south poles, respectively. Every homeomorphism $f \in \Ho_2 (\A, \Y)$ induces a unique homeomorphism $F \colon \mathbb S^2 \onto \mathbb S^2$ such that $F \circ \varphi = \psi \circ f$. Note that $F$ takes the north pole into the north pole  and the south pole into the south pole. Moreover, if a sequence of homeomorphisms $h_j \in \Ho^{1,2}_{\lim} (\A, \Y)$ converges weakly to $h$, then upon lifting to $\mathbb S^2$ we obtain homeomorphisms $\,H_j \colon \mathbb S^2 \onto \mathbb S^2\,$ converging uniformly to a surjective mapping $H \colon \mathbb S^2 \onto \mathbb S^2$. This is immediate from the uniform bounds
\[
\begin{split}
\dist \left[ h_j (x), \partial \Y_I \right] & \le \eta_{_I} (x) \sqrt{\mathscr E_\X [h_j]} \\
\dist \left[ h_j (x), \partial \Y_O \right] & \le \eta_{_O} (x) \sqrt{\mathscr E_\X [h_j]}
\end{split}
\]
and
\[\abs{h_j(x_1)-h_j(x_2)}^2 \le \frac{C_{\mathbb K} \, \mathscr E_{\X}[h_j]}{\log \left( e+ \frac{\diam \mathbb K}{\abs{x_1-x_2}}  \right)}\]
for $x_1, x_2 \in \mathbb K\,$, where $\mathbb K$ is any compact subset of $\,\A\, $. Here the functions $\, \eta_{_I}\,,\, \eta_{_O} \in \mathscr C(\mathbb R^2)\,$ are nonnegative, $\,\eta_{_I}(x) = 0 \,$ on $\,\mathbb A_I\,$ and $\,\eta_{_O}(x) = 0 \,$ on $\,\mathbb A_O\,$, see Lemma~\ref{Equicontinuity}. The limit map $H \colon \mathbb S \onto \mathbb S^2$ keeps the north and south poles fixed and is monotone, see the Kuratowski-Lacher Theorem \ref{Kuratowski-Lacher}. In particular, if $C$ is a continuum in $\mathbb S^2$ then so is $H^{-1}(C)$. Now Lemma~\ref{lemycircgen} follows if we return to the domains $\A$ and $\Y$ via the inverse homeomorphisms $\varphi^{-1} \colon \mathbb S^2_\ast \to \A$ and $\psi^{-1} \colon \mathbb S^2_\ast \to \Y$.
\end{proof}

\begin{lemma}\label{lemycirc}
Suppose that $h\in \Ho^{1,2}_{\lim}(\mathbb A, \mathbb Y)$ and
\[h_z \overline{h_{\bar z}} = \frac{c}{z^2} \qquad \mbox{ for some } c<0. \]
If $C \subset \A$ is a nonempty connected set such that $h_{ |_{C}}=\{y_\circ\}$. Then $C \subset \RR^\theta = \{\rho e^{i \theta} \colon r< \rho < R \}\,$, \; for some $0 \le \theta < 2\pi$.
\end{lemma}

\begin{proof}
Consider the projection $p \colon \A \to \mathbb S^1$, $p(z)= \frac{z}{\abs{z}}$. Thus $p(C)$, being a continuous image of a connected set, is a connected subset of $\mathbb S^1$. If $p(C)$ is a single point say $e^{i \theta}$, for some $\theta$, then $C \subset \RR^\theta$.  Suppose to the contrary that there is an arc $\{e^{i \theta} \colon \alpha < \theta < \beta\} \subset p(C)$. Thus every ray $\RR^\theta$, $\alpha < \theta < \beta$, intersects $C$. On the other hand, by Proposition~\ref{lily}, the union of rays which intersect $h^{-1}(y_\circ)$ has zero measure,  a contradiction.
\end{proof}

\begin{proof}[Proof of Theorem~\ref{thmcracks}]
Consider a ray $\RR^\theta \subset \A$. There is a point $z=\rho e^{i\theta} \in \mathcal A$, for some $\rho \in (r,R)$, since otherwise the set $\C \setminus \mathcal A= \mathcal A_I \cup \mathcal A_O \cup \RR^\theta$ would be connected, contradicting double connectivity of $\mathcal A$, see Lemma~\ref{double}. Denote
\[
\begin{split}
r_\theta &= \inf \{\rho \in (r,R) \colon \rho e^{i \theta} \in \mathcal A \} \\
R_\theta &= \sup \{\rho \in (r,R) \colon \rho e^{i \theta} \in \mathcal A \}
\end{split}
\]
so we have  $r \le r_\theta < R_\theta \le R$. Associated with these numbers is the partition of $\RR^\theta$ into: left-open right-closed interval $\RR^\theta_I = \{\rho e^{i\theta} \colon r < \rho \le r_\theta\}$  (possibly empty),  the open non-empty interval
$\RR^\theta_{\mathcal A} = \{\rho e^{i\theta} \colon r_\theta< \rho < R_\theta\}$,
and left-closed right-open interval $\RR^\theta_O = \{\rho e^{i\theta} \colon R_\rho \le \rho < R \}$ (possibly empty).

{\bf Claim.} {\it Every $\,z_\circ=\rho_\circ e^{i\theta} \in \RR^\theta_{\mathcal A}\,$; that is, $\,r_\theta < \rho_\circ < R_\theta\,$, lies in $\mathcal A$.} \\Suppose to the contrary that $h(z) \in \partial \Y$, say $h(z)=y_\circ \in \partial \Y_O\,$ (the case $y_\circ \in \partial \Y_I$ is similar).  Consider the set
\[C=h^{-1} (y_\circ) \subset \A\]
and its all connected components $\,\{C_\nu\,\}_{\nu \in \mathcal M}\,$, so $\,C= \bigcup\limits_{\nu \in \mathcal M} C_\nu$. Let $\,C_\circ \,$ denote the component that contains $z_\circ$. According to Lemma~\ref{lemycircgen} the set $\overline{C_\circ}$ intersects $\partial \A_O$. On the other hand, by Lemma~\ref{lemycirc}, $C_\circ$ lies in one and only one of the rays of $\A$, so in $\RR^\theta$. Thus $\,C_\circ\,$, being a connected set, takes the form
\[C_\circ = \{\rho e^{i \theta} \colon R'_\theta \le \rho < R\} \quad \mbox{ for some } r< R_\theta' <R.\]
Since $\rho_\circ e^{i \theta} \in C_\circ$, we see that
\begin{equation}\label{tah}
R_\theta' \le \rho_\circ < R_\theta.
\end{equation}
On the other hand, since $h(\rho e^{i \theta})=y_\circ$ for every $R_\theta' \le \rho <R$ and $h(\rho e^{i \theta}) \not= y_\circ$ for every $r_\theta < \rho < R_\theta$, it follows that $R'_\theta \ge R_\theta$, contradicting~\eqref{tah}. The claim is verified.

Now, by the definition of $R_\theta$, no point $\rho e^{i \theta}$ with $R_\theta < \rho < R$ lies in $\mathcal A$. Thus the interval $\RR_O^\theta$ lies in $\A \setminus \mathcal A = \mathcal B= \mathcal B_r \cup \mathcal B_R$. Actually $\RR^\theta_O$ has to lie in the outer component of $\C \setminus \mathcal A$,  the point $Re^{i \theta}$ belongs to $\overline{\RR_O^\theta}\,$. Thus $\RR_O^\theta \subset \A \cap \mathcal B_R$ and $h(\rho e^{i \theta})= y_\circ = h(R_\theta e^{i \theta})$ for every $R_\theta \le \rho < R$. In the same way we see that $\RR_I^\theta \subset \A \cap \mathcal B_r$ and $h(\rho e^{i \theta})= h(r_\theta e^{i \theta})$ for every $r <  \rho \le r_\theta$.
\end{proof}

\section{An integral identity}

\begin{lemma}\label{lemmaidentity}
 Let $\X$, $\Y$ and $\mathbb G$ be bounded  domains in $\C$. Suppose that $h \colon \mathbb G \onto\mathbb Y$ and $H \colon \mathbb \X \onto \Y$ are orientation preserving $\mathscr C^\infty$-diffeomorphisms  of finite energy.
Define
$\,f = H^{-1} \circ h \,\colon \mathbb G \onto \mathbb X$. Then we have
\begin{equation}\label{identity}
\begin{split}
\mathscr E_{\X}[H]-\mathscr E_{\mathbb G}[h] & = 4 \iint_{\mathbb G} \left[\frac{\abs{f_z-\gamma (z) f_{\bar z}}^2}{\abs{f_z}^2-\abs{f_{\bar z}}^2} -1 \right]
\, \abs{h_zh_{\bar z}}\, \dtext z\\
& + 4 \iint_{\mathbb G} \frac{(\,\abs{h_z}\,-\,\abs{h_{\bar z}}\,)^2\cdot \abs{f_{\bar z}}^2
}{\abs{f_z}^2-\abs{f_{\bar z}}^2}\,\dtext z
\end{split}
\end{equation}

where
\[
\gamma = \gamma(z) = \begin{cases}
{h_z\overline{h_{\bar z}}}{\, \abs{h_z  \overline{h_{\bar z}}}^{-1}} \qquad &\textnormal{\;if } h_z\overline{h_{\bar z}}\ne 0 \\
0 & \textnormal{\;otherwise.}
 \end{cases}
\]

 The integrals in  (\ref{identity}) converge .
\end{lemma}
\begin{proof}
It is worth noting that $\,f \colon \mathbb G \onto \mathbb X\,$ need not have finite energy. The convergence of the integrals, not obvious at the first glance, is a consequence of the finite energy condition imposed on the mappings $\,h\,$ and $\, H\,$.

We begin with the chain rule applied to $\,H=h\circ f^{-1}\colon \X \onto \Y\,,$
\begin{equation*}
\begin{split}
\frac{\partial H(w)}{\partial w} &= h_z(z)\frac{\partial f^{-1}}{\partial w}+h_{\bar z}(z)  \frac{\overline{\partial f^{-1}} }{\partial \overline{w}} \\
\frac{\partial H(w)}{\partial \bar w} &= h_z(z)\frac{\partial f^{-1}}{\partial \bar w}+h_{\bar z}(z)\overline{\frac{\partial f^{-1}}{\partial w}}
\end{split}
\end{equation*}
where $\,w= f(z)\,$. We express the complex partial derivatives of $\,f^{-1} \colon \X \to \X\,$ at $\,w\,$ in terms $\,f_z(z)\,$ and $\,f_{\bar z}(z) \,$ at $\,z= f^{-1} (w)\,$,
\begin{equation*}
\frac{\partial f^{-1}}{\partial w} = \frac{\overline{ f_z (z)}}{J(z,f)} \quad \mbox{ and } \quad
\frac{\partial f^{-1}}{\partial \bar w} = - \frac{f_{\bar z}(z)}{J(z,f)}
\end{equation*}
Note that the Jacobian determinant $\,J(z,f)= |f_z|^2 - |f_{\overline{z}}|^2\, $ is strictly positive. These expressions yield
\begin{equation*}
\frac{\partial H}{\partial w}  =\frac{h_z\overline{f_z}-h_{\bar z} \overline{f_{\bar z}} }{|f_z|^2 - |f_{\overline{z}}|^2} \quad \mbox{ and } \quad \frac{\partial H}{\partial \bar w}  =\frac{h_{\bar z}{f_z}-h_{ z}{f_{\bar z}}}{|f_z|^2 - |f_{\overline{z}}|^2}
\end{equation*}
 Next we compute the energy of $H$ over the set $\,f (\mathbb G) = \mathbb X\,$ by substitution $w= \chi (z)$,
\begin{equation*}
\begin{split}
\mathscr E_{f(\mathbb G)}[H]&=2\iint_{f(\mathbb G)} \left(\abs{H_w}^2+\abs{H_{\bar w}}^2\right)\,\dtext w \\
&=2\iint_{\mathbb G}\frac{\abs{h_z \overline{f_z} - h_{\bar z} \overline{f_{\bar z}}}^2+
\abs{h_{\bar z} f_z - h_{z} f_{\bar z}}^2}{\abs{f_z}^2-\abs{f_{\bar z}}^2}\,\dtext z.
\end{split}
\end{equation*}
On the other hand, the energy of $h$ over the set $\,\mathbb G\,$ equals
\[
\mathscr E_{\mathbb G}[h]=2\iint_{\mathbb G}\left(\abs{h_z}^2+\abs{h_{\bar z}}^2\right)\,\dtext z
\]
The desired formula follows by subtracting these two integrals,
\begin{equation}\label{cchain}
\begin{split}
\mathscr E_{\mathbb X}[H]-\mathscr E_{\mathbb G}[h] & =  4\iint_{\mathbb G}  \frac{\left(\abs{h_z}^2+\abs{h_{\bar z}}^2\right)\cdot \abs{f_{\bar z}}^2
-2\re \left[h_z\overline{h_{\bar z}} \overline{f_z} f_{\bar z}\right]}{\abs{f_z}^2-\abs{f_{\bar z}}^2}\,\dtext z\\
& = 4 \iint_{\mathbb G} \frac{ 2\abs{h_zh_{\bar z}} \cdot \abs{f_{\bar z}}^2
-2\re \left[h_z\overline{h_{\bar z}} \overline{f_z}\chi_{\bar z}\right]}{\abs{f_z}^2-\abs{f_{\bar z}}^2}\,\dtext z\\
& + 4 \iint_{\mathbb G} \frac{(\,\abs{h_z}\,-\,\abs{h_{\bar z}}\,)^2\cdot \abs{f_{\bar z}}^2
}{\abs{f_z}^2-\abs{f_{\bar z}}^2}\,\dtext z\\
& = 4 \iint_{\mathbb G} \left[\frac{\abs{f_z-\gamma(z) f_{\bar z}}^2}{\abs{f_z}^2-\abs{f_{\bar z}}^2} -1 \right]
\, \abs{h_zh_{\bar z}} \,\dtext z\\
& + 4 \iint_{\mathbb G} \frac{(\,\abs{h_z}\,-\,\abs{h_{\bar z}}\,)^2\cdot \abs{f_{\bar z}}^2
}{\abs{f_z}^2-\abs{f_{\bar z}}^2}\,\dtext z
\end{split}
\end{equation}
\end{proof}

\section{Some Free Lagrangians, Normal and Tangential Distortions}
In the 1980-ies a novel approach towards minimization of polyconvex energy
functionals for mappings between domains in $\R^n$ was developed and published by  Ball \cite{Ba0}. The
underlying idea was to view the integrand as convex function of null
Lagrangians. The term null Lagrangian pertains to a nonlinear
differential expression whose integral over any open region depends
only on the boundary values of the mapping, see \cite{BCO, Ed, Iw}. In this paper the mappings of our interest are  homeomorphisms $h \colon \X \onto \Y \,$ that are  given only on a part of the boundary (possibly empty part, in which case we are dealing with a traction free problem).
There  exist some nonlinear differential forms, defined on a class of
 homeomorphisms, whose integral means remain independent of their boundary values.  These
are rather special null Lagrangians, called \emph{free Lagrangians}\,\cite{IOnAnnuli}. Precisely, a free Lagrangian for a pair of domains $\,\X, \Y \subset \mathbb R^n\,$ is a nonlinear differential $n$ -form $\,L(x,h,Dh)\, \dtext x\,$ whose integral mean over $\,\mathbb X\,$ depends only on the homotopy class of a homeomorphism $ h \colon \X \onto \Y$. Here are a few of them in the planar domains, where we make use of polar coordinates $\rho$ and $\theta$
\begin{equation}
z= \rho e^{i\theta} \, , \qquad 0 \le \rho < \infty \; \; \textnormal{ and }\; \; 0 \leqslant \theta < 2 \pi \, .
\end{equation}
 The normal (radial) and tangential (angular) derivatives of a Sobolev mapping $f$ are defined by
\begin{equation}
f_N (z) := \frac{\partial f(\rho e^{i \theta})}{\partial \rho} \, , \qquad  \rho=\abs{z}
\end{equation}
and
\begin{equation}
f_T (z) :=  \frac{1}{\rho}\frac{\partial f(te^{i \theta})}{\partial \theta} \, , \hskip1cm \rho=\abs{z} \, .
\end{equation}
The Jacobian determinant of $f$ is
\[
J(\cdot, f)=J_f= \abs{f_z}^2-\abs{f_{\bar z}}^2 = \im \overline{f_N} f_T.
\]
The following three free-Lagrangians are concerned with a pair of the annuli $\A= \{x \in \C \colon r< \abs{x} <R \}$ and $\A^\ast= \{y \in \C \colon r_\ast< \abs{y} <R_\ast \}$.
\begin{itemize}
\item Pullback of a $2$-form in $\Y$ via an orientation preserving  homeomorphism  $h \in \Ho (\X, \Y) \cap \W^{1,2} (\X, \Y)$ is a free Lagrangian
\begin{equation}\label{jacfree}
\iint_\X N(|h|)\, J(x,h)\, \dtext x  =\iint_{\Y} N(\abs{y}) \, \dtext y \, .
\end{equation}
\item Normal differentiation gives rise to a free Lagrangian for  $h \in \Ho (\X, \A^\ast) \cap \W^{1,1} (\X, \A^\ast)$ defined by
\begin{equation}\label{norfree}
\begin{split}
\left| \iint_{\X} A(\abs{h})  \frac{\abs{h}_N}{\abs{x}} \, \dtext x \right|  & =  2 \pi \left| \int_r^R A \big(|h|\big) \frac{\partial |h|}{\partial \rho}\, \dtext \rho \right| \\ &= 2 \pi  \left| \int^{R_\ast}_{r_\ast} A(\tau) \, \dtext \tau \right| \, .
\end{split}
\end{equation}
\item A dual free Lagrangian for  $h \in \Ho (\A, \Y) \cap \W^{1,1} (\A, \Y)$ arises from tangential differentiation
\begin{equation}\label{tanfree}
\begin{split}
\left| \iint_{\A} B \big( |x|\big) \textnormal{Im} \frac{h_T}{h} \, \dtext x \right| & = \left| \int_r^R {B(t)} \left(\int_{|x|=t} \frac{\partial \textnormal{Arg}\, h}{\partial \theta} \, \dtext \theta\right) \dtext t \right|\\ & =  2 \pi \left| \int_{r}^{R} {B(t)} \, \dtext t \right|
\end{split}
\end{equation}
\end{itemize}
\begin{lemma}\label{lemKn}
Let $\mathbb X$  be a bounded doubly connected domain that separates the origin $\,0\,$ from $\,\infty\,$, and let $\,\A^\ast=A(r_\ast, R_\ast)\,$ be a circular annulus. If $h\in \mathscr{H} (\X, \A^\ast) \cap \W^{1,2}(\X, \A^\ast)\,$ then
\[\iint_{\mathbb X}  \frac{\abs{h_N}^2}{J_h}\, \frac{\dtext z}{\abs{z}^2} \ge 2 \pi \log (R_\ast/r_\ast) .\]
\end{lemma}
\begin{proof}
Choosing $A(\tau)= 1/\tau$ in~\eqref{norfree} we have
\[2 \pi \log (R_\ast/r_\ast) = 2 \pi \left| \int^{R_\ast}_{r_\ast} \frac{\dtext \tau}{\tau} \right| \le \iint_{\X} \frac{\abs{h_N}}{\abs{h}} \frac{\dtext z}{\abs{z}}\]
and
\[\left( \iint_{\X} \frac{\abs{h_N}}{\abs{h}} \frac{\dtext z}{\abs{z}} \right)^2 \le \iint_{\X} \frac{\abs{h_N}}{J_h} \frac{\dtext z}{\abs{z}^2} \iint_{\X} \frac{J_h}{\abs{h}^2} \, \dtext z =  \iint_{\X} \frac{\abs{h_N}}{J_h} \frac{\dtext z}{\abs{z}^2} \iint_{\A^\ast} \frac{\dtext y}{\abs{y}^2} \]
by~\eqref{jacfree}.

\end{proof}

\begin{lemma}\label{lemKt}
Let $\A=A(r,R)$ be a circular annulus and $\Y$ a bounded doubly connected domain of  finite conformal modulus. If $h\in \mathscr{H} (\mathbb A, \Y) \cap \W^{1,2}(\A, \Y)$, then
\[\iint_{\A}\frac{\abs{h_T}^2}{J_h} \, \frac{\dtext z}{\abs{z}^2} \ge 2 \pi \frac{\log^2 (R/r)}{\Mod \Y} .\]
\end{lemma}
\begin{proof}
There exists a conformal transformation $\,F \colon \Y \onto \A^\ast\,$, of $\,\mathbb Y\,$ onto an annulus  $\A^\ast = \{z \colon 0< r_\ast < \abs{z}< R_\ast\}$. We define $g=F\circ h \colon \Y \onto \A^\ast$.  Since $\,F\,$ is conformal
\[\iint_{\A}\frac{\abs{h_T}^2}{J_h} \, \frac{\dtext z}{\abs{z}^2}  =\iint_{\A}\frac{\abs{g_T}^2}{J_g} \, \frac{\dtext z}{\abs{z}^2}  \]
 Applying~\eqref{tanfree} with $B(t)=1/t$ we have
 \[2 \pi \log R/r = 2\pi \left| \int_r^R \frac{\dtext t}{t}\right| \le \iint_{\A} \frac{\abs{g_T}}{\abs{g}} \frac{\dtext z}{\abs{z}} \, .\]
 Now, it follows by H\"{o}lders inequality that
 \[\left(\iint_{\A} \frac{\abs{g_T}}{\abs{g}} \frac{\dtext z}{\abs{z}} \right)^2  \le \iint_{\A} \frac{\abs{g_T}^2}{J_g} \frac{\dtext z}{\abs{z}^2} \iint_{\A} \frac{J_g}{\abs{g}^2}= \iint_{\A} \frac{\abs{g_T}^2}{J_g} \frac{\dtext z}{\abs{z}^2}  \iint_{\A^\ast} \frac{\dtext y}{\abs{y}^2}\, . \]
 \end{proof}

\begin{lemma}\label{ctheory}
Let $\A=A(r,R)$ be a circular annulus, $0<r<R<\infty$, and $\Y$ a bounded doubly connected domain.
Suppose $\,h\in \mathscr W_{\loc}^{1,1}(\A, \Y)\,$ satisfies the Hopf-Laplace equation
\begin{equation}\label{hopf1}
h_z\overline{h_{\bar z}} \equiv \frac{c}{z^2}\qquad \text{in }\; \A\;,\;\text{and}\;\; J_h \ge 0 \;\;\text{almost everywhere}
\end{equation}
where $\,c\in\R\,$ is a constant.  Then we have point-wise inequalities
\begin{equation}\label{important}
\begin{cases}
\abs{h_N}^2 \le J_h, & \quad \mbox{if } c \le 0 \\
\abs{h_T}^2 \le J_h, & \quad \mbox{if } c \ge 0
\end{cases}
\end{equation}
\end{lemma}

\begin{proof}
The complex Hopf-Laplace equation ~\eqref{hopf1} reduces to the system of two real equations,
\begin{align}
\abs{h_N}^2 - \abs{h_T}^2 &= \frac{4c}{\abs{z}^2}\,; \label{hopf2a}\\
\re (\overline{h_N} h_T) &=0. \label{hopf2b}
\end{align}
Recall that $J_h = \im \overline{h_N} \,  h_T\ge 0$ which in view of~\eqref{hopf2b} reads as
\begin{equation}\label{hopf5}
J_h = \abs{h_N} \abs{h_T}
\end{equation}
Combining this and~\eqref{hopf2a} the inequalities~\eqref{important} follow.
\end{proof}

Geometric function theory is concerned with the distortion function
\[K^f = \frac{\abs{Df}^2}{J_f} = \frac{2 \left(\abs{f_z}^2\,+\,\abs{f_{\bar z}}^2\right)}{J_f}\, . \]
For the purpose of this paper we will decompose it as $K^f=K^f_N+K^f_T\,$,  where (using polar coordinates) the {\it normal} and {\it tangential distortions}  of $\,f\,$ are defined by the rules
\begin{align}
K_N^f & := \frac{\abs{f_z+ \frac{\bar z}{z}f_{\bar z}}^2}{J_f}= \frac{\abs{f_N}^2}{ J_f}\\
K_T^f & := \frac{\abs{f_z- \frac{\bar z}{z}f_{\bar z}}^2}{J_f}= \frac{\abs{f_T}^2} {J_f}
\end{align}
By convention, these two quotients are understood as $\,0\,$ whenever the numerator vanishes. Naturally,
they assume the value $\,+\infty\,$ if the Jacobian vanishes but the numerator does not.
For a mapping $\,f\in \W^{1,1}_{\rm loc}\,$ the quantities $\,f_N\,$, $\,f_T\,$, and $\,J_f\,$ are finite a.e.
and, therefore, $\,K_N^f\,$ and $\,K_T^f\,$ are well defined  measurable functions in the domain of
definition of $\,f\,$.

\subsection{The distortion of the difference of two solutions}\label{Distortion}
Suppose $h, H \in \mathscr W^{1,2}(\X, \C)$, $\,J_h \ge 0\,$ and $\,J_H \ge 0\,$, have the same \textit{Hopf-product},
\[h_z \overline{h}_{\bar z} = H_z \overline{H}_{\bar z} = \varphi (z) \not = 0 \quad \textnormal{ almost everywhere\;(not necessarily analytic).}\]
Consider the difference
\[F(z)=H(z)=h(z) \in  \mathscr W^{1,2}(\X, \C).\]
We have
\[h_{\bar z} \overline{F}_z = h_{\bar z} (\overline{H}_z - \overline{h}_z)  =h_{\bar z} \overline{H}_z  - \overline{\varphi} = h_{\bar z} \overline{H}_z -  \overline{H}_z H_{\bar z} =  \overline{H}_z  (h_{\bar z}- H_{\bar z})= - F_{\bar z}  \overline{H}_z     \]
where we notice that
\[
\begin{split}
\abs{h_{\bar z}}^2 \le \abs{h_{\bar z}} \abs{h_z}= \abs{\varphi} \\
\abs{H_{\bar z}}^2 \le \abs{H_{z}} \abs{H_{\bar z}}= \abs{\varphi}
\end{split}
\]
Hence $\abs{\varphi}^2 \abs{F_z}^2 \ge \abs{\varphi}^2 \abs{F_{\bar z}}^2$ so $J_F \ge 0$ almost everywhere. Next we introduce the Beltrami distortion coefficients
\[k_h(z) = \frac{\abs{h_{\bar z}}}{\abs{h_z}} \le 1 \quad \mbox{and} \quad k_H(z) = \frac{\abs{H_{\bar z}}}{\abs{H_z}} \le 1.\]
We find that
\[\abs{F_{\bar z}} = k_F (z) \abs{F_z} \quad \textnormal{ where } k_F (z)= \sqrt{k_h(z)\, k_H(z)} \le 1.\]
Indeed, we have
\[\abs{h_{\bar z}}^2 \abs{F_z}^2=\abs{H_{z}}^2 \abs{F_{\bar z}}^2 \;\;,\quad \textnormal{where } \abs{h_{\bar z}}^2 = k_h \abs{\varphi} \quad \textnormal{and} \quad k_H \abs{H_z}^2 = \abs{\varphi}.\]
Hence
\[k_hk_H \abs{F_z}^2 \abs{\varphi}= k_H \abs{h_{\bar z}}^2 \abs{F_z}^2 = k_H \abs{H_z}^2 \abs{F_{\bar z}}^2 = \varphi \, \abs{F_{\bar z}}^2\]
and therefore
\[\abs{F_{\bar z}}= \sqrt{k_h k_H} \abs{F_z} = k_F(z) \abs{F_z}. \]
Note that
\[
k_F(z) \; \begin{cases} <1 \quad \textnormal{whenever } J_h+J_H \not =0 \\
=1 \quad \textnormal{whenever } J_h+J_H  =0.
 \end{cases}
\]
In particular, $\, F\,$ has finite distortion whenever $\, J_h \not=0 \,$ or $\, J_H(z) \not=0\,$
\[\abs{DF}^2 =2 \left(\abs{F_z}^2+ \abs{F_{\bar z}}^2\right)  =2\, \frac{1+k_hk_H}{1-k_hk_H} J_F\;.\]

\section{Hopf solutions are energy-minimal, proof of Theorem \ref{thmhopf} }

Throughout this section $\X$ and $\Y$ are arbitrary bounded doubly connected domains, $\X$ being nondegenerate. We only need to prove the "if " part of Theorem~\ref{thmhopf}. The "only if " part follows from Corollary \ref{Weakreal} , because the energy-minimal mappings are critical points for all variations along  $\,\partial \mathbb X\,$. In the case of negative Hopf-differential we actually obtain slightly better result;  it holds in a larger class of the Hopf solutions. Precisely,

\begin{proposition}\label{propcp}
Suppose the Hopf-differential $h_z \overline{h_{\bar z}} \, \dtext z \otimes \dtext z$ defined for $h \in \Ho^{1,2}_{\lim} (\X, \Y)$ is holomorphic and real negative along $\partial \X$. Then
\begin{equation}\label{minHO}
\mathscr E_\X [h] = \inf \{\,\mathscr E_\X [g] \colon g \in \Ho^{1,2}_{\lim} (\X, \Y) \,\} \, .
\end{equation}
Furthermore, $h$ is a unique  (up to the conformal change of variables in $\X$) minimizer in $\Ho^{1,2}_{\lim} (\X, \Y)$.
 \end{proposition}
 Obviously, if $\,h\,$ happen to belong to $\,\overline{\Ho}_2 (\X, \Y)\subset \Ho^{1,2}_{\lim} (\X, \Y)\,$, then the equation (\ref{minHO}) yields  (\ref{EM}), as desired. However, as yet,  the equality  $\, \Ho^{1,2}_{\lim} (\X, \Y)\, = \,\overline{\Ho}_2 (\X, \Y)\,$ has been established only under Lipschitz regularity of  $\,\mathbb Y\,$, \cite{IOa}.

\begin{proof}

By virtue of  Theorem~\ref{thmcpo} $\,h \colon \X \onto \Y$  is a harmonic diffeomorphism. A conformal transformation of $\X$ onto an annulus $\A = \{z \colon r< \abs{z}<R\}$ takes $\Ho^{1,2}_{\lim} (\X, \Y)$ onto $\Ho^{1,2}_{\lim} (\A, \Y)$. Thus it involves no loss of generality in assuming that $\X=\A$, so we have the equality
\[h_z \overline{h_{\bar z}}= \frac{c}{z^2} \qquad c \in \R \setminus \{0\} \, .\]
The assumption that $\varphi (z)\, \dtext z \otimes \dtext z$ is negative along $\partial \A$ yields $c>0$. In view of Lemmas~\ref{ctheory} and~\ref{lemKt} we have
\[2 \pi \frac{\log^2 \frac{R}{r}}{\Mod \Y} \le \iint_{\A} \frac{\abs{h_T}^2}{J_h} \frac{\dtext z}{\abs{z}^2} \le \iint_{\A}  \frac{\dtext z}{\abs{z}^2}  = 2 \pi \log \frac{R}{r}\]
This shows that $\Mod \A \le \Mod \Y$. Now we appeal to Theorem~\ref{exstmodY} which asserts that there exists a $\C^\infty$-diffeomorphism $H \colon \A \onto \Y$  such that
\[\mathscr E_\X [H] = \inf \{\;\mathscr E_\X [g] \;\colon\; g \in \Ho^{1,2}_{\lim} (\X, \Y) \;\} \, .\]
Furthermore, $\,H\,$ is  unique  (up to the conformal change of variables in $\X$) minimizer in $\Ho^{1,2}_{\lim} (\X, \Y)$.
It remains to show that $\mathscr E_\X [H]  = \mathscr E_\X [h]$. For this purpose we consider a diffeomorphism
\[f:=H^{-1} \circ h \,\colon\, \A \onto \A.\]
Applying Lemma~\ref{lemmaidentity}, we have
\begin{equation}\label{identitykt}
\begin{split}
\mathscr E_{\A}[H]-\mathscr E_{\mathbb A}[h] & = 4 \, {c}\iint_{\mathbb A} \left[\frac{\abs{f_z-\frac{z}{\bar z}f_{\bar z}}^2}{\abs{f_z}^2-\abs{f_{\bar z}}^2} -1 \right]
\, \frac{ \dtext z }{\abs{z}^2}\\
& + 4 \iint_{\mathbb A} \frac{(\,\abs{h_z}\,-\,\abs{h_{\bar z}}\,)^2\, \abs{f_{\bar z}}^2
}{\abs{f_z}^2-\abs{f_{\bar z}}^2}\,\dtext z
\end{split}
\end{equation}
Since
\[\frac{\abs{f_z-\frac{z}{\bar z} f_{\bar z}}^2}{\abs{f_z}^2-\abs{f_{\bar z}}^2}= K_T^f\]
by Lemma~\ref{lemKt} we conclude with the inequality
\[\iint_{\A} \frac{\abs{f_z-\frac{z}{\bar z} f_{\bar z}}^2}{\abs{f_z}^2- \abs{f_{\bar z}}^2} \frac{\dtext z}{\abs{z}^2} \ge 2\pi \log (R/r)= \iint_{\A} \frac{\dtext z}{\abs{z}^2}.\]
Furthermore, $\,f\,$ is an orientation-preserving diffeomorphism. Therefore,
\begin{equation}\label{unikt}
\mathscr E_{\A}[H]-\mathscr E_{\mathbb A}[h]  \ge  0,
\end{equation}
and so $\mathscr E_{\A}[H]=\mathscr E_{\mathbb A}[h] $.
\end{proof}
It is the case of positive Hopf-differential that, for the proof of
Theorem~\ref{thmhopf},  we need to take the solutions in the class of strong limits of homeomorphisms.

\begin{proposition}\label{propcn}
Suppose that a Hopf-differential $\,h_z \overline{h_{\bar z}} \, \dtext z \otimes \dtext z\,$, defined for $h \in \overline{\Ho}_2 (\X, \Y)\,$, is holomorphic and real positive along $\,\partial \X\,$. Then $\,h\,$ is an energy-minimal.
 \end{proposition}
\begin{proof}
 A conformal transformation of $\X$ onto an annulus $\A = \{z \colon r< \abs{z}<R\}$ induces an isometry of $\overline{\Ho}_2 (\X, \Y)$ onto $\overline{\Ho}_2 (\A, \Y)$. Thus we may  assume that $\X=\A$, so as to apply the equality
\[h_z \overline{h_{\bar z}}= \frac{c}{z^2} \qquad c \in \R \setminus \{0\} \, .\]
The assumption that $\varphi (z)\, \dtext z \otimes \dtext z$ is positive  along $\partial \A$ simply means that $c<0$.

We write $\mathbb G:=h^{-1}(\Y)$. In view of  Theorem \ref{Partial Harmonicity} the mapping $h \colon \A \to \overline{\Y}$ is a harmonic diffeomorphism  from $\mathbb G \subset \A$ onto $\Y$.
Let $H \colon \A \onto \Y$ be an orientation preserving $\mathscr C^\infty$-diffeomorphism.  We denote
\[f=H^{-1} \circ h \colon \mathbb G \onto \A.\]
Applying Lemma~\ref{lemmaidentity}, we have
\begin{equation}\label{identitykn0}
\begin{split}
\mathscr E_{\A}[H]-\mathscr E_{\mathbb G}[h] & = 4 \, \abs{c}\iint_{\mathbb G} \left[\frac{\abs{f_z+\frac{z}{\bar z}f_{\bar z}}^2}{\abs{f_z}^2-\abs{f_{\bar z}}^2} -1 \right]
\, \frac{ \dtext z }{\abs{z}^2}\\
& + 4 \iint_{\mathbb G} \frac{(\abs{h_z}\,-\,\abs{h_{\bar z}})^2\, \abs{f_{\bar z}}^2
}{\abs{f_z}^2-\abs{f_{\bar z}}^2}\,\dtext z \\
&= 4 \, \abs{c} \iint_{\mathbb G} \left[ K_N^f -1 \right] \frac{\dtext z}{\abs{z}^2} \\
& + 4 \iint_{\mathbb G} \frac{(\abs{h_z}-\abs{h_{\bar z}})^2\, \abs{f_{\bar z}}^2
}{J_f}\,\dtext z .
\end{split}
\end{equation}
Before estimating the right hand side we will show that
\begin{equation}\label{remaining}
\mathscr E_{\mathbb A}[h]= \mathscr E_{\mathbb G}[h] + 4 \, \abs{c} \iint_{\A \setminus \mathbb G} \frac{\dtext z}{\abs{z}^2}.
\end{equation}
In view of Lemma~\ref{jacoprop} $J_h=0$ in $\mathbb A \setminus \mathbb G$. Since $c<0$, by Lemma~\ref{ctheory}, $\abs{h_N}=0$ in $\mathbb A \setminus \mathbb G$. Therefore, $\abs{Dh}^2=\abs{h_T}^2$  in $\mathbb A \setminus \mathbb G$. On the other hand, in view of~\eqref{hopf2a}, we have $\abs{h_T}^2 = - 4 c \abs{z}^{-2}\,$  in $\,\mathbb A \setminus \mathbb G$. Therefore
\[\iint_{\A \setminus \mathbb G} \abs{Dh}^2 = -4\, c \int_{\A \setminus \mathbb G} \frac{\dtext}{\abs{z}^2}. \]
Combining~\eqref{identitykn0} with~\eqref{remaining} we arrive at the identity
\begin{equation}\label{givename}
\begin{split}
\mathscr E_{\A}[H]-\mathscr E_{\mathbb A}[h]  & = 4 \, \abs{c}\left[ \iint_{\mathbb G} K_N^f (z)\, \dtext z - \iint_{\A} \frac{\dtext z}{\abs{z}^2}\right]\\ & + 4 \iint_{\mathbb G} \frac{(\abs{h_z}\,-\,\abs{h_{\bar z}})^2\, \abs{f_{\bar z}}^2
}{J_f}\,\dtext z .
\end{split}
\end{equation}
According to Lemma~\ref{lemKn}
\[ \iint_{\mathbb G} K_N^f \ge 2 \pi \log \left(R/r\right) =  \iint_{\A} \frac{\dtext z}{\abs{z}^2}. \]
Therefore, if $H \colon \A \onto \Y$ is a $\mathscr C^\infty$-diffeomorphism, we can write
\begin{equation}\label{identitykn}
\mathscr E_{\A}[H]-\mathscr E_{\mathbb A}[h]  \ge  4 \iint_{\mathbb G} \frac{(\abs{h_z}-\abs{h_{\bar z}})^2\, \abs{f_{\bar z}}^2
}{J_f}\,\dtext z \ge 0.
\end{equation}
The last inequality follows from the fact that $f$ preserves the  orientation. Hence $\mathscr E_{\A}[H] \ge \mathscr E_{\mathbb A}[h]$ for an arbitrary $H\in \overline{\mathscr{H}_{2}}(\mathbb A, \mathbb Y)$, meaning that $h$ is an energy-minimal map.

\section{Uniqueness, Proof of Theorem~\ref{thmuni}}
Let $H$ and $h$ be energy-minimal mappings in $\overline{\mathscr{H}_{2}}(\mathbb X, \mathbb Y)$. According to Theorem~\ref{thmhopf} the Hopf differentials of $h$ and $H$ are analytic and real along $\partial \X$.

{\bf Case 1} If the mapping $h$ has  negative Hopf-differential along $\partial \X\,$, then the uniqueness ($\,h\,$ differs from $\,H\,$ by a conformal change of variables in $\,\X\,$) was already established in  Proposition~\ref{propcp}.

{\bf Case 2}   Suppose the mapping $\,h\,$ has  positive Hopf-differential along $\,\partial \X\,$.  We may assume that $\,\X\,$ is an annulus. Therefore,
\[
\begin{split}
h_z\overline {h_{\bar z}} &= \frac{c}{z^2} \qquad \textnormal{ for some } c<0 \\
H_z\overline {H_{\bar z}} &= \frac{c_{_H}}{z^2} \qquad \textnormal{ for some } c_{_H}\in \R.
\end{split}
\]
The reader may wish to notice that at this point we do not know (though it will be shown to be true) that the real valued constant $\,c_{_H} \,$ coincides with $\,c\,$. We also cannot assume that $\,c_{_H} < 0\,$.
Let us denote by  $\mathbb G:=h^{-1} (\Y)$ and $\mathbb G_{_H}:=H^{-1}(\Y)$. These are doubly connected domains separating the boundary circles, Lemma~\ref{double}.  In view of Theorem ~\ref{Partial Harmonicity} the mapping $\,H\,$ is a harmonic diffeomorphism from $\mathbb G_{_H}$ onto $\Y$. Recall that $\mathbb G = h^{-1}(\Y)$.  Thus $f:= H^{-1} \circ h\colon  \mathbb G \onto \mathbb G_H$ is an orientation  preserving diffeomorphism which also preserves the order of the boundary components. We denote the inverse of $\,f\,$ by $\,g =f^{-1}=h^{-1} \circ H \colon \mathbb G_H \onto \mathbb G\,$. Choose and fix a disk $\,\mathbb D \subset \overline{\mathbb D}\subset \mathbb G_H\,$. The aim is to prove that $g$ is analytic in $\mathbb D$ and, therefore, analytic in $\mathbb G_H\,$.
Since $H\in \overline{\mathscr{H}}_2(\mathbb A, \mathbb Y)= \overline{\textit{Diff}}_{_{\,2}}\,(\mathbb A \, ,\mathbb Y)$  there exists a sequence of diffeomorphisms $H_k \colon \mathbb A \onto \Y$, $\,k =1,2, \dots\,$, converging to $\,H\,$ in $\,\mathscr W^{1,2}(\A)\,$ and $\,c$-uniformly. In analogy  to $\,f\,$ and $\,g\,$ we define
\[f^k := H_k^{-1} \circ h \;\colon \mathbb G_h \onto \mathbb A \quad \mbox{ and } \quad g^k := H_k^{-1} \circ h \;\colon \mathbb A \onto \mathbb G_h . \]
Since $\,H_k \colon \overline{\mathbb D} \to \Y\,$ converge uniformly to $\,H \colon \mathbb D \to H(\overline{\mathbb D})\,$, where $\,H(\overline{\mathbb D})\,$ is a compact subset of $\,\Y\,$,  there is a neighborhood $\,\mathbb U\,$ of $\,H(\overline{\mathbb D})\,$, compactly contained in $\,\Y\,$, such that $\,H_k (\overline{\mathbb D}) \subset \mathbb U\,$
for all sufficiently large $\,k\,$, say $\,k \ge k_\circ$. Since $\,\overline{\mathbb U}\,$ is compact in $\,\Y\,$ the set $\,\mathbb F:= h^{-1} (\mathbb U)\,$ is compact in $\,\mathbb G\,$. Then we note that
\[g^k (\overline{\mathbb D})= h^{-1} \big(H_k (\overline{\mathbb D})  \big) \subset h^{-1} (\overline{\mathbb U}) = \mathbb F \quad \mbox{ for } k \ge k_\circ.\]
Furthermore $\,g^k\,$ converge uniformly to $\,g=h^{-1} \circ H \colon \overline{\mathbb D} \to h^{-1}\big( H (\overline{\mathbb D})\big) \subset \mathbb F\,$. In view of~\eqref{identitykn} we have
\[
\begin{split}
\mathscr E_{\A}[H_k]-\mathscr E_{\mathbb A}[h]  & \ge  4 \iint_{\mathbb G} \frac{(\abs{h_z}-\abs{h_{\bar z}})^2\, \abs{f^k_{\bar z}}^2
}{J_{f^k}}\,\dtext z \\
&  \ge  4 \iint_{\mathbb F} \frac{(\abs{h_z}-\abs{h_{\bar z}})^2\, \abs{f^k_{\bar z}}^2
}{J_{f^k}}\,\dtext z \\
& \ge 4 \epsilon^2 \iint_{\mathbb F} \frac{ \abs{f^k_{\bar z}}^2
}{J_{f^k}}\,\dtext z, \qquad \epsilon >0.
\end{split}
\]
Here we used the inequality
\[\left(\abs{h_z}- \abs{h_{\bar z}}  \right)^2 \ge \epsilon^2 >0 \quad \mbox{ for } z \in \mathbb F\]
because $\,h\,$ is a diffeomorphism on $\,\mathbb G\,$.
In the later integral we will make a substitution $\,z= g^k (w)\,$, where $\,w\in f^k (\mathbb F)\,$, to obtain
\[
\begin{split}
\mathscr E_{\A}[H_k]-\mathscr E_{\mathbb A}[h]   \ge 4 \epsilon^2 \iint_{f^k (\mathbb F)} \abs{g^k_{\bar w} (w)}^2  \,\dtext w \ge 4 \epsilon^2 \iint_{\mathbb D} \abs{g^k_{\bar w} (w)}^2  \,\dtext w
\end{split}
\]

Letting $\,k \to \infty$ we find that $\,g^k_{\bar w} \to 0\,$ in $\,\mathscr L^2 (\mathbb D)\,$. Since $\,g^k \to g\,$ uniformly on $\,\mathbb D\,$ we see that $\,g_{\bar w}=0\,$ on $\,\mathbb D\,$, where we recall that $\,\mathbb D\,$ was an arbitrary open disk compactly contained in $\mathbb G_H$.  Consequently, $g \colon \mathbb G_H \onto \mathbb G$ is a conformal map, so is $f \colon \mathbb G \onto \mathbb G_H\,$. This conformal map preserves the order of the boundary components, so
\[\frac{1}{2\pi i}\int_{\gamma} \frac{f'(z)\, \dtext z}{f(z)}=1.\]
where $\gamma$ can be any  smooth Jordan curve that separates the boundary components of $\mathbb G$, and is oriented counterclockwise.
Next we differentiate the equation $\,H \big(f(z) \big) =h(z)\,$ with respect to  $\,z \in \mathbb G\,$. Chain rule yields,
\[\frac{\partial H}{\partial f} f'(z) =h_z (z) \quad \mbox{ and } \quad \frac{\partial H}{\partial \overline{f}}  \overline{f'(z)}= h_{\bar z} (z). \]
These formulas together with the following two Hopf-Laplace equations
\[\frac{\partial H}{\partial f} \frac{\overline{\partial H}}{\partial \overline{f}} = \frac{c_{_H}}{f^2} \quad \mbox{ and } \quad h_z \overline{h_{\bar z}} = \frac{c}{z^2} \]
yield a differential equation for $f$,
\begin{equation}\label{kav1}
\left(\frac{z\, f'}{f}\right)^2 = \frac{c}{c_{_H}}\;; \quad \textnormal{this is a nonzero constant.}
\end{equation}
Therefore, we have a continuous branch of the square root
\[\frac{z\, f'}{f} = \alpha \quad \textnormal{a constant (real or imaginary).}
\]
On the other hand,
\[\alpha = \frac{1}{2\pi i} \int_{\gamma} \frac{\alpha \, \dtext z}{z} = \frac{1}{2\pi i}\int_{\gamma} \frac{f'(z)\, \dtext z}{f(z)}=1.\]
Therefore, equation~\eqref{kav1} reduces to  $\left(\frac{f}{z}\right)'=0$ in $\mathbb G$. Thus $f= \lambda z$ for some constant $\lambda \not= 0$. Substituting into~\eqref{kav1} we find that $c=c_{_H}$. We may, and do, assume that $\lambda$ is a real positive number, simply by rotating the $z$-variable  $\,H = H(z)\,$ if necessary. We may also assume that $0< \lambda \le 1$, because $H(z)= h (z/\lambda)$ for $z\in \mathbb G_{_H}$ and the mappings $\,H\,$ and $\,h\,$ are interchangeable. Summarizing
\[\mathbb G_{_H}= \lambda \mathbb G \subset \mathbb A, \quad 0< \lambda \le 1, \quad H(\lambda z)=h(z) \quad \mbox{for } z\in \mathbb G \quad \mbox{and} \quad c=c_{_H}.\]

Now come the key ingredients of the proof.  Consider the difference
\[F(z)=H(z)-h(z).\]
{\bf Claim.} {\it For every $0 \le \theta < 2 \pi$ we have
\begin{equation}\label{claim68976}
\lim_{\rho \searrow r} F(\rho e^{i \theta}) = 0= \lim_{\rho \nearrow r} F(\rho e^{i \theta})
\end{equation}
}
\begin{proof}[Proof of Claim] We recall that
\[H(\lambda z)= h(z) \qquad \textnormal{ for } z \in \mathbb G\]
and $\mathbb G_{_H}= \lambda \mathbb G \subset \mathbb A$ for some $0 < \lambda \le 1$. Using the same notation as in Theorem \ref{thmcracks} we recall the ray segments $\,\RR^\theta\,$ and the corresponding partitions
\[
\begin{split}
r \le r_\theta < R_\theta \le R \qquad \textnormal{ for } h \\
r \le \lambda\, r_\theta < \lambda\, R_\theta \le R \qquad \textnormal{ for } H.
\end{split}
\]
We shall only demonstrate the proof of the limit in (\ref{claim68976}) at the inner boundary, the proof for the outer boundary is similar.

The case $\lambda=1$ is obvious, because $\mathbb G = \mathbb G_H$ and $H(z)=h(z)$ for $z\in \mathbb G$. According to Remark~\ref{unirem} the values of $H$ and $h$ outside $\mathbb G$ are uniquely determined by those on $\mathbb G$, so $H \equiv h$ on $\mathbb A$. Now assume that $0 < \lambda < 1$.

{\bf Case 1}. Let  $\,r<\lambda\, r_\theta\,$, so the intervals $(r, r_\theta]$ and $(r, \lambda \,r_\theta]$ are not empty. We have, for every $r< \rho \le \lambda\, r_\theta < r_\theta$, the following equalities,
\[H(\rho \,e^{i \theta}) = H(\lambda\, r_\theta \,e^{i \theta}) = h (r_\theta\, e^{i \theta}) = h(\lambda\, r_\theta \,e^{i \theta}).\]
Therefore,
\[F(\rho \,e^{i \theta})=0 \qquad \textnormal{ for all } r< \rho \le \lambda\, r_\theta .\]

{\bf Case 2}. Let $\,r= \lambda \,r_\theta < r_\theta$. For every $\,r< \delta \le r_\theta$ we have
\[h(\delta \,e^{i \theta})= h(r_\theta\, e^{i \theta}) = \lim_{\delta \searrow r_\theta} h(\delta \,e^{i \theta}), \quad \textnormal{because } r_\theta\, e^{i \theta} \in \overline{\mathbb G} \cap \mathbb A .\]	
Since $\,\delta \,e^{i \theta} \in \mathbb G\,$ we it follows that
\[h(r_\theta \,e^{i \theta})=  \lim_{\delta \searrow r_\theta }h(\lambda\, \delta  \,e^{i \theta}) = \lim_{\rho \searrow r} H(\rho\, e^{i \theta}) .\]
Thus
\[\lim_{\rho \searrow r} F(\rho \,e^{i \theta})=0.\]
\end{proof}
The function $F=H-h \in \mathscr W^{1,2}(\mathbb A)$, admits an extension to the entire complex plane as a member of $\mathscr W^{1,2}(\mathbb C)$. In particular, $F$ restricted to almost every ray
\[\RR^\theta = \{\rho \,e^{i \theta} \colon r< \rho < R  \}, \qquad 0 \le \theta < 2 \pi\]
has an extension  as an absolutely continuous function on $\,\RR_\ast^\theta = \{\rho\, e^{i \theta} \colon r_\ast \le \rho \le R_\ast  \}\,$, where $\,0<r_\ast <r<R <R_\ast < \infty\,$. Since
\[\lim_{\rho \searrow r} F(\rho \,e^{i \theta}) = 0= \lim_{\rho \nearrow r} F(\rho \,e^{i \theta})  \]
by~\eqref{claim68976} the zero extension of $\,F \colon \mathbb A \to \C\,$ to the entire complex plane, denoted by $\,F_\circ \colon \C \to \C\,$, is absolutely continuous on almost every ray $\,\RR^\theta_\circ =\{\rho \,e^{i \theta} \colon 0 \le \rho \le R  \}\,$. Obviously, $\,F_\circ\,$ is absolutely continuous on almost every circle $\,\mathcal C_\rho =\{\rho\, e^{i \theta} \colon 0 \le \theta < 2 \pi \}\,$. In other words $\,F_\circ \in \mathscr W^{1,2}(\mathbb C)\,$ and  vanishes outside $\,\mathbb A\,$.  This yields
\[\iint_{\mathbb A} J_F = \iint_{\C} J_{F_\circ}=0.\]
Since $\,J_{F} \ge 0\,$ almost everywhere in $\,\mathbb A\,$ we conclude that $\,J_{F} =0\,$ almost everywhere in $\,\mathbb A\,$. Next for $\,z\in \mathbb G \cup \mathbb G_{_H}\,$ we have the distortion inequality
\[\abs{DF(z)}^2 \le \frac{1+k_F(z)}{1-k_F(z)} J_F (z), \quad 0 \le k_F(z) < 1.\]
Thus $\,DF =0\,$ almost everywhere on $\,\mathbb G \cup \mathbb G_{_H}\,$ which implies that $\,F= \textnormal{constant}\,$ on every component of $\,\mathbb G \cup \mathbb G_{_H}\,$. We claim that $\,\mathbb G = \mathbb G_{_H}\,$. Suppose to the contrary that $\, \mathbb G_{_H} \setminus \mathbb G \not= \emptyset\,$, so there is a disk $\,\mathbb D \subset \mathbb G_{_H} \setminus \mathbb G \,$, and $\,F= \textnormal{constant}\,$ on $\,\mathbb D\,$. Let $\,\RR \subset \mathbb D\,$ be a ray segment. Since $\,\mathbb D \subset \mathbb A \setminus \mathbb G\,$ the function $\,h\,$ is constant on $\,\RR\,$. In particular, $\,H=F+h\,$ is constant on $\,\mathbb I \subset \mathbb G_{_H}\,$, contrary to the injectivity of $\,H\,$ when restricted to $\,\mathbb G_{_H}\,$. In summary, $\,\mathbb G = \mathbb G_{_H}\,$, so $\,\lambda =1\,$ and $\,H(z)=h(z)\,$ on $\,\mathbb G\,$. This yields $\,H=h\,$ in $\,\mathbb A\,$, by Remark~\ref{unirem}.
\end{proof}

\section{How to make a square hole washer from an annulus}
To illustrate the ideas  developed above we shall take on stage the energy minimal deformations of an annulus
\[ \mathbb X = \{x_1\,+\,i\,x_2 \colon 1< x_1^2 \,+\,x_2^2 < R^2\} \]
  onto a \textit{square washer}
  $$\,\mathbb Y = \{y_1\,+\,i\,y_2 \colon 1< |y_1|\, +\,|y_2|\, , \quad y_1^2\, +\,y_2^2 < \rho^2\} \, , \qquad \rho >1  $$
The method of exploring symmetries in  this example actually applies to hexagonal nuts or many more symmetric doubly connected polygonal figures.  We shall look for the energy-minimal deformations among strong $\,\mathscr W^{1,2}$-limits  of orientation preserving homeomorphisms $\,h\colon \mathbb X \onto\mathbb Y\,$. Recall the notation $\,\overline{\mathscr H_{_2}}(\mathbb X\,,  \mathbb Y)\,$ for the class of such mappings. By virtue of Theorems \ref{Emin}\, and \,\ref{thmuni} ,  we see that

\begin{theorem}\label{washer-nut}  There is a map  $\,\hbar \in \overline{\mathscr H_{_2}}(\mathbb X\,,  \mathbb Y)\,$ (unique up to the rotation of $\,\mathbb X\,$) with least energy.
\end{theorem}
 The geometric features of the energy-minimal mapping $\,\hbar\,$ depend on  the outer radius of the annulus $\,\mathbb X\,$, so we vary it while keeping the target $\,\mathbb Y\,$ unchanged.
\begin{theorem}\label{critical radii}
There are two \textit{critical outer radii} $\, 1< R_\circ \leqslant R^\circ\,$ such that:

\flushleft
\begin{enumerate}

\item[1.] For $\,1 < R \leqslant R_\circ\,$, the energy-minimal map $\, \hbar\,$ is a harmonic diffeomorphism of $\,\mathbb X\,$ onto $\,\mathbb Y \,$; no cracks occur in $\mathbb X\,$.
\item[2.] For $\,R \geqslant R^\circ\,$,   $\,\hbar\,$  fails to be injective; cracks are unavoidable.
  \end{enumerate}
For all $\,R > R^\circ\,$, the energy-minimal deformation has the following properties:
\begin{enumerate}
\item[3.]  There are four straight line segments (possibly points in $\,\partial \mathbb X_I\,$)
  $\,\textbf{I},\,i \textbf{I},\, - \textbf{I} ,\, -i \textbf{I}\,$, \;\;where\;\; $\;\,\textbf{I} = [1 ,\,a ] \subset \mathbb R\,\;, 1 \leqslant  a < R \,$, such that\\
  \begin{itemize}
\item
     $\hbar(\textbf{I}) = \{1\},\,\;\;\hbar(i\textbf{I}) = \{i\},\,\;\; \hbar(-\textbf{I}) = \{-1\}\,\; and \,\;\;\hbar(-i\textbf{I}) = \{-i\}\,$
\item The set $\,\mathbb X'\deff \mathbb X \setminus [\,\textbf{I}\,\cup\,i \textbf{I}\,\cup\, - \textbf{I} \,\cup\, -i \textbf{I} \;]\,$ is a doubly connected domain and $\,\hbar \colon \mathbb X'  \onto \mathbb Y\,$ is a harmonic diffeomorphism.
\item  The map   $\,\hbar\,$ is invariant under the rotation by the right angle;  $\,\hbar(iz) = i \hbar(z)\,$ for all $\,z \in \mathbb X\,$.
\end{itemize}
\end{enumerate}
\end{theorem}

\begin{center}\begin{figure}[h]
\includegraphics[width=0.8\textwidth]{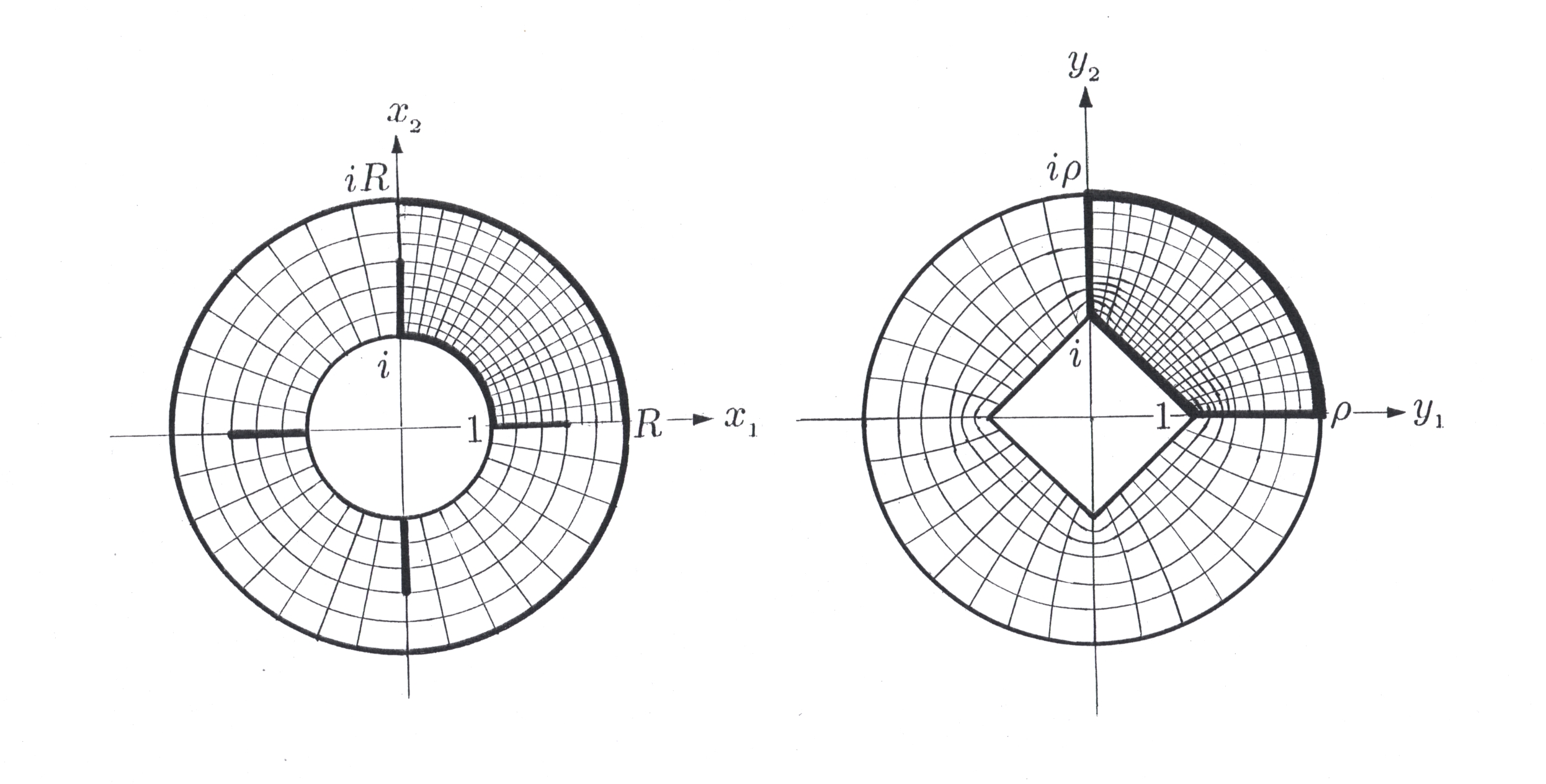} \caption{How to make a square hole washer from an annulus.}\label{fig1}
\end{figure}\end{center}

 An interpretation of this result is that in order to make a square hole washer from a thick annulus one must make cuts  in the annulus. The energy-minimal deformation takes those cracks into the corners of the inner boundary  of $\,\mathbb Y \,$, exactly  where it fails to be convex.

 \begin{theorem}\label{globLip}
In spite of the presence of cracks in case $\,R \geqslant R^\circ\,$, the energy-minimal map is Lipschitz continuous up to the closure  $\,\overline{\mathbb X}\,$.
\end{theorem}

\subsection{Proof of Theorem \ref{critical radii}}
\begin{proof}
\textit{Conditions 1}   and \textit{2} are immediate from Theorem  \ref{exstmodY}. In the proof of \textit{Condition 3} we shall explore the symmetries in both the domain $\X$ and the target $\Y$. We shall first construct an energy-minimal harmonic diffeomorphism of the first quadrant sector
\[\mathbb S = \{\rho e^{i \theta} \colon 1< \rho < R \, , \quad 0< \theta < \pi/2\}\]
onto the first quadrant of the square hole washer,
\[\mathbb T = \{\,y_1 \,+\, i\,y_2 \;\colon\, 1< y_1 \,+\, y_2 \;\;,\;\;\; y_1^2 + y_2^2 < \rho^2 \, , \quad y_1>0 , \; y_2 >0\,\} \, . \]
We view $\mathbb S$ and $\mathbb T$ as quadrilaterals with corners.
\[
1,\, R,\, iR,\, i\, \in \partial \mathbb S \qquad \textnormal{and} \qquad
1,\, \rho,\, i\rho, \,i \,\in \partial \mathbb T  \,, \;\;\textnormal{respectively}.
\]

Consider the class $\,\mathbb Q_2 (\mathbb S , \mathbb T)\,$ of all quadrilateral mappings from $\,\mathbb S\,$ onto $\,\mathbb T\,$. These are homeomorphisms $\,f \colon \mathbb S \onto \mathbb T\,$ which extend continuously to $\,f \colon \overline{\mathbb S} \onto \overline{\mathbb T}\,$ and take the corners of $\,\mathbb S\,$ into the corners of $\,\mathbb T\,$ in the following order
\[f(1)=1\, , \quad f(R)=\rho \, , \quad f(iR)=i\rho \, , \quad f(i)=i \, .\]
\begin{lemma}\label{lem34}
There exists $h \in \mathbb Q_2 (\mathbb S, \mathbb T)$ of smallest energy.
\end{lemma}
\begin{remark}
The interested reader may wish to notice that this lemma and its proof given below actually hold for every pair of quadrilaterals in which the target is convex.
\end{remark}
\begin{proof}[Proof of Lemma~\ref{lem34}]
Let $\{f_j\}_{j=1,2, \dots}$ denote the minimizing sequence of quadrilateral mappings $f_j \colon \overline{\mathbb S} \onto \overline{\mathbb T}\,$, so
\[\mathscr E_{_{\mathbb S}} [f_j] \to \inf \{\;\mathscr E_{_{\mathbb S}}[f]\; \colon \;\;f \in \mathbb Q_2 (\mathbb S , \mathbb T)\;\}\, \deff \,m \, . \]
We assume that $\{f_j\}$ converges weakly in $\W^{1,2} (\mathbb S)$ to some mapping $f \in \W^{1,2} (\mathbb S)\,$ and observe that $\,f_j \colon \overline{\mathbb S} \onto \overline{\mathbb T}\,$ also converge uniformly to $\,f \colon \overline{\mathbb S} \to \overline{\mathbb T}\,$, see Remark \ref{normalization}.

The boundary map $\,f \colon \partial \mathbb S \onto \partial \mathbb T\,$, being a uniform limit of homeomorphisms $\,f_j \colon \partial \mathbb S \onto \partial \mathbb T\,$, is monotone. At this point we appeal to a generalized variant of Theorem
 of Rad\'o-Kneser-Choquet~\cite{Duren}, see Theorem \ref{RaKnCh},  which ensures that the harmonic extension of $f \colon \partial \mathbb S \onto \partial \mathbb T$, denoted by $h \colon \mathbb S \onto \mathbb T$ is a diffeomorphism. Certainly $h$ is a quadrilateral map, i.e. $h \in \mathbb Q_2 (\mathbb S , \mathbb T)$. We shall now see that $h$ is an energy-minimal map. Indeed,  by the Dirichlet energy principle and because $\,f\,$ is a weak $\W^{1,2}\,$ -limit of $f_j$, the energy of $h$ is estimated by
\[\mathscr E_\X [h] \le \mathscr E_{\X}[f] \le \liminf \mathscr E_{\X}[f_j]=m \, .\]
But $\,m\,$ was the infimum energy in the class $\,\mathbb Q_2 (\mathbb S , \mathbb T)\,$, so $\,\mathscr E_{\X}[h]=m\,$. Thus $\,h \in \mathbb Q_2 (\mathbb S , \mathbb T)\,$ is the energy-minimal map.
\end{proof}
The reader may see that the sequence $\,f_j\,$ actually converges strongly in $\W^{1,2} (\mathbb S)$ and that $f=h$ in $\mathbb S$.
\begin{lemma}
The energy-minimal map $h \colon \mathbb S \onto \mathbb T$ satisfies the Hopf-Laplace equation
\begin{equation}\label{hl5}
h_z \overline{h_{\bar z}}= \frac{c}{z^2} \qquad \textnormal{for some } c \in \R \, .
\end{equation}
\end{lemma}
\begin{proof}
This is immediate from Proposition  \ref{HopfAnnulus}, because $\, h\,$ is a critical point for all quadrilateral variations in $\,\mathbb S\,$.
\end{proof}

Our next step is to extend $\,h \colon \overline{\mathbb S} \to \overline{\mathbb T}\,$ to the whole annulus $\,\X\,$. First we reflect $\mathbb S$ about its horizontal side $\,(1,R) \subset \partial \mathbb S\,$,  to obtain right-half of the annulus
\[\{z \colon 1 < \abs{z}<R \, , \quad \re z >0\}\, , \]
and then about the imaginary axis to obtain the whole annulus $\X$. These reflections, when applied to $\,h\,$, give rise to a mapping, denoted by $\,\hslash \colon \overline{\X} \onto \overline{\Y}\,$, which satisfies the Hopf-Laplace equation
\begin{equation}\label{Hopfagain}
\hslash_z \overline{\hslash_{\bar z}}= \frac{c}{z^2} \qquad \textnormal{in } \X \,
\end{equation}
Recall that $\,\hslash \colon \overline{\mathbb S} \onto \overline{\mathbb T}\,$ was obtained as a limit of homeomorphisms $f_j \colon \overline{\mathbb S} \onto \overline{\mathbb T}\,$ converging uniformly and strongly in $\,\W^{1,2} (\mathbb S , \mathbb T)\,$.
With the aid of the same reflections we obtain a sequence of homeomorphisms $\,f_j \colon \mathbb X \onto \mathbb Y\,$ converging uniformly and  strongly in $\,\mathscr W^{1,2}(\mathbb X, \mathbb Y)\,$ to $\,\hslash\,$. Thus $\,\hslash \in \,\overline{\mathscr H}_2(\mathbb X, \mathbb Y)\,$. Now Theorem \,\ref{thmhopf}  tells us that $\,\hslash\,$ is an energy-minimal map within the class $\,\overline{\mathscr H}_2(\mathbb X, \mathbb Y)\,$, and it is unique up to the rotation of $\,\mathbb X\,$, by Theorem \,\ref{thmuni}\,. In particular $\,\hslash\,$ is unique once we normalize it by setting $\,\hslash(1) = 1 \,$. In addition, the uniqueness combined with the fact that the quadrilateral map $\,\hslash \colon \overline{\mathbb S} \onto \overline{\mathbb T }\,$ takes corners into corners in the respective order imply the equation
\begin{equation}\label{symmetry}
\hslash (i\,z ) \;=\; i\,\hslash (z)\;,\;\;\;\textnormal{for all}\;\; z \in \mathbb X .
\end{equation}

Now let us look more closely at the case $\,R \geqslant R^\circ\,$, with $\,R^\circ\,$ small enough so the cracks in $\,\mathbb X\,$ are present. This yields, via Theorem \,\ref{thmcpo},  that the constant $\,c\,$ in (\ref{Hopfagain})\, is negative; that is, the differential  $\,\hslash_z \overline{\hslash_{\bar z}}\,\,\textnormal{d}z \otimes \textnormal{d}z\,$  is positive along the boundary circles. According to Lemma\, \ref{lemycirc} the cracks must lie in the rays of the annulus $\,\mathbb X \,$. Since $\,\hslash \colon \overline{\mathbb S} \onto \overline{\mathbb T}$ is  a diffeomorphism the cracks can occur only along the rays  $\,\textbf{I},\,i \textbf{I},\, - \textbf{I} ,\, -i \textbf{I}\,$, \;where\; $\,\textbf{I}\,$ is a straight line segment $\,\textbf{I} = [1 ,\,a ] \subset [1, \, R\,]\,$; they cannot reach the outer boundary of $\mathbb X\,$
because $\,\mathbb Y\,$ is convex at the outer boundary, see Theorem \,\ref{nocracks}\,. Moreover, in view of the equation  (\ref{symmetry}), all the cracks have the same length. The proof of Theorem  \ref{critical radii}\,\,is complete.

\end{proof}

\subsection{Lipschitz regularity, Proof of Theorem \ref{globLip}}
\subsubsection{Lipschitz regularuity of $\,\hslash\,$} Local Lipschitz continuity of the admissible solutions to the Hopf-Laplace equation (with analytic right hand side) has been shown in \cite{CIKO}. Thus, it only remains to establish Lipschitz  continuity of  $\,\hslash\,$ near the boundary points. We follow the classical method of extending $\,\hslash\,$  beyond the boundary of $\,\mathbb X\,$ in order to  take advantage of the interior regularity results. It  is characteristic of this method that sometimes the delicate structure of the equation is lost beyond the boundary. And that is exactly  what happens here. For this reason we shall need a generalization of \cite{CIKO} that has been  settled in \cite{IKOli} as follows:

\begin{theorem}\label{LipHopf2}
Let $h\in \mathscr W^{1,2}(\Omega)$ be  a mapping with nonnegative Jacobian.
Suppose that the Hopf product $\;h_z\,\overline{h_{\bar z}} \; $ is bounded and H\"{o}lder continuous.
Then $h$ is locally Lipschitz (but not necessarily $\mathscr C^1$-smooth).
\end{theorem}

  Recall that $\hslash \colon \overline{\X} \onto \overline{\Y}\,$ is continuous and $\,|\hslash(x)| \equiv \rho\,$ for $\,|x| = R\,$. We extend $\,\hslash\,$ to the annulus $\,\mathbb A = \mathbb A(1, R^2)\,$ by the rule

  $$\,\hslash(z) =  \rho^2\, \big /\,\overline{\hslash(R^2/\bar{z})} \,,\;\;\;\textnormal{for}\;\;\; R\leqslant  |z|  < R^2$$

  This extension certainly belongs to the Sobolev space $\,\mathscr W^{1,2}(\mathbb A)\,$ and it has nonnegative Jacobian determinant. Its Hopf product equals

\begin{equation}\label{Hopfagain}
\hslash_z \overline{\hslash_{\bar z}}=  \begin{cases}  c \, z^{-2} \qquad\qquad\qquad\qquad \textnormal{for }\;\; 1<|z|\leqslant R  \,\\
\,\\
 c\,\rho^4   \; |\hslash(R^2/\bar{z}) |^{- 4}\, z ^{-2} \qquad \textnormal{for }\;\; R\leqslant |z| < R^2  \,
 \end{cases}
\end{equation}
 From here we grasp the pointwise inequality $\,|\hslash_z \,\hslash_{\bar z}| \leqslant | c |\,$. Since the Jacobian is nonnegative, it follows that $\,|\hslash_{\bar z}|^2  \leqslant |\hslash_z \,\hslash_{\bar z}| \leqslant | c |\,$, so $\,\hslash_{\bar z} \in \mathscr L^\infty (\mathbb A)\,$. Hence $\,\hslash_{z} \in \mathscr L^p_{\textnormal{loc}} (\mathbb A)\,$, for every exponent $\,2\leqslant p < \infty\,$. Actually, $\,\hslash_{z} \in \textrm{BMO}_{\textnormal{loc}} (\mathbb A)\,$, which places $\,\hslash\,$ somewhere close to $\,{Lip}_{\,_{\textnormal{loc}}}(\mathbb A)\,$  but not quite there. Nonetheless, Sobolev imbedding theorem tells us the $\,\hslash\,$ is locally H\"{o}lder continuous of exponent $\,\alpha = 1 - \frac{2}{p} \,> 0\,$, so is the Hopf product $\,\hslash_z \overline{\hslash_{\bar z}}\,$.  By Theorem \ref{LipHopf2} we conclude that $\,\hslash \in \,{Lip}_{\,_{\textnormal{loc}}}(\mathbb A)\,$, as desired.\\
 Similarly, one obtains Lipschitz continuity of $\,\hslash\,$ near any of the four circular arcs in $\,\partial \mathbb X_I\,$, which $\,\hslash\,$ takes into a  straight line segment of $\,\partial \mathbb Y_I\,$, by analogous reflection trick.

But dealing near the corners $\, 1,\,i,\,-1,\,-i\,\in \partial \mathbb X_I\,$ requires an additional trick. Recall the length $\,a\,$ of the crack, $\,1 < a < R\,$ and consider the sector $\,\Omega = \{ z ;\; \frac{1}{a} <|z| < a \;\,, \;\; \re z > 0\,\}\, \deff \mathbb S^{\pi/2}_{-\pi/2} \left( \frac{1}{a} ,\, a  \right )\, $  as a neighborhood of the point $\,1 \in \partial \mathbb X_I\,$. We shall extend $\,\hslash\,$ onto $\Omega\,$ using  two different reflection procedures.
\begin{itemize}
\item reflection in the $\,\mathbb X\,$ -space, along the arc $\,\{ z = e ^{i\,\theta}\,;\; 0 < \theta < \pi/2\,\} \subset \partial \mathbb X\,$, followed by the reflection in the $\,\mathbb Y\,$ -space, along the straight line that passes through  the segment $\,(-1 ,\, i )\subset \partial \mathbb Y\,$.
\item reflection in the $\,\mathbb X\,$ -space, along the arc $\,\{ z = e ^{i\,\theta}\,;\; -\pi/2 < \theta < 0\,\} \subset \partial \mathbb X\,$, followed by the reflection in the $\,\mathbb Y\,$ -space, along the straight line that passes through  the segment $\,(-i ,\, 1 )\subset \partial \mathbb Y\,$.
\end{itemize}
Precisely, we set the values of $\,\hslash\,$ for $\, z = t e^{i\,\theta}\,$, with $\, -\pi/2 < \theta < \pi/2\,$ and $\, \frac{1}{a} < t \leqslant 1\,$, by the rule

\begin{equation}\label{Extension}
\hslash_z \overline{\hslash_{\bar z}}=  \begin{cases}  -i\, \overline{\hslash(1/\bar{z})} \,+\,1 +\, i \, & \textnormal{\;,\;for}\; 0 \leqslant \theta < \pi/2\\
\hslash(z) = 1\, & \textnormal{\;,\;for} \,\, \theta = 0\\
i\, \overline{\hslash(1/\bar{z})} \,+\, 1 - \,i \, & \textnormal{\;,\;for}\; 0 \leqslant \theta < \pi/2
 \end{cases}
\end{equation}

This is certainly a continuous map of Sobolev class $\,\mathscr W^{1,2}(\Omega)\,$, with nonnegative Jacobian, and its Hopf product equals
$$
\hslash _z \,\overline{\hslash_{\bar{z}}} = \frac{c}{z^2}\;\;,\;\;\textnormal{almost everywhere in }\;\;\Omega\;.
$$
By Theorem \ref{LipHopf2}   we conclude that $\,\hslash \in Lip_{\textnormal{loc}}(\Omega)\,$,  as desired.

\section{Two examples}
Let us illustrate how Theorem \,\ref{thmhopf}\, works in practice for mappings $\,h \colon \mathbb X \onto\mathbb Y\,$ between doubly connected domains. Thus we shall look at the Hopf differential $\,\,\mathfrak h_z \overline{\mathfrak h_{\bar{z}}}\,\,\textnormal{d} z \otimes\textnormal{d} z \,\,$ to check as to whether  it is real along $\,\partial\mathbb X\,$. The two examples here also serve to show  a delicate difference between Hopf differentials being positive  or  negative. We can, and do, assume without affecting the results that the domain $\,\mathbb X\,$ is an  annulus . Thus in either case we are dealing with admissible solutions to the Hopf-Laplace equation
\begin{equation}
\mathfrak h_z \overline{\mathfrak h_{\bar{z}}}\, = \frac{c}{z^2} \,\,,  \;\;\;\textnormal{in an annulus}\;\;\; \mathbb X
\end{equation}
\subsection{Case $\,c > 0\,$, Hopf differentials are negative along $\,\partial \mathbb X\,$}
This is the case in which no cracks emerge, by virtue of Theorem  \ref{thmcpo}.
Consider the following infinite series of orthogonal harmonic functions
 \begin{equation}
 \begin{split}
 \mathfrak h(z) &= -\frac{2}{R} \,\log |z| \;+\; \frac{R^2-1}{R} \sum_{n=1}^\infty \frac{z^n\,-\,\bar{z}^{-n}}{n\,R^n} \\ & =\;
 \big(R -\frac{1}{R} \big)\, \textnormal{\large{Log}} \frac{ R \,z\bar{z} - z }{ R -z}\;\;-\;\; 2 R \log |z|
 \end{split}
 \end{equation}
 The series converges in the closed annulus $\, \mathbf A \deff \{ z\,;\, R^{-1} \leqslant |z| \leqslant R\,\}\,$, except for two boundary  points $\, z= R^{\pm 1}$. Her the symbol $\, \textnormal{\large{Log}}\,$   stands for the continuous branch of logarithm in  $\,\mathbb C+ \deff \{\xi \in \mathbb C\,; \;\re  \xi > 0 \,\}\,$ that is specified by $\, \textnormal{\large{Log}}\,1 \,=\,0  $.
 Observe that the expression $\,\xi = \frac{ R \,z\bar{z} - z }{ R -z}\,$ takes values in $\,\mathbb C+\,$, whenever $\,R^{-1} \leqslant |z|\leqslant R\,$ and $\; z \neq R^{\pm 1}\,$. Clearly $\,\mathfrak h(z) = 0\,$ for $\,|z| = 1\,$ and we have the following identity $\,\mathfrak h(1/z) = - \mathfrak h(\bar{z}) \,$. Elementary geometric considerations show that $\,\mathfrak h \,$ takes the open
 annulus $\,\mathbb A_R = \{ z\,; 1< |z| < R\,\}\,$ homeomorphically into a simply connected domain with puncture at the origin which is contained in a horizontal strip
 \begin{equation}\label{strip}
   \mathfrak h(\mathbb A_R)  \subset\Big\{ \zeta\,; \;|\im  \zeta | <  \frac{\pi}{2} \Big(R - \frac{1}{R}  \Big)\;\Big\}
 \end{equation}
 Passing to the limit as $\,R \rightarrow \infty\,$, we obtain harmonic homeomorphism $\,\mathfrak h^\infty =z - \bar{z}^{-1}\,$ of $\,\mathbb A_\infty = \{ z\,; 1< |z| < \infty\,\}\,$ onto the punctured complex plane  $\,\mathbb C_\circ\,$.
 Next observe that the function $\,\xi(z) = \frac{ R \,z\bar{z} - z }{ R -z}\,$ agrees with the M\"{o}bius transformation  $\,\xi(z) = \frac{ R \,\rho^2 - z }{ R -z}\,$  when restricted to any circle $\,\mathcal C_\rho \,=\,\{ z\,;\; |z| = \rho\,\}$. Thus the image of $\,\mathcal C_\rho\,$ under $\,\xi = \xi(z)\,$ is a circle in $\,\mathbb C+\,$. We now observe that $\, \textnormal{\large{Log}}\,$ takes circles in $\,\mathbb C+\,$ into strictly convex smooth Jordan curves. The curves $\,\mathfrak h(\mathcal C_\rho)\,, 1<\rho< R\, $,  resemble a family of ellipses with common focus at the origin. But the image of the outer circle, $\,\rho = R\,,$ looks more like a parabola, but it has been  flattened to fit into the horizontal strip \, (\ref{strip}).

\begin{center}\begin{figure}[h]
\includegraphics[width=0.9\textwidth]{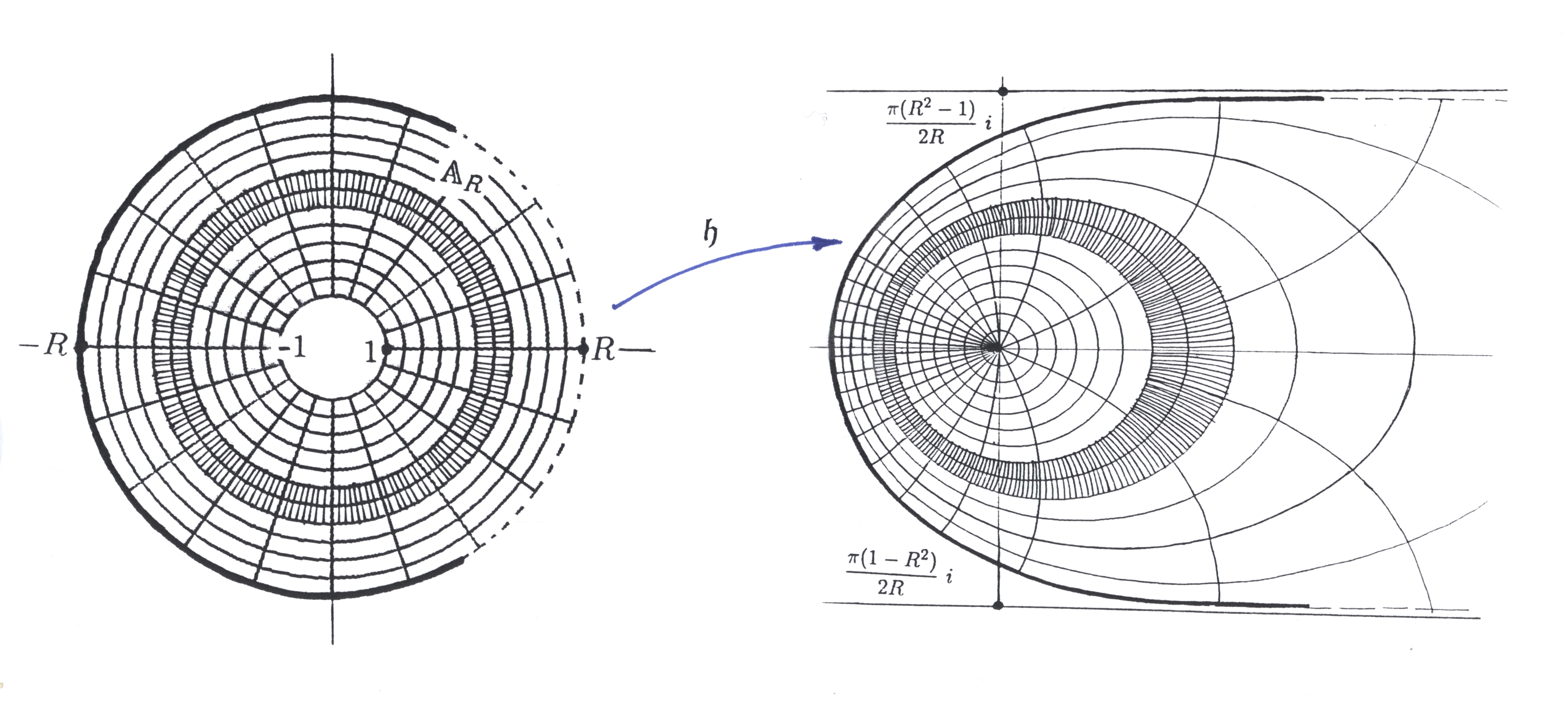} \caption{The map $\mathfrak h$ is energy-minimal because $\mathfrak h_z \overline{h_{\bar{z}}} \;=\; \frac {1}{z^2}$}\label{fig3}
\end{figure}\end{center}

\subsubsection{Hopf-Laplace Equation}
The Hopf differential $\,\,\mathfrak h_z \overline{\mathfrak h_{\bar{z}}}\,\,\textnormal{d} z \otimes\textnormal{d} z \,\,$  is real and negative on every circle $\,\mathcal C_\rho\,,\;\; R^{-1} < \rho <\R$. These circles are horizontal trajectories while rays are vertical trajectories. Precisely, we have
\begin{equation}\label{ExampleHopf}
 \mathfrak h_z \overline{h_{\bar{z}}} \;=\; \frac {1}{z^2}\;,\;\;\;\;\;\; \textnormal{for all}\;\; R^{-1} \leqslant |z| \leqslant R\,, \;\; \textnormal{except for}\;\;  z = R^{\pm 1}
\end{equation}
Indeed, a straightforward differentiation shows that
\begin{equation}
\mathfrak h_z = \frac{Rz - 1}{R -z}\,\frac{1}{z}\,, \;\;\;\textnormal{ and}\;\;\;\overline{\mathfrak h_{\bar{z}}} = \frac{R - z}{R z -1}\,\frac{1}{z}\;,
\end{equation}
 whence the equation~\eqref{ExampleHopf}. The Jacobian determinant
$$
 J(z, \mathfrak h)  = |
 \mathfrak h_z|^2 - |\mathfrak h_{\bar{z}}|^2 \;=\; \frac{(R^2 -1 ) (|z|^2 - 1 )}{|R-z|^2 |Rz -1 |^2}\;
$$
changes sign when crossing the unit circle.

\subsubsection{The map $\,\mathfrak h \,$ is energy-minimal}
In general, finding the energy-minimal homeomorphism between designated domains is not a trivial matter. Sometimes it comes unplanned, like in the above example, in which the Hopf equation~\eqref{ExampleHopf} combined with Theorem~\ref{thmhopf} yields:
\begin{proposition}
Denote by  $\,\mathcal A_\rho \;= \mathfrak h(\mathbb A_\rho)\,$,\,where $\,1<\rho < R\,$. Among all homeomorphisms $\, f \colon  \mathbb A_\rho \onto \mathcal A_\rho\,$ the minimum Dirichlet energy is attained for $f = \mathfrak h$, uniquely up to a rotation of $\A_\rho$. The minimum of energy equals
$$
 2 \,\iint_{\mathbb A_\rho}\left( |\mathfrak h_z |^2 \; +\; |\mathfrak h_{\bar{z}}\,|^2 \right ) \; = \frac{4\,\pi\, \log\,\rho}{R^2} \;+\;
 \frac{2\,\pi ( R^2 - 1 )^2}{R^2} \;\textnormal{\Large{log}}\,\frac{R^2 \rho^2 - 1}{R^2 \rho^2 - \rho^4}
$$
\end{proposition}
\begin{proof}
That $\,\mathfrak h \,$ represents the energy-minimal deformation and is unique up to the rotation follows from Theorems~\ref{thmhopf}  and ~\ref{thmuni}. As for the computation of the minimum energy, we write
\begin{equation}
\begin{split}
 \mathscr E_{\mathbb A_\rho}[\mathfrak h ] \;&=\;2 \,\iint_{\mathbb A_\rho}\left( |\mathfrak h_z |^2 \; +\; |\mathfrak h_{\bar{z}}\,|^2 \right ) \; \\& =
  2 \,\iint_{1<|z|<\rho}\left( \left|\frac{R z - 1}{R - z}\right|^2 \; +\; \left|\frac{R - z}{R z - 1 }\right|^2 \right ) \frac{\textnormal{d} z}{|z|^2}\;
\end{split}
\end{equation}
We expand the holomorphic functions $\,\frac{R z - 1}{R - z}\,$ and $\,\frac{R - z}{R z - 1 }\,$ into Laurent series and use orthogonality of the system $\,\{z^n\}_{ n = 0, \pm 1, \pm2, ...}\,$ to arrive at the claimed formula.
\end{proof}
\begin{remark}
The mapping $\,\mathfrak h \,$ also represents the energy-minimal deformation of any sub-annulus $\,\mathbb A_{r_2} \setminus \mathbb A_{r_1}\,$  onto a doubly connected shell $\, \mathcal A_{r_2} \setminus \mathcal A_{r_1}\,$ where $\, 1< r_1 < r_2<\R\, $. The  energy  $\,\mathscr E_{\mathbb A_{r_2} \setminus \mathbb A_{r_1}}[\mathfrak h ] =  \mathscr E_{\mathbb A_{r_2}}[\mathfrak h ] - \mathscr E_{\mathbb A_{r_1}}[\mathfrak h ] \,$.
\end{remark}

\subsubsection{A dual Example}
Somewhat dual harmonic map can be assembled from the orthogonal basis of harmonic functions in the annulus $\, \{ z\,; 1 \leqslant |z| \leqslant R\,\}\,$ with a cut along the interval $\, [0,\,R]\,$. View it as degenerate quadrilateral.

\begin{equation}
\begin{split}
 \mathfrak h &= \mathfrak h_R (z) =   - \frac{2\,i}{R}\, \textnormal{Arg}\, z \; +\; \frac{R^2-1}{R} \sum_{n=1}^\infty \,\frac{z^n + \bar{z}^{-n}}{n R^n}\\&=
 - \frac{2\,i}{R}\, \textnormal{Arg\,z}  \;-\;\frac{R^2-1}{R} \,\textnormal{\Large{log}}\, \frac{(R-z)(\bar{z} R - 1 )}{\bar{z} R^2}\;,\;\;\;0 <  \textnormal{Arg\,z} < 2\,\pi
 \end{split}
\end{equation}
Straightforward computation reveals that
$$
\mathfrak h_z = \frac{zR - 1}{z(R - z)}\,\;\;\;\textnormal{and}\;\;\; \overline{\mathfrak h_{\bar{z}}} = \frac{z - R}{z(zR - 1)} \;,\;\;\textnormal{hence}\;\;\; \mathfrak h_z\,\overline{\mathfrak h_{\bar{z}}} = \frac{-1}{z^2}
$$
We observe that the Hopf differential
\begin{equation}
\mathfrak h_z\,\overline{\mathfrak h_{\bar{z}}}\;\,\textnormal{d}z\otimes \textnormal{d}z = \,-\,\frac{ \,\textnormal{d}z\otimes \textnormal{d} z}{z^2}
\end{equation}
 is  positive along the boundary circles  and negative along the cut; this is the energy-minimal quadrilateral deformation.

\subsection{Case $\,c < 0\,$, Hopf differentials are positive along $\,\partial \mathbb X\,$ }\label{NitscheHammering}
This is the case in which cracks may, though need not, emerge. We shall exploit the typical example from the studies of the energy minimization problem via the concept of \textit{free Lagrangians}\,\cite{IOnAnnuli}. As suggested by their name, free Lagrangians  are nonlinear differential expressions defined for mappings  $\,h \colon \mathbb X \onto\mathbb Y\,$ whose integral means are independent of the boundary values, they depend only on the homotopy class of the mapping. These are special\textit{ null Lagrangians}. The latter were proposed by J. Ball \cite{Ba0} and  successfully exploited in \cite{BCO} . The most obvious novelty of free Lagrangians is that, once they are identified properly for a specific energy-minimization problem, the problem  is solved instantly.\\ Theorem \,\ref{thmhopf} offers another approach. The utility of this approach is illustrated by the quick proof of the following\\

\begin{theorem}\label{thn=2}
Let $\mathbb X= \mathbb A(r,R)$ and $\mathbb Y = \mathbb A(r_\ast, R_\ast)$ be planar annuli. \\
\underline{Case 1.}  If
\begin{equation}\label{Nitsche1}
\frac{R_\ast}{r_\ast} \geqslant \frac{1}{2} \left(\frac{R}{r}+ \frac{r}{R}\right)
\end{equation}
then the harmonic homeomorphism
\[\frak h(z)= \frac{r_\ast}{2} \left(\frac{z}{r}+ \frac{r}{\bar z}\right), \quad \frak h \, \colon \mathbb X \onto \mathbb Y\]
has the smallest energy among all homeomorphisms $h \colon \mathbb X \onto \mathbb Y\,$, and  is unique up to a rotation of $\mathbb A$.\\
\underline{Case 2.}  If
\begin{equation}\label{Nitsche2}
\frac{R_\ast}{r_\ast} < \frac{1}{2} \left(\frac{R}{r}+ \frac{r}{R}\right)
\end{equation}
then the infimum energy among all homeomorphisms $\,h \colon \mathbb X \onto \mathbb Y$ is not attained. Let a radius $r< \sigma < R$ be determined by the equation
\[\frac{R_\ast}{r_\ast}= \frac{1}{2} \left(\frac{R}{\sigma}+ \frac{\sigma}{R}\right) \]
Then the following mapping
\[\frak h(z) =  \left\{\begin{array}{lll} r_\ast \frac{z}{|z|} \; \;\;&\;\;\; r <|z|\leqslant \sigma \;\;\;\;\;\textnormal{\textit{cracks along the rays}}\;[r, \rho]\, e^{i\, \theta}\\ \\
 \frac{r_\ast}{2} \left(\frac{z}{\sigma} + \frac{\sigma}{\bar z}  \right)\; & \;\;\;\;\sigma \leqslant |z|<R   \;\;\;\;\;\textnormal{\textit{harmonic diffeomorphism}}\;\end{array} \right.    \]
has smallest energy within the class $\,\overline{\mathscr H}(\mathbb X, \mathbb Y)\,$. This energy-minimal map is unique up to a rotation of $\mathbb X\,$.
\end{theorem}
\begin{proof}
The proof  is immediate from Theorem \,\ref{thmhopf} once we notice that
\begin{equation}
\mathfrak h_z\,\overline{\mathfrak h_{\bar{z}}}\;\,\textnormal{d}z\otimes \textnormal{d}z = \,-\,\frac{r_*^2}{4}\frac{ \,\textnormal{d}z\otimes \textnormal{d} z}{z^2}\;,\;\;\;\textnormal{in either case}
\end{equation}
\end{proof}

\begin{center}\begin{figure}[h]
\includegraphics[width=0.9\textwidth]{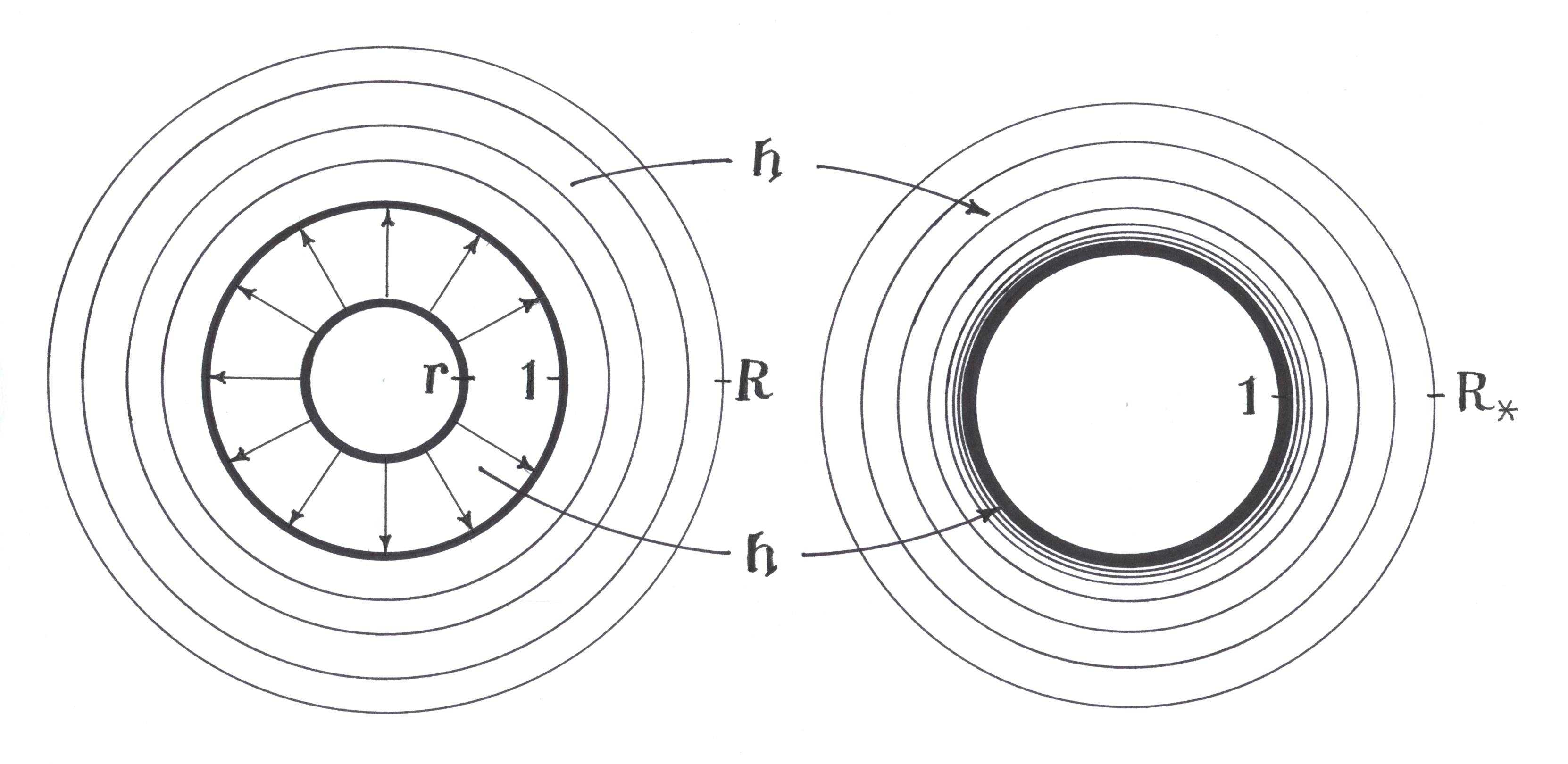} \caption{Energy-minimal map; $\,\mathfrak h(z) = \frac{z}{|z|}\,$ for $\,r < |z| \leqslant 1\,$ and  $\,\mathfrak h(z) = \frac{1}{2}(z + \bar{z}^{-1})\,$ for $\,1 \leqslant |z| < R\,$.}\label{fig4}
\end{figure}\end{center}

\subsubsection{The upper halves of the annuli} Let $\,\mathbb X^+ = \mathbb X \cap \mathbb C^+\,$ and $\,\mathbb Y_+ = \mathbb Y \cap \mathbb C^+\,$
denote the upper halves of the annuli. One can view them as quadrilaterals with corners at $\,\{-R,\, -r,\,r ,\, R\}\,$ and  $\,\{-R_* \,,\, -r_*\,,\,r_*\, ,\, R_*\}\,$, respectively. Every quadrilateral map $\,  h \, \colon \mathbb X^+ \onto \mathbb Y^+ \,$ extends, homeomorphically by reflection, into $\,  H \, \colon \mathbb X \onto \mathbb Y \,$. The energy computation,
$$\,\mathscr E_{\mathbb X^+}[h] = \frac{1}{2} \mathscr E_{\mathbb X}[H] \geqslant \frac{1}{2} \mathscr E_{\mathbb X}[\frak h] = \mathscr E_{\mathbb X^+}[\frak h]\,$$
shows that the same mapping $\, \frak h \, \colon \mathbb X^+ \onto \mathbb Y^+ \,$ is the unique energy-minimal among strong limits of all quadrilateral homeomorphisms $\,  h \, \colon \mathbb X^+ \onto \mathbb Y^+ \,$. Thus, in\textit{ Case 2},  the same cracks develop along the rays $\;[r, \rho]\, e^{i\, \theta}\,, 0< \theta < \pi\,$. \\


\begin{remark}
 A priori these results do not rule out the existence of
univalent harmonic mappings from $\,\mathbb X\,$ onto $\,\mathbb
Y\,$, or quadrilateral harmonic mappings from $\,\mathbb X^+\,$ onto $\,\mathbb
Y^+\,$,  simply because harmonic homeomorphisms need not minimize the energy. Such mappings are admissible solutions to the Hopf-laplace equation, good enough to include them into a study of elastic deformations.   Nevertheless, nonexistence of harmonic homeomorphisms between annuli $\,\mathbb X\,$ and $\,\mathbb Y\,$ in \textit{Case 2} was conjectured by
J. C. C. Nitsche~\cite{Ni}, ~\cite{Nib}\,, prominent conjecture indeed. After several efforts in~\cite{HS, Ka, Ly, Lyz0, Lyz, We}, the conjecture  was finally confirmed  in~\cite{IKO2}.
  Similar cracking phenomena  are observed in higher dimensions \cite{IOnAnnuli}, see \cite{JM} for the $\,\mathscr L^p$ -setting.\\
\end{remark}
We end this section by showing, via a normal family  argument, that
 \begin{proposition}
 There is no quadrilateral harmonic homeomorphism  $\,  h \, \colon \mathbb X^+ \onto \mathbb Y^+ \,$ between upper halves of annuli $\mathbb X= \mathbb A(r,R)$ and $\mathbb Y = \mathbb A(r_\ast, R_\ast)$ if the ratio $\,\frac{R_*}{r_*} \,$  is too small relative to $\,\frac{R}{r} \,$.
 \end{proposition}
\begin{proof} Fix $\, R > r  = r_* = 1  \,$ while varying  $\, R_* = \rho_k >1 \,$,  $\,\rho_k \rightarrow  1\,$. Suppose that, contrary to our claim, there exist harmonic quadrilateral mappings $\,  h_k  = u_k + i\,v_k \, \colon \mathbb X^+ \onto \mathbb Y_k^+ \,$.  The key fact is that  the imaginary parts $ \, v_k = v_k(x, y)\, $ extend continuously up to the horizontal sides of the quadrilateral. Indeed,
$$\, \lim_{\varepsilon\rightarrow 0} v_k(x, \varepsilon) = \; \lim_{\varepsilon\rightarrow 0}\; \textnormal{dist}[ h_k(x, \varepsilon),\, \partial \mathbb Y_k^+ \cap \mathbb R ] \; =  0 $$

Thus the odd extension $\,v_k(x, \,-y) = - v_k(x, y)\,$ beyond the horizontal sides defines a harmonic function, again denoted by $\,v_k(x,y)\,$,   in the entire annulus $ \,\mathbb X\,$. The real part, however, may not admit harmonic extension. Note that each of the two sequences of harmonic functions $\,  h_k  \colon \mathbb X^+ \onto \mathbb Y_k^+ \,$ and $\,v_k \, \colon \mathbb X \onto [-\rho_k,\, \rho_k ] \,$,    being bounded, contains a subsequence converging uniformly on compact subsets of $\,\mathbb X^+\,$and $\,\mathbb X\,$, respectively. Let $\,  h  \colon \mathbb X^+ \to \mathbb C \,$  and $\,  v\,  \colon \mathbb X \to \mathbb R \,$ be their limits. Certainly, $\,v\,$ is not constant in $\,\mathbb X^+\,$, because it assumes  all values from the interval  $\,(0,\,1 )\,$. On the other hand $\, |h(z)| \equiv 1\,$ on $\,\mathbb X\,$. Hence
$$
0 = \frac{\partial ^{\,2}}{\partial z\; \partial{\bar{z}}}\; |h(z)| ^2 = \frac{\partial }{\partial z }\left( h_{\bar{z}} \,\overline{h}\; +\; h\, \overline{h_z}  \right) = \,|h_{\bar{z}}|^2 \; +\; |h_z|^2 \;\;,\;\;\textnormal{so}\;\;\; h \equiv\, const.
$$
clear contradiction.
\end{proof}
However, the question as to whether the same condition (\ref{Nitsche1})
  draws the line between the existence and nonexistence of quadrilateral harmonic mappings $\,  h \, \colon \mathbb X^+ \onto \mathbb Y^+ \,$ remains unclear.

\subsection{Open Question}
\begin{question}

It is not difficult to see that if the class of harmonic diffeomorphisms $\,h \colon \mathbb X\,\onto \,\mathbb Y\,$ is not empty,  then it contains the one with smallest energy. Does such harmonic  diffeomorphism represent the energy-minimal mappings among all homeomorphisms  $\,f \colon \mathbb X\,\onto \,\mathbb Y\,$?
The same question is pertinent to $\,p$ -harmonic homeomorphisms.
\end{question}

\bibliographystyle{amsplain}

\begin{thebibliography}{99}

\bibitem{A1}   L. Ahlfors \textit{Conditions for quasiconformal deformations in several variables,} In: Contributions to analysis, London-New York: Academic Press 1974.

\bibitem{A2} L. Ahlfors \textit{Quasiconformal Deformations and Mappings in $\mathbb R^n$,} Journal D'Analyse Math\'{e}matique, vol. 30, 745--97 (1976)


\bibitem{AIM}
K. Astala, T. Iwaniec, and G. Martin, \textit{Deformations of annuli with smallest mean distortion}, Arch. Ration. Mech. Anal. {\bf 195} (2010), no. 3, 899--921.

\bibitem{AIMb}
K. Astala, T. Iwaniec, and G. Martin, \textit{Elliptic partial differential equations and quasiconformal mappings in the plane},  Princeton University Press, Princeton, NJ, 2009.

\bibitem{AIMO}
K. Astala, T. Iwaniec, G. J. Martin, and J.  Onninen, \textit{Extremal mappings of finite distortion}, Proc. London Math. Soc. (3) {\bf 91} (2005), no. 3, 655--702.



\bibitem{Ba0}
J. M. Ball, \textit{Convexity conditions and existence theorems in nonlinear elasticity}, Arch. Rational Mech. Anal. {\bf 63} (1976/77), no. 4, 337--403.
\bibitem{Ba2}
J. M. Ball, \textit{Global inveribility of Sobolev functions and the interpenetration of matter}, Proc. Roy. Soc. Edinburgh Sect. A 88, no. 3-4. (1981). 315-328. MR616782 (83f:73017).
\bibitem{Ba}
J.  M. Ball,  \textit{Minimizers and the Euler-Lagrange equations}, Trends and applications of pure mathematics to mechanics (Palaiseau, 1983), 1--4, Lecture Notes in Phys., 195, Springer, Berlin, 1984.

\bibitem{Ba1}
J.  M. Ball,  \textit{Some open problems in elasticity},  Geometry, mechanics, and dynamics, 3--59, Springer, New York, 2002.

\bibitem{BCO}
J. M. Ball, J. C.  Currie,  and P. J.  Olver,  \textit{Null Lagrangians,
weak continuity, and variational problems of arbitrary order}, J.
Funct. Anal. {\bf 41} (1981), no. 2, 135--174.

\bibitem{BOP}
P. Bauman, N. C. Owen, and D. Phillips, \textit{Maximal smoothness of solutions to certain Euler-Lagrange equations from nonlinear elasticity}, Proc. Roy. Soc. Edinburgh Sect. A {\bf 119} (1991), no. 3-4, 241--263.

\bibitem{Br}
K.  B. Broberg,  \textit{Cracks and Fractures},  1999 Academic Press.

\bibitem{Cob}
R. Courant, \textit{Dirichlet's principle, conformal mapping, and minimal surfaces},
With an appendix by M. Schiffer.  Springer-Verlag, New York-Heidelberg, 1950.

\bibitem{CIKO}
J. Cristina, T. Iwaniec, L. V. Kovalev, and J. Onninen, \textit{Lipschitz regularity for the Hopf-Laplace equation}, arXiv:1011.5934.

\bibitem{Duren}
P. Duren, \textit{Harmonic mappings in the plane}, Cambridge Tracts in Mathematics, 156. Cambridge University Press, Cambridge, 2004.


\bibitem{Ed}
 D. G. B. Edelen, \textit{The null set of the Euler-Lagrange operator},
Arch. Rational Mech. Anal. {\bf 11} (1962) 117--121.



\bibitem{HS}
W. Hengartner and G. Schober, \textit{Univalent harmonic functions}, Trans. Amer. Math. Soc. {\bf 299} (1987), no. 1, 1--31.


\bibitem{Iw}
T. Iwaniec, \textit{On the concept of the weak Jacobian and Hessian},
Rep. Univ. Jyv\"askyl\"a Dep. Math. Stat., {\bf 83}, Univ. Jyv\"askyl\"a,
Jyv\"askyl\"a, (2001) 181--205.

\bibitem{IKO2}
T. Iwaniec, L. V. Kovalev and J. Onninen,
\textit{The Nitsche conjecture}, J. Amer. Math. Soc. \textbf{24} (2011), no.~2, 345--373.

\bibitem{IKKO}
T. Iwaniec, N.-T. Koh, L. V. Kovalev, and J. Onninen,
\textit{Existence of energy-minimal diffeomorphisms between doubly connected domains}, Invent. Math. (2011), to appear.

\bibitem{IKosO}
T. Iwaniec,  P. Koskela, and J. Onninen, \textit{Mappings of finite distortion: monotonicity and continuity}, Invent. Math. 144 (2001), no. 3, 507--531.

\bibitem{IKOhopf}
T. Iwaniec, L. V. Kovalev, and J. Onninen, \textit{Hopf differentials and smoothing Sobolev homeomorphisms}, Int. Math. Res. Not. IMRN, to appear. 
\bibitem{IKOApproximation}
T. Iwaniec, L. V. Kovalev,  and J. Onninen, \textit{
 Diffeomorphic approximation of Sobolev homeomorphisms}, Arch. Ration. Mech. Anal. 201 (2011), no. 3, 1047–1067.
\bibitem{IKOsa}
T. Iwaniec, L. V. Kovalev,  and J. Onninen, \textit{Approximation up to the boundary of homeomorphisms of finite Dirichlet energy},  Bull. London Math. Soc., to appear.
\bibitem{IKOli}
T. Iwaniec, L. V. Kovalev,  and J. Onninen, \textit{Lipschitz regularity for inner-variational equations}, arXiv:1109.0720.



\bibitem{IMb}
T. Iwaniec and G. Martin, \textit{Geometric function theory and non-linear analysis}, Oxford University Press, New York, 2001.
\bibitem{IOnAnnuli}
T. Iwaniec and J. Onninen, \textit{$\,n$-Harmonic Mappings between Annuli, The Art of Integrating Free Lagrangians}, Memoirs of Amer. Math. Soc. 2011, 1--105.



\bibitem{IOtr}
T. Iwaniec and J. Onninen, \textit{Deformations of finite conformal energy: Boundary behavior and limit theorems}, Trans. Amer. Math. Soc. {\bf 363} (2011), no. 11, 5605--5648.

\bibitem{IOa}
T. Iwaniec and J. Onninen, \textit{The weak  and  strong closures of Sobolev homeomorphisms are the same}, arXiv:1201.3864


\bibitem{IVV}
T. Iwaniec, G. Verchota  and A. Vogel, \textit{The Failure of Rank-One Connections}, Arch.Rational Mech. Anal. 163 (2002) 125--169.

\bibitem{JM}
M. Jordens and G. J. Martin, \textit{Deformations with smallest weighted $L^p$ average distortion and Nitsche type phenomena},
 J. Lond. Math. Soc., to appear.


\bibitem{Job}
J. Jost, \textit{Two-dimensional geometric variational problems}, John Wiley \& Sons, Ltd., Chichester, 1991.


\bibitem{Ka}
D. Kalaj,  \textit{On the Nitsche conjecture for harmonic mappings in $\R^2$ and $\R^3$.} Israel J. Math. {\bf 150} (2005), 241--251.


\bibitem{Ki}
B. Kirchheim, \textit{Rigidity and Geometry of microstructures}, Habilitation thesis, University of Leipzig, 2003.

\bibitem{KS}
B. Kirchheim and L. Sz\'ekelyhidi, \textit{On the gradient set of Lipschitz maps}, J. Reine Angew. Math. {\bf 625} (2008), 215--229.

\bibitem{K}
K. Kuratowski, \textit{On the completeness of the space of monotone mappings and some related problems}, Bull. Pol. Acad. Sc. 16 (1968), 283--285.

\bibitem{KL}
K. Kuratowski and C. Lacher, \textit{A theorem on the space of monotone mappings}, Bull. Acad. Polon. Sci. Ser. Sci. Math. Astronom. Phys. 17 (1969), 797--800.


\bibitem{Lyz0}
A. Lyzzaik, \textit{Univalent harmonic mappings and a conjecture of J. C. C. Nitsche},  XII-th Conference on Analytic Functions (Lublin, 1998). Ann. Univ. Mariae Curie-Sk\l odowska Sect. A 53 (1999), 147--150.

\bibitem{Ly}
A. Lyzzaik,  \textit{The modulus of the image annuli under univalent harmonic mappings and a conjecture of J.C.C. Nitsche},  J. London Math. Soc., {\bf 64} (2001), 369--384.

\bibitem{Lyz}
A. Lyzzaik,  \textit{Univalent harmonic mappings of annuli}, Publ. Inst. Math. (Beograd) (N.S.) 75(89) (2004), 173--183.










\bibitem{McAuley}

L.F. McAuley, \textit{Some Fundamental Theorems and Problems Related to Monotone Mappings}, The Proceedings of the First Conference on Monotone Mappings and Open Mappings, dedicated to the memory of Gordon Thomas Whyburn, October 8-11, 1970, Binghamton, New York, edited by Louis F. McAuley,


\bibitem{Mor}
 C.B. Morrey \textit{The Topology of Path Surfaces},  Amer. Journ. Math., 1935 [p.26]

\bibitem{Mo}
R. Moser,  \textit{On a variational problem with non-differentiable constraints}, Calc. Var. Partial Differential Equations {\bf 29} (2007), no.~1, 119--140.


\bibitem{Muller}
S. M\"{u}ller,  \textit{Higher integrability of determinants and weak convergence in $\,L^1\,$}, J. Reine Angew. Math., 412 (1990), 20--43.

\bibitem{Ni}
J. C. C. Nitsche,   \textit{On the modulus of doubly connected regions under harmonic mappings},  Amer. Math. Monthly,  {\bf 69} (1962), 781--782.

\bibitem{Nib}
J. C. C. Nitsche,  \textit{Vorlesungen \"uber Minimalfl\"achen}, Springer-Verlag, Berlin-New York, 1975.

\bibitem{Rado}
T. Rad\'{o}, \textit{Length and Area}, American Mathematical Society, New York, 1948.







\bibitem{R}
 H.M. Reimann,  \textit{Ordinary Differential Equations and Quasiconformal Mappings,} Inventiones Math. 33, 247--270 (1976)




\bibitem{SSe}
E. Sandier and S.  Serfaty,  \textit{Limiting vorticities for the Ginzburg-Landau equations}, Duke Math. J. {\bf 117} (2003), no.~3, 403--446.





\bibitem{Ta}
A. Taheri,  \textit{Quasiconvexity and uniqueness of stationary points in the multi-dimensional calculus of variations}, Proc. Amer. Math. Soc. \textbf{131} (2003), no.~10, 3101--3107.


\bibitem{We}  A. Weitsman,  \textit{Univalent harmonic mappings of annuli and a conjecture of J.C.C. Nitsche}, Israel J. Math., {\bf 124}  (2001),  327--331.

\bibitem{Why}
G. T. Whyburn, \textit{Analytic topology}, American Mathematical Society, Providence, R.I. (1963).

\bibitem{Ya}
X. Yan, \textit{Maximal smoothness for solutions to equilibrium equations in 2D nonlinear elasticity}, Proc. Amer. Math. Soc. {\bf 135} (2007), no. 6, 1717--1724.

\end{thebibliography}

\end{document}